\documentclass[10pt]{article}%
\usepackage{amssymb}
\usepackage{amsmath}
\usepackage{amsfonts}
\usepackage{graphicx}
\usepackage{color}
\usepackage{caption}
\usepackage{multirow}
\usepackage{diagbox}
\usepackage{subcaption}
\usepackage{mathrsfs}
\usepackage{amscd}
\usepackage{url}
\usepackage{enumitem}
\usepackage{booktabs}
\usepackage{bm}
\newtheorem{theorem}{Theorem}[section]

\newtheorem{definition}{Definition}[section]

\newtheorem{lemma}{Lemma}[section]

\newtheorem{remark}{Remark}[section]

\newenvironment{proof}[1][Proof]{\noindent\textbf{#1.} }{\ \hfill\rule{0.3em}{0.5em}}

\newcommand*{\R}{\mathbb{R}}

\def\beq#1{\begin{equation}\label{#1}}
\def\eeq{\end{equation}}
\def\bep{\begin{proof}}

\def\ep{\end{proof}}
\def\reff#1{(\ref{#1})}

\def\M{\mathcal M}

\def\A{{\mathcal A}}
\def\B{\mathcal B}
\def\C{\mathcal C}

\def\ignore#1{}
\def\X{\mathcal X}
\def\Y{\mathcal Y}
\def\G{\mathcal G}
\def\H{\mathcal H}
\def\F{\mathcal F}

\def\U{\mathcal U}

\def\bea#1{\begin{array}{#1}}
\def\ea{\end{array}}

\newcommand{\re}{\mathbb{R}}

\input epsf

\title{Multi-Tubal Rank of Third Order Tensor and Related Low Rank Tensor Completion Problem}
\author{Quan Yu\thanks{School of Mathematics, Tianjin University, Tianjin 300354, P.R. China. E-mail: \emph{QuanYu527@163.com}.},\quad Xinzhen Zhang\thanks{School of Mathematics, Tianjin University, Tianjin 300354, P.R. China. E-mail: \emph{xzzhang@tju.edu.cn}. Her work is supported by NSFC
	(11871369).
},\quad and Zheng-Hai Huang \thanks{School of Mathematics, Tianjin University, Tianjin 300354, P.R. China.
E-mail: \emph{huangzhenghai@tju.edu.cn}. His work is supported by NSFC (11871051).}}

\begin{document}

\maketitle

\textbf {Abstract.}   % Abstract of not more than 200 words.
Recently, a tensor factorization based method for a low  tubal rank tensor completion problem of a third order tensor was proposed, which performed better than some existing methods. Tubal rank is only defined on one mode of third order tensor without low rank structure in the other two modes. That is, low rank structures on the other two modes are missing. Motivated by this, we first introduce multi-tubal rank,
 and then establish a relationship between multi-tubal rank and Tucker rank. Based on the multi-tubal rank, we propose a novel low rank tensor completion model. For this model, a tensor factorization based method is applied and the corresponding convergence analysis is established. In addition, spatio-temporal characteristics are intrinsic features in video and internet traffic tensor data. To get better performance, we make full use of such features and improve the established tensor completion model. Then we apply tensor factorization based method for the improved model. Finally, numerical results are reported on the completion of image, video and internet traffic data to show the efficiency of our proposed methods. From the reported numerical results, we can assert that our methods outperform the existing methods.



\section{Introduction}

A tensor is a multidimensional array, and an Nth-order tensor is an element of the tensor product space of N vectors,  which have their own dimensions \cite{KB09}. Tensors, as higher order generalizations of vectors and matrices, have wide applications in various fields \cite{BZNC16,CCQY20,GRY11,K16,DM98, LV04,LMWY13, LFCLLY16, MB05,VT02}. Tensor decompositions, various generalizations of matrix singular value decomposition,  have attracted more and more attentions, including CANDECOMP/PARAFAC
(CP) decomposition \cite{CC70,K00},  Tucker decomposition \cite{T66} and tensor singular value decomposition (SVD) \cite{HKBH13,HTZXY17,KBHH13,KM11,MSL13}. Corresponding to such tensor decompositions, tensor ranks are called the CP rank, Tucker rank and tubal rank, respectively.

Third order tensors
  are widely used  in chemometrics \cite{MB94, SBG04}, psychometrics \cite{K83} and  image inpainting \cite{BPTD17,LZZ19,ZLCZ19,ZLLZ18}.
 Unless otherwise specialized, tensors in this paper are of third order.
For a third order $(n_1,n_2, n_3)$-dimensional tensor  $\X\in \mathbb R^{n_1\times n_2\times n_3}$, the CP decomposition is to decompose $\X$ as a sum of some outer products of three vectors:
$$ \X = \sum\limits_{i = 1}^r {a_1^{(i)} \circ a_2^{(i)}  \circ a_3^{(i)}}, $$
where the symbol ``$\circ$'' denotes the outer product and $a_j^{(i)}\in \mathbb R^{n_j}$ is a vector ($i\in \{1,2,\ldots,r\}$ and $j\in \{1,2,3\}$). The smallest $r$ in CP decomposition is called  CP rank of $\X$. From \cite{HL13}, it is NP-hard to determine the CP rank.
Compared with CP rank,
Tucker rank is easy to compute, and hence most of low rank tensor completion and recovery models are based on Tucker rank. Precisely, Tucker rank is a vector of the matrix ranks
$$\operatorname{rank}_{TC}(\mathcal{C}) = \left( \operatorname{rank}(C_{(1)}),\operatorname{rank}(C_{(2)}), \operatorname{rank}(C_{(3)}) \right),$$
where $ C_{(1)}\in \re^{n_1\times (n_2n_3)}$  ($C_{(2)}\in \re^{n_2\times (n_1n_3)}$ and $C_{(3)}\in \re^{n_3\times (n_1n_2)}$)
is mode-$1$ (mode-$2$ and mode-$3$, respectively) matricization of tensor.
More recently, Kilmer \emph{et al.} \cite{KBHH13} introduced tensor-tensor product (t-product) and tensor singular value decomposition (t-SVD). Based on these definitions, tubal rank was introduced and studied in \cite{KBHH13,KM11,MSL13}.

The low rank tensor completion problem is  to find a low rank tensor from observed incomplete data, which arises from various fields including internet traffic recovery \cite{ACRUW06,AMDJ16,TWSJR16,ZZXC15}, image and video inpainting \cite{JLLS18, LL17, LMWY13, ZLLZ18}.
Low rank tensor completion is modeled as
\begin{equation}\label{tensor-completion}
\mathop {\min }\limits_{\mathcal C}  \operatorname{rank}(\mathcal C),\quad \mbox{\rm s.t.}\quad  {P_\Omega }({\mathcal C}) = {P_\Omega }(\mathcal M),
\end{equation}
where $\operatorname{rank}(\cdot)$ is a tensor rank and $\Omega$ is an index set locating the observed data.  $P_\Omega $ is a linear operator that extracts the entries in $ \Omega $ and fills the entries not in $ \Omega $ with zeros, and $\mathcal{M}$ is a given tensor.

Different  tensor ranks lead to different low rank tensor completion models of (\ref{tensor-completion}) with different methods. The following low Tucker rank  tensor completion is considered
\begin{equation*}
	\mathop {\min }\limits_{\mathcal{C}}\, \left( \operatorname{rank}(C_{(1)}),\, \operatorname{rank}(C_{(2)}),\, \operatorname{rank}(C_{(3)})\right) ,\quad \mbox{\rm s.t.}\quad {P_\Omega }({\mathcal C}) = {P_\Omega }(\mathcal M).
\end{equation*}
To keep things simple, the weighted Tucker rank minimization problems is formulated as
\begin{equation}\label{Tucker-rank}
\mathop {\min }\limits_{\mathcal{C}} \sum\limits_{i = 1}^3 \operatorname{rank}(C_{(i)}),\quad \mbox{\rm s.t.}\quad {P_\Omega }({\mathcal C}) = {P_\Omega }(\mathcal M).
\end{equation}
Note that problem \eqref{Tucker-rank} is  non-convex  since matrix rank function is nonconvex. To solve \eqref{Tucker-rank}, the convex optimization problem is considered as
\begin{equation}\label{nuclear-norm}
\mathop {\min }\limits_{\mathcal{C}} \sum\limits_{i = 1}^3 {{{\left\| {{C_{(i)}}} \right\|}_*}},\quad \mbox{\rm s.t.}\quad {P_\Omega }({\mathcal C}) = {P_\Omega }(\mathcal M).
\end{equation}
In general,  SVD is needed in each iteration of numerical methods for \eqref{nuclear-norm}, which leads to high computational cost. To lower the computational cost, a matrix factorization method was considered by Xu \emph{et al}. \cite{XHYS13}, which preserves the low rank structure of  matrix. Precisely, \eqref{Tucker-rank} is modeled as
\begin{equation}\label{matrix-decomposition}
\mathop {\min }\limits_{{X^i},{Y^i},{\mathcal C}} \sum\limits_{i = 1}^3 {{\alpha _i}{{\left\| {X^i}{Y^i} - C_{(i)} \right\|}_F^2}},\quad \mbox{\rm s.t.}\quad {P_\Omega }({\mathcal C}) = {P_\Omega }(\mathcal M).
\end{equation}
This method has been widely used in various areas \cite{LS13}.  As pointed in \cite{KBHH13,KM11,MSL13},  unfolding a tensor directly will destroy the original multi-way structure of the data, which leads to vital information loss and degraded performance. Note that
the sizes of $C_{(i)},\, i=1,2,3$ in \eqref{matrix-decomposition} are  the same as $\C$ in principle, which  makes it difficult to lower the computational efforts.

Based on tubal rank, the following model was considered in \cite{ZLLZ18} based on tensor factorization,
\begin{equation}\label{model-tubal}
\mathop {\min }\limits_{\mathcal{X},\mathcal{Y},\mathcal{C}} \frac{1}{2}\left\| {\mathcal{X}*\mathcal{Y} - \mathcal{C}} \right\|_F^2,\quad \mbox{\rm s.t.} \quad  {P_\Omega }(\mathcal{C} - \mathcal{M}) = 0,
\end{equation}
where  ``$*$'' denotes the t-product. By analysis in \cite{KBHH13,KM11,MSL13,ZLLZ18}, the t-product
can be computed by some block diagonal matrices of smaller sizes, which makes  a significant reduction of computational cost.
Later, a corrected tensor nuclear norm minimization method was proposed in \cite{ZN19} for noisy observations.

It is valuable to mention that only one mode is considered in tubal rank and the other two modes are ignored. That is, low rank structure on the other two modes is missing. Motivated by this, we introduce a vector of tubal ranks on three different modes, called multi-tubal rank, which is similar to Tucker rank. Then a relationship between multi-tubal rank and Tucker rank is established.
Based on the new introduced multi-tubal rank,  a new tensor completion model is proposed. Similar to TCTF in \cite{ZLLZ18}, a tensor factorization based method is applied to solve the  proposed model.
 In video and internet traffic tensor completion, spatio-temporal characteristics are intrinsic features.
To make full use of such features, we improve the proposed low multi-tubal rank tensor completion model, and then apply tensor factorization based method for the improved model.
To the best of authors' knowledge, this paper is the first one to introduce multi-tubal rank, to present the relationship between tubal rank and Tucker rank and to introduce the spatio-temporal characteristics to recover video data. The reported numerical examples show that our results have less relative error and higher peak signal-to-noise ratio (PSNR) within less computational time than those of some existing methods. That is, our models and methods outperform the existing methods.

The paper is organized as follows. Section 2 introduces the multi-tubal rank of a third order tensor with motivation in both theory and application. In Section 3, a new model of tensor completion based on the multi-tubal rank is proposed and tensor factorization based method is applied  with its corresponding convergence analysis. In Section 4, the tensor completion model is modified to tensor data with some characteristics when the involved data have spatio-temporal characteristics. For this improvement, tensor factorization based method is also modified.
Finally, some numerical results on  colorful image recovery, gray video recovery and internet traffic data recovery  are reported, which show the efficiency of the proposed methods.

\section{Multi-tubal rank: definition and motivation}
Before proceeding, we  present some notations here. For a positive integer $n$,
$ {\bf [n]}:=\{1,2,\ldots, n\}$. Scalars, vectors and matrices are denoted as lowercase letters ($a,b,c,\ldots$), boldface lowercase letters ($\bm{a} ,\bm{b},\bm{c},\ldots$) and uppercase letters ($A,B,C,\ldots $), respectively.
Third order tensors are denoted as $\mathcal{A},\mathcal{B},\mathcal{C},\ldots$, and the set of all the third order real tensors is denoted as $\R^{n_1\times n_2\times n_3}$. For a third order tensor $\mathcal{A}$, we use the Matlab notations $ \mathcal{A}(i,:, :),\, \mathcal{A}(:, j, :) $ and $  \mathcal{A}(:, :, k)  $ to denote its $ i $-th horizontal, $ j $-th lateral and $ k $-th frontal slice, respectively. Let $\A=(\A_{ijk})\in \R^{n_1\times n_2\times n_3}$, then $(A_1^{(i)})_{jk} = (A_2^{(j)})_{ik}=(A_3^{(k)})_{ij}=\A_{ijk}$ for all $i\in {\bf [n_1]}$, $j\in {\bf [n_2]}$ and $k\in {\bf [n_3]}$. The inner product of two tensors  $ \mathcal{A},\,\mathcal{B} \in {\R^{{n_1} \times {n_2} \times {n_3}}}$ is the sum of products of their entries, i.e.
$$\left\langle {\mathcal{A},\mathcal{B}} \right\rangle  = \sum\limits_{i = 1}^{{n_1}} {\sum\limits_{j = 1}^{{n_2}} {\sum\limits_{k = 1}^{{n_3}} {{\mathcal{A}_{ijk}}{\mathcal{B}_{ijk}}} } }. $$
The Frobenius norm is ${\left\| \mathcal{A} \right\|_F} = \sqrt {\left\langle {\mathcal{A},\mathcal{A}} \right\rangle } $.
%$ {\left\| \mathcal{A} \right\|_{\infty}} $  represents the maximum absolute value of $\mathcal{A}  $.
For a martix $A$, $ A^* $ and $ A^{-1} $ represent the conjugate transpose and the inverse of $ A $, respectively.
$I_n$ represents the identity matrix of size $n \times n$.
For any $u\in {\bf {\bf [3]}}$, the $u$-mode matrix product of a tensor $\A=(\A_{ijk})\in \R^{n_1\times n_2\times n_3}$ with a matrix $M_u\in \R^{J\times n_u}$ is denoted by $\A\times_u M_u$  with its entries
$$
\begin{array}{rl}
&(\A\times_1 M_1)_{ii_2i_3}=\sum_{i_1=1}^{n_1}\A_{i_1i_2i_3}(M_1)_{ii_1},\\ &(\A\times_2 M_2)_{i_1ji_3}=\sum_{i_2=1}^{n_2}\A_{i_1i_2i_3}(M_2)_{ji_2},\\
&(\A\times_3 M_3)_{i_1i_2k}=\sum_{i_3=1}^{n_3}\A_{i_1i_2i_3}(M_3)_{ki_3}.
\end{array}
$$

\subsection{Generalized $T_u$-product and multi-tubal rank}

In this subsection, we will introduce multi-tubal rank, which is a generalization of tubal rank in \cite{KBHH13}.
Before proceeding,  we review the Discrete Fourier Transformation
(DFT), which plays a key role in tensor-tensor product (t-product). For $\A\in \mathbb{R}^{n_1 \times n_2 \times n_3}$ and $u\in{\bf {\bf [3]}}$,
let ${{\bar \A}_u} \in {{\mathbb C}^{{n_1} \times {n_2} \times {n_3}}}$ be the result of
Discrete Fourier transformation (DFT) of ${{ \A}} \in {{\mathbb R}^{{n_1} \times {n_2} \times {n_3}}}$
along the $ u $-th mode. Specifically, let  $F_{n_u}=[f_1,\dots, f_{n_u}]\in \mathbb C^{n_u\times n_u}$, where
$$f_i=\left[ \omega^{0\times (i-1)}; \omega^{1\times (i-1)};\dots; \omega^{(n_u-1)\times (i-1)}\right] \in \mathbb C^{n_u}$$
with $\omega=e^{-\frac{2\pi i}{n_u}}$ and $i=\sqrt{-1}$. Then
\begin{equation*}
\bar \A_1(:,j,k)=F_{n_1}\A(:,j,k),\quad
\bar \A_2(i,:,k)=F_{n_2}\A(i,:,k),\quad
\bar \A_3(i,j,:)=F_{n_3}\A(i,j,:),
\end{equation*}
which can be computed by Matlab command ``$\bar \A_u=fft(\A,[\; ],u)$''. Furthermore, $\A$ can be computed by $\bar \A_u$ with the inverse DFT $ \A=ifft({\bar \A}_u,[\; ],u) $.

 For $\A\in \mathbb{R}^{n_1\times n_2\times n_3}$, we define matrices ${\bar A}_1 \in {{\mathbb C}^{{n_1}{n_2} \times {n_1}{n_3}}}$, $\bar A_2 \in {{\mathbb C}^{{n_1}{n_2} \times {n_2}{n_3}}}$ and ${{\bar A}_3} \in {{\mathbb C}^{{n_1}{n_3} \times {n_2}{n_3}}}$ as
\begin{equation}\label{bdiag}
{{\bar A}_u} = bdia{g_u}(\bar {{\mathcal{A}_u}} ) \hfill \\
= \left[ {\begin{array}{*{20}{c}}
	{\bar A_u^{(1)}}&{}&{}&{} \\
	{}&{\bar A_u^{(2)}}&{}&{} \\
	{}&{}& \ddots &{} \\
	{}&{}&{}&{\bar A_u^{({n_u})}}
	\end{array}} \right],~ \forall u\in  {\bf [3]}. \end{equation}
Here, $ bdiag_{u}(\cdot) $ is an operator which maps the tensor $ {{\bar {\mathcal A_u}}} $ to the block diagonal matrix $ \bar A_u$. The block circulant matrices $bcirc_{1}({\A}) \in {{\mathbb R}^{{n_1}{n_2} \times {n_1}{n_3}}}$, $bcirc_{2}({\A}) \in {{\mathbb R}^{{n_1}{n_2} \times {n_2}{n_3}}}$ and  $bcirc_{3}({\A}) \in {{\mathbb R}^{{n_1}{n_3} \times {n_2}{n_3}}}$ of $\A$ are defined  as
$$bcirc_{u}({\A}) = \left[ {\begin{array}{*{20}{c}}
	{A_u^{(1)}}&{A_u^{(n_u)}}& \cdots &{A_u^{(2)}}\\
	{A_u^{(2)}}&{A_u^{(1)}}& \cdots &{A_u^{(3)}}\\
	\vdots & \vdots & \ddots & \vdots \\
	{A_u^{(n_u)}}&{A_u^{({n_u} - 1)}}& \cdots &{A_u^{(1)}}
	\end{array}} \right],~ \forall u \in {\bf {\bf [3]}}.$$

Based on these notations, the generalized $T_u$-product and multi-tubal rank are introduced as follows.

\begin{definition}\label{def:T-pro}\textbf{(Generalized $T_u$-product)}
For $\A_1\in \mathbb{R}^{n_1\times n_2\times r_1}$ and $\B_1\in \mathbb{R}^{n_1\times r_1\times n_3}$, define
$$\A_1\ast_1\B_1:=fold_1(bcirc_1(\A_1)\ \cdot unfold_1(\B_1)) \in \mathbb{R}^{n_1\times n_2\times n_3}.$$
For $\A_2\in \mathbb{R}^{n_1\times n_2\times r_2}$ and $\B_2\in \mathbb R^{r_2\times n_2\times n_3}$, define
$$\A_2\ast_2\B_2:=fold_2(bcirc_2(\A_2)\ \cdot unfold_2(\B_2)) \in \mathbb{R}^{n_1\times n_2\times n_3}.$$
For $\A_3\in \mathbb{R}^{n_1\times r_3\times n_3}$ and $\B_3\in \mathbb R^{r_3\times n_2\times n_3}$, define
$$\A_3\ast_3\B_3:=fold_3(bcirc_3(\A_3)\ \cdot unfold_3(\B_3)) \in \mathbb{R}^{n_1\times n_2\times n_3}.$$
Here
$$ unfold_{u}(\B_u) = [B_u^{(1)};B_u^{(2)}; \cdots ;B_u^{(n_u)}], $$
	and its inverse operator ``fold$ _{u} $" is defined by $fold_{u}(unfold_{u}(\B_u)) = \B_u$.
\end{definition}

\begin{definition}\label{definition2.6}{\bfseries (Multi-tubal rank)} For any tensor $\A \in {{\mathbb R}^{{n_1} \times {n_2} \times {n_3}}}$ and $u\in {\bf [3]}$, let $r_u^l=rank(\bar A_u^{(l)})$ and $l\in {\bf [n_u]}$.
	Then multi-tubal rank of $\A$ is defined as $$rank_{mt}(\A)=(r_1(\A),r_2(\A),r_3(\A)),$$ where $r_u(\A)=\max\{r_u^1,r_u^2,\dots, r_u^{n_u}\}$  for $u \in  {\bf [3]}$.
\end{definition}

In fact, the $T_3$-product is the classical $t$-product and $r_3(\A)$ is
tubal rank \cite{KBHH13} of tensor $\A$, respectively.

\begin{lemma}\label{lem:equ} \cite{KM11}
		Suppose that $\A,\, \B$ are tensors such that $\F:=\A\ast_u \B$ ($u\in [\bf 3]$) is well defined as in Definition \ref{def:T-pro}. Let $ \bar A_u,\bar B_u, \bar F_u $ be defined as in \eqref{bdiag} and $ r_u(\cdot) $ be defined as in Definition \ref{definition2.6}. Then
	\begin{enumerate}
		\item [(1).]$\left\| \A \right\|_F^2 = \frac{1}{n_u}\left\| {{{\bar A}_u}} \right\|_F^2$;
		\item [(2).]${\mathcal F} = \A{ \ast _u}\B$ and $\bar F_u= {\bar A_u}{\bar B_u}$ are equivalent;	
		\item [(3).]$r_u(\F)\leq \min\{r_u(\A), r_u(\B)\}$.
	\end{enumerate}		
\end{lemma}
From Lemma \ref{lem:equ}, we can assert that the generalized tensor factorization can be computed by matrix factorization, which is computable.

\subsection{Motivation of multi-tubal rank}
We first discuss the relationship between Tucker rank and multi-tubal rank.  To this end, we need the following lemma.
\begin{lemma}\label{lem:martix-to-tensor}
	Suppose that $\mathcal{C} \in {\mathbb{R}^{{n_1} \times {n_2} \times {n_3}}}$, $ F \in {\mathbb{R}^{{n_1} \times {n_1}}}, \,G \in {\mathbb{R}^{{n_2} \times {n_2}}}$ and $H \in {\mathbb{R}^{   {{n_3} }  \times {n_3}}} $. Let $\F\in \mathbb{R}^{n_1\times n_2\times n_1},\, \tilde\F\in \mathbb{R}^{n_1\times n_1\times n_3},\, \G\in \mathbb{R}^{n_2\times n_2\times n_3}, \,\tilde\G\in \mathbb{R}^{n_1\times n_2\times n_2}, \, \H\in \mathbb{R}^{n_1\times n_3\times n_3},\, \tilde\H\in \mathbb{R}^{n_3\times n_2\times n_3}$ be the tensors with their slices
	$$\begin{gathered}
		F_2^{(1)} = F,\: F_2^{(2)} =  \cdots  = F_2^{({n_2})} = 0,
		\quad \tilde F_3^{(1)} = F,\: \tilde F_3^{(2)} =  \cdots  = \tilde F_3^{({n_3})} = 0, \hfill \\
		G_3^{(1)} = G^T,\: G_3^{(2)} =  \cdots  = G_3^{({n_3})} = 0, \quad \tilde G_1^{(1)} = G,\: \tilde G_1^{(2)} =  \cdots  = \tilde G_1^{({n_1})} = 0, \hfill \\
		H_1^{(1)} = H^T,\: H_1^{(2)} =  \cdots  = H_1^{({n_1})} = 0, \quad \tilde H_2^{(1)} = H^T,\: \tilde H_2^{(2)} =  \cdots  = \tilde H_2^{({n_2})} = 0. \hfill \\
	\end{gathered}$$	
	Then
	\begin{equation*}
		\begin{aligned}
\left\{ \begin{array}{l}
	{\F}{*_2}{\C} = {\C}{ \times _1}F,\\
	{\tilde\F}{*_{3}}{\C} = {\C}{ \times _1}{ F},
\end{array} \right.\quad \left\{ \begin{array}{l}
	{\C}{*_3}{\G} = {\C}{ \times _2}G,\\
	{\tilde\G}{*_1}{\C} = {\C}{ \times _2}G,
\end{array} \right.\quad \left\{ \begin{array}{l}
	{\C}{*_1}{\H} = {\C}{ \times _3}H,\\
	{\C}{*_2}{\tilde\H} = {\C}{ \times _3}{H}.
\end{array} \right.
		\end{aligned}
	\end{equation*}
\end{lemma}
\begin{proof}
	It clear to see that
	\begin{equation*}
		\begin{aligned}
			&unfol{d_2}\left( \mathcal{F}{ * _2}\mathcal{C} \right) = bcir{c_2}\left( \F \right) \cdot unfol{d_2}\left( {\C} \right)\\
			=& \left[ {\begin{array}{*{20}{c}}
					{F_2^{(1)}}&{F_2^{({n_2})}}& \cdots &{F_2^{(2)}}\\
					{F_2^{(2)}}&{F_2^{(1)}}& \cdots &{F_2^{(3)}}\\
					\vdots & \vdots & \ddots & \vdots \\
					{F_2^{({n_2})}}&{F_2^{({n_2} - 1)}}& \cdots &{F_2^{(1)}}
			\end{array}} \right]\left[ {\begin{array}{*{20}{c}}
					{C_2^{(1)}}\\
					{C_2^{(2)}}\\
					\vdots \\
					{C_2^{({n_2})}}
			\end{array}} \right]\\
			=& \left[ {\begin{array}{*{20}{c}}
					{F}&0& \cdots &0\\
					0&F& \cdots &0\\
					\vdots & \vdots & \ddots & \vdots \\
					0&0& \cdots &F
			\end{array}} \right]\left[ {\begin{array}{*{20}{c}}
					{C_2^{(1)}}\\
					{C_2^{(2)}}\\
					\vdots \\
					{C_2^{({n_2})}}
			\end{array}} \right] = \left[ {\begin{array}{*{20}{c}}
					F{C_2^{(1)}}\\
					F{C_2^{(2)}}\\
					\vdots \\
					F{C_2^{({n_2})}}
			\end{array}} \right].	
		\end{aligned}
	\end{equation*}
	Then
	\begin{equation*}
		\begin{aligned}
			\left( \mathcal{F}{ * _2}\mathcal{C} \right)_{ijk} = {\left( F{C_2^{(j)}} \right)_{ik}} = \sum\limits_{p = 1}^{{n_1}} {{F_{ip}}{{\left( {C_2^{(j)}} \right)}_{pk}}}= \sum\limits_{p = 1}^{{n_1}} {{{\C}_{pjk}}{F_{ip}} = {{\left( {{\C}{ \times _1}F} \right)}_{ijk}}}. 	
		\end{aligned}
	\end{equation*}	
Similarly,
\begin{equation*}
	\begin{aligned}
		\left( \mathcal{\tilde F}{ * _3}\mathcal{C} \right)_{ijk} = {\left( F{C_3^{(k)}} \right)_{ij}} = \sum\limits_{p = 1}^{{n_1}} {{F_{ip}}{{\left( {C_3^{(k)}} \right)}_{pj}}}= \sum\limits_{p = 1}^{{n_1}} {{{\C}_{pjk}}{F_{ip}} = {{\left( {{\C}{ \times _1}F} \right)}_{ijk}}}. 	
	\end{aligned}
\end{equation*}	
Now we can assert that  $ {\mathcal{F}{ * _2}\mathcal{C}}  =  {\mathcal{C}{ \times _1}F} $ and $ {\tilde\F}{*_{3}}{\C} = {\C}{ \times _1}{ F} $.			
	
	Furthermore,
	\begin{equation*}
		\begin{aligned}
			&unfol{d_3}\left( {{\C}{ * _3}{\G}} \right) = bcir{c_3}\left( {\C} \right) \cdot unfol{d_3}\left( {\G} \right)\\
			=& \left[ {\begin{array}{*{20}{c}}
					{C_3^{(1)}}&{C_3^{({n_3})}}& \cdots &{C_3^{(2)}}\\
					{C_3^{(2)}}&{C_3^{(1)}}& \cdots &{C_3^{(3)}}\\
					\vdots & \vdots & \ddots & \vdots \\
					{C_3^{({n_3})}}&{C_3^{({n_3} - 1)}}& \cdots &{C_3^{(1)}}
			\end{array}} \right]\left[ {\begin{array}{*{20}{c}}
					{G_3^{(1)}}\\
					{G_3^{(2)}}\\
					\vdots \\
					{G_3^{({n_3})}}
			\end{array}} \right]\\
			=& \left[ {\begin{array}{*{20}{c}}
					{C_3^{(1)}}&{C_3^{({n_3})}}& \cdots &{C_3^{(2)}}\\
					{C_3^{(2)}}&{C_3^{(1)}}& \cdots &{C_3^{(3)}}\\
					\vdots & \vdots & \ddots & \vdots \\
					{C_3^{({n_3})}}&{C_3^{({n_3} - 1)}}& \cdots &{C_3^{(1)}}
			\end{array}} \right]\left[ {\begin{array}{*{20}{c}}
					{{G^T}}\\
					0\\
					\vdots \\
					0
			\end{array}} \right] = \left[ {\begin{array}{*{20}{c}}
					{C_3^{(1)}{G^T}}\\
					{C_3^{(2)}{G^T}}\\
					\vdots \\
					{C_3^{({n_3})}{G^T}}
			\end{array}} \right].	
		\end{aligned}
	\end{equation*}
	Then
	\begin{equation*}
		\begin{aligned}
			\left( {{\C}{ * _3}{\G}} \right)_{ijk} = {\left( {C_3^{(k)}{G^T}} \right)_{ij}} = \sum\limits_{p = 1}^{{n_2}} {{{\left( {C_3^{(k)}} \right)}_{ip}}{{\left( {{G^T}} \right)}_{pj}}} = \sum\limits_{p = 1}^{{n_2}} {{{\C}_{ipk}}{G_{jp}} = {{\left( {{\C}{ \times _2}G} \right)}_{ijk}}}. 	
		\end{aligned}
	\end{equation*}
Similarly,
	\begin{equation*}
	\begin{aligned}
		\left( {\tilde\G}{ * _1}{{\C}} \right)_{ijk} = {\left( {G}{C_1^{(i)}} \right)_{jk}} = \sum\limits_{p = 1}^{{n_2}} {{{\left( {{G}} \right)}_{jp}}{{\left( {C_1^{(i)}} \right)}_{pk}}} = \sum\limits_{p = 1}^{{n_2}} {{{\C}_{ipk}}{G_{jp}} = {{\left( {{\C}{ \times _2}G} \right)}_{ijk}}}. 	
	\end{aligned}
\end{equation*}
Then $\mathcal{C}{*_3}\mathcal{G} = \mathcal{C}{ \times _2}G $ and $ {\tilde\G}{*_1}{\C} = {\C}{ \times _2}G $. Similarly,  $ {\mathcal{C}{ * _1}\mathcal{H}} =  {\mathcal{C}{ \times _3}H} $ and $ 	{\C}{*_2}{\tilde\H} = {\C}{ \times _3}{H} $. Hence the desired results are arrived.	
\end{proof}
\begin{theorem}
	For any tensor $\A=(\A_{ijk}) \in {{\mathbb R}^{{n_1} \times {n_2} \times {n_3}}}$, the following properties hold:
	\[\begin{array}{rl}&{r_1}(\A) \le \min\left\lbrace r\left( {{A_{(2)}}} \right),r\left( {{A_{(3)}}} \right)\right\rbrace ,~ {r_2}(\A) \le \min\left\lbrace r\left( {{A_{(1)}}} \right),r\left( {{A_{(3)}}} \right)\right\rbrace,\\~ &{r_3}(\A)\le \min\left\lbrace  r\left( {{A_{(1)}}} \right),r\left( {{A_{(2)}}} \right)\right\rbrace .\end{array}\]
\end{theorem}
\begin{proof}
Let $ \A=\B\times_1U^{(1)}\times_2U^{(2)}\times_3U^{(3)} $ be a Tucker rank decomposition of $ \A $, then $ r\left( {{A_{(3)}}} \right)=r\left( {{U^{(3)}}} \right) $ and $ \A=\tilde\B\times_3U^{(3)} $, where $\tilde\B=\B\times_1U^{(1)}\times_2U^{(2)}$. By Lemma \ref{lem:martix-to-tensor}, we have $ \A=\tilde\B\times_3U^{(3)}=\tilde\B*_1\U $, where $ \U \in \mathbb R^{n_1\times n_3\times n_3} $ with its slices	
\[ U_1^{(1)} = \left( U^{(3)}\right)^T,\: U_1^{(2)} =  \cdots  = U_1^{({n_1})} = 0. \]
Denote $ \bar\U_1=fft(\U,[~],1) $, then $ \bar\U_1(:,j,k)=F_{n_1}\U(:,j,k) $, then
\[ \bar\U_1(i,:,:)=\sum\limits_{l=1}^{n_1} F_{n_1}\left(i,l \right) \U(l,:,:)=F_{n_1}\left(i,1 \right) \left( U^{(3)}\right)^T,\quad \forall i \in [\bf n_1]. \]
From the definition of $ F_{n_1} $, $ F_{n_1}\left(1,i \right)\ne 0 $.
Thus $ r_1\left( \U\right)=r\left( {{U^{(3)}}} \right)$. From Lemma \ref{lem:equ} and $ \A=\tilde\B*_1\U $, we have
\[ r_1\left(\A \right)\leq r_1\left( \U\right)=r\left( {{U^{(3)}}} \right)=r\left({{A_{(3)}}}\right).  \]
Similarly, $ r_1\left(\A \right)\leq r\left({{A_{(2)}}}\right)$.
Now we can assert that $ {r_1}(\A) \le \min\left\lbrace r\left( {{A_{(2)}}} \right),r\left( {{A_{(3)}}} \right)\right\rbrace $. Similarly, $$ {r_2}(\A) \le \min\left\lbrace r\left( {{A_{(1)}}} \right),r\left( {{A_{(3)}}} \right)\right\rbrace ,~ {r_3}(\A)\le \min\left\lbrace  r\left( {{A_{(1)}}} \right),r\left( {{A_{(2)}}} \right)\right\rbrace,$$
which show the desired results.
\end{proof}

Low Tucker rank tensor completion model were considered in various references. Note that Tucker rank considers low rank structures on all modes of tensor, while
only one low rank structure in tubal rank is considered, which leads to low rank structures on the other two modes missed.
To consider low rank structures on all the three modes of tensor, it is necessary to consider multi-tubal rank in tensor completion problem.

Now we take the video tensor data in real world for example to see the low rank structures of tensors.
In video tensor\footnote{\url{http://trace.eas.asu.edu/yuv/}}, there are two spatial dimensions and one temporal dimension. We take the first 30 frames of size $ 144 \times 176 $ as a video tensor $ \A $, that is $ \A \in \mathbb{R}^{144 \times 176 \times 30} $. Figure \ref{fig:video} (a) shows the sampled frames in the video. Figure \ref{fig:video} (b) shows the first 30 singular values of the matrix $ \bar A_3^{(1)} $. %and we can see these singular values decreased slowly.
Apply SVD to $ \bar A_1^{(1)} $ and $ \bar A_2^{(1)} $ to obtain their singular values, shown in Figure \ref{fig:video} (c) and Figure \ref{fig:video} (d), respectively. From Figure \ref{fig:video} (c) and (d),   the low rank structures of tensor $\A$ on mode 1 and mode 2 are presented.

\begin{figure}[htbp]
	\centering
	\begin{subfigure}[t]{0.4\linewidth}
		\centering
		\includegraphics[width=1.8in]{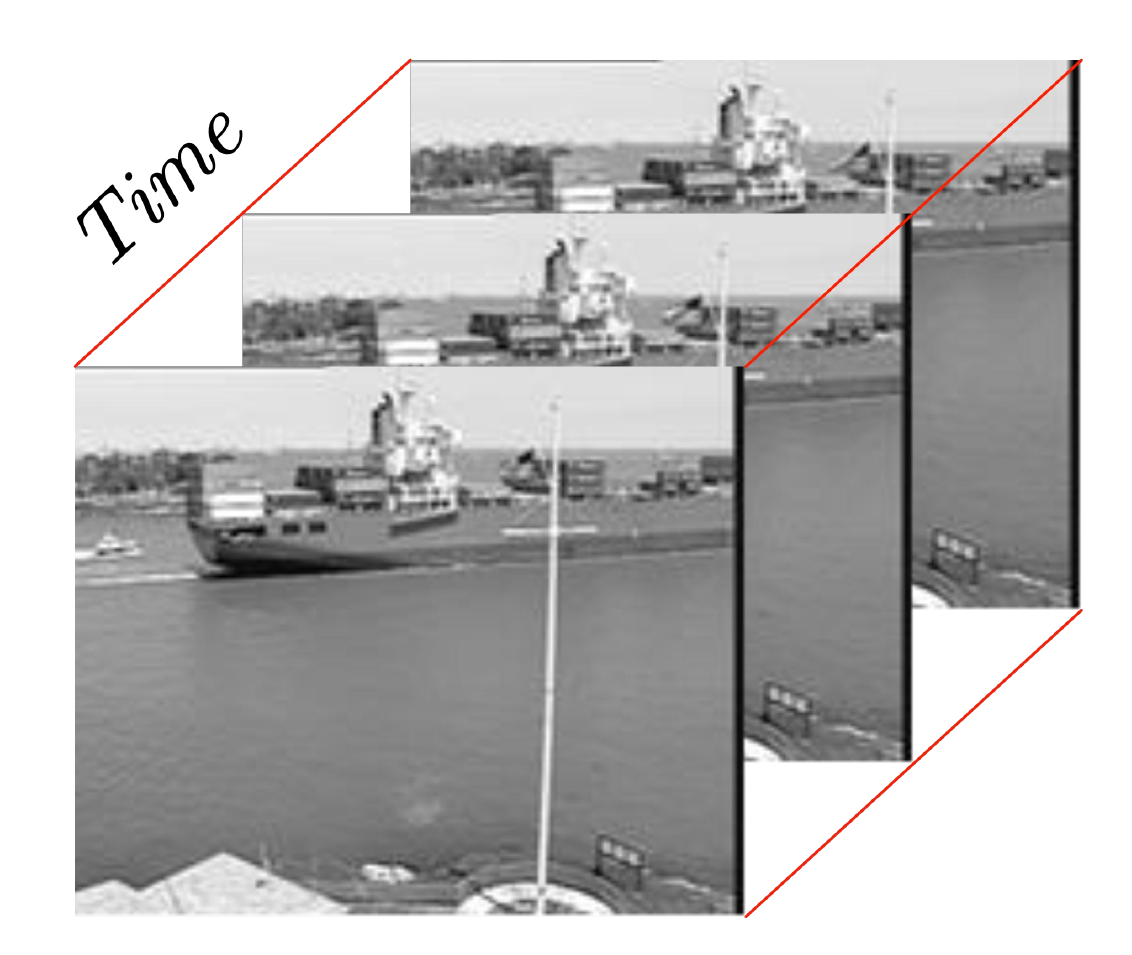}
		\caption{Sampled frames in video}
	\end{subfigure}
	\begin{subfigure}[t]{0.4\linewidth}
		\centering
		\includegraphics[width=2in]{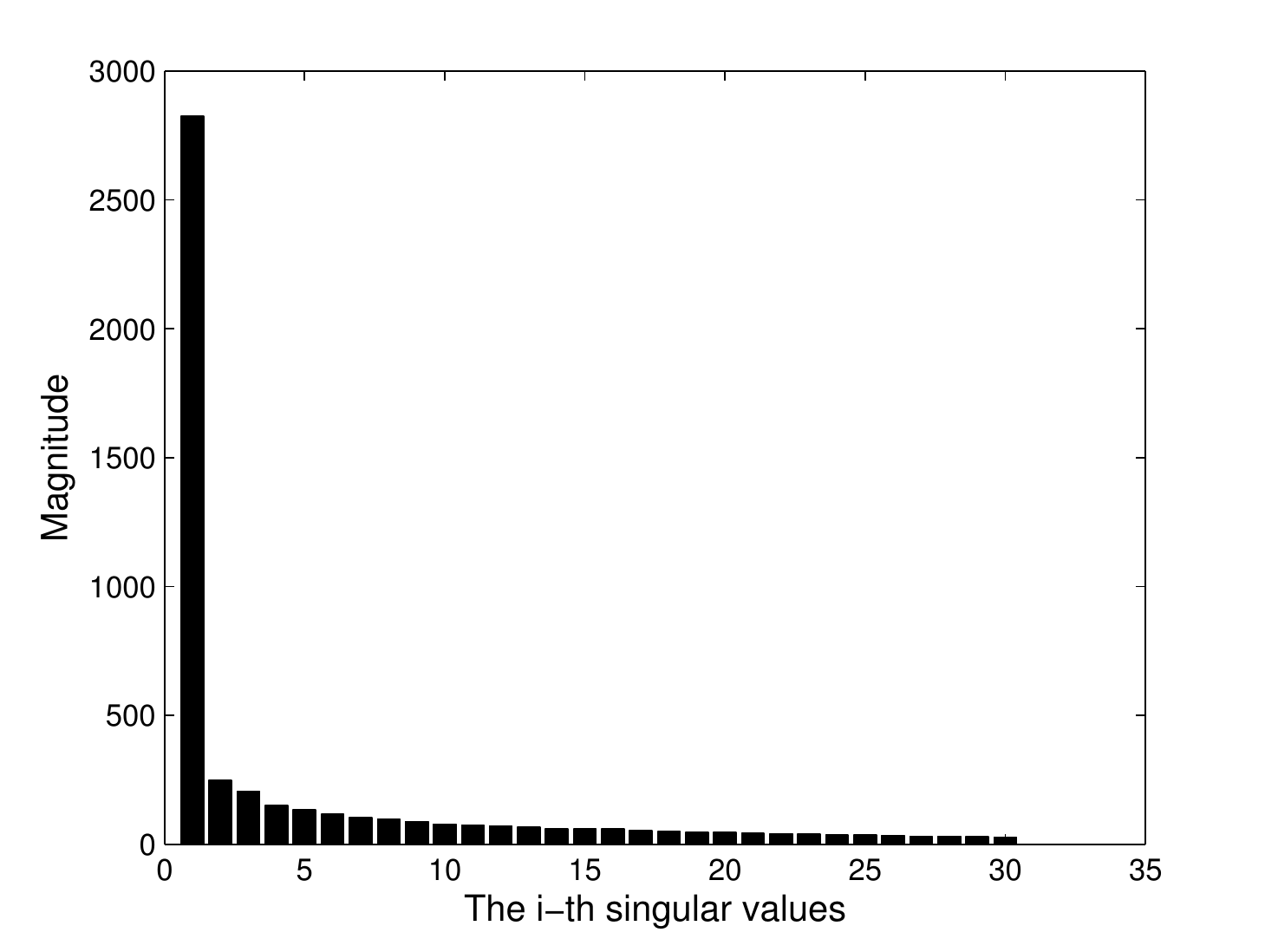}
		\caption{The first $30$ singular values of $\bar A_3^{(1)}$}
	\end{subfigure}
	
	\begin{subfigure}[t]{0.4\linewidth}
		\centering
		\includegraphics[width=2in]{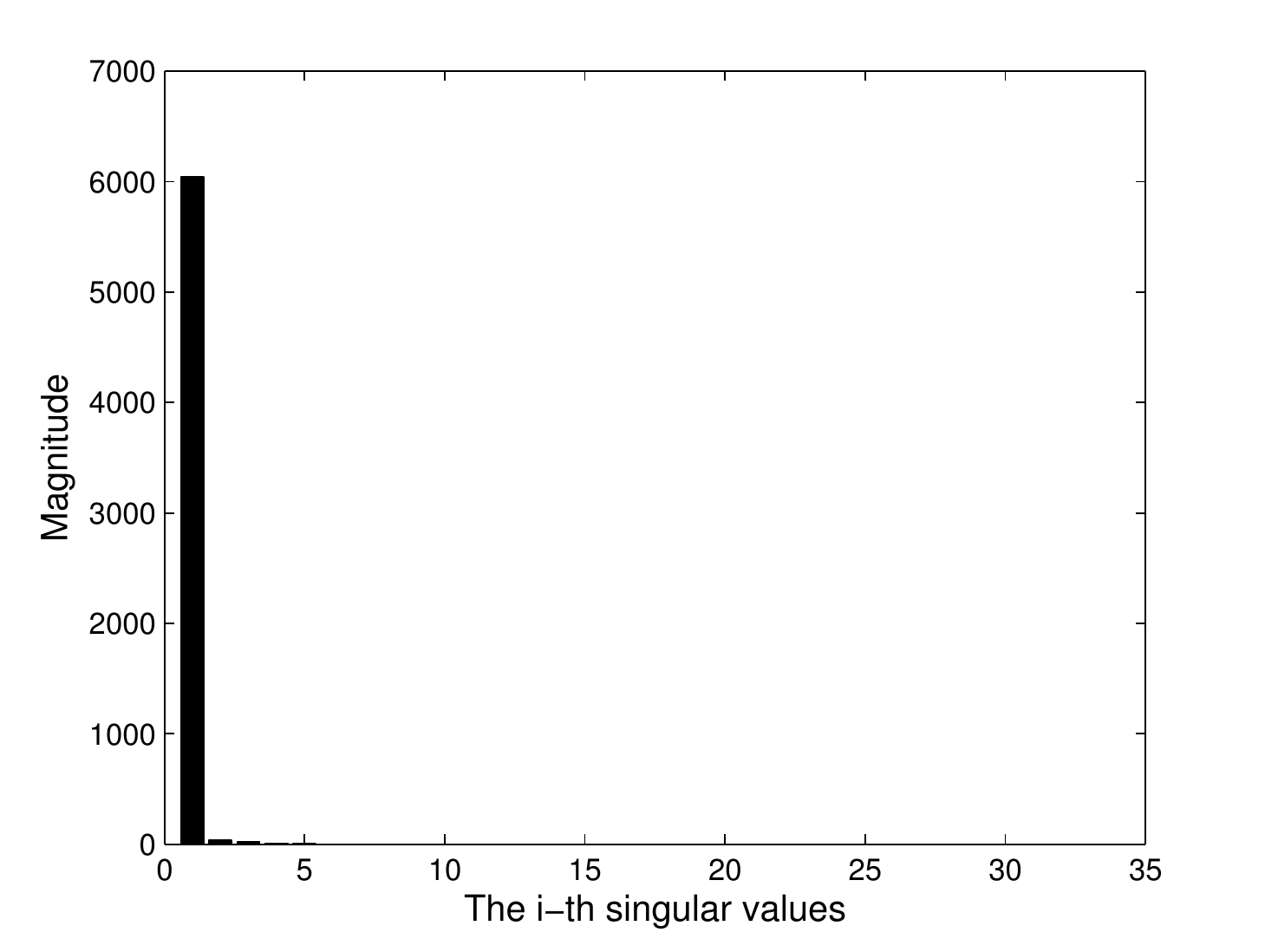}
		\caption{The singular values of $\bar A_1^{(1)}$}
	\end{subfigure}
	\begin{subfigure}[t]{0.4\linewidth}
		\centering
		\includegraphics[width=2in]{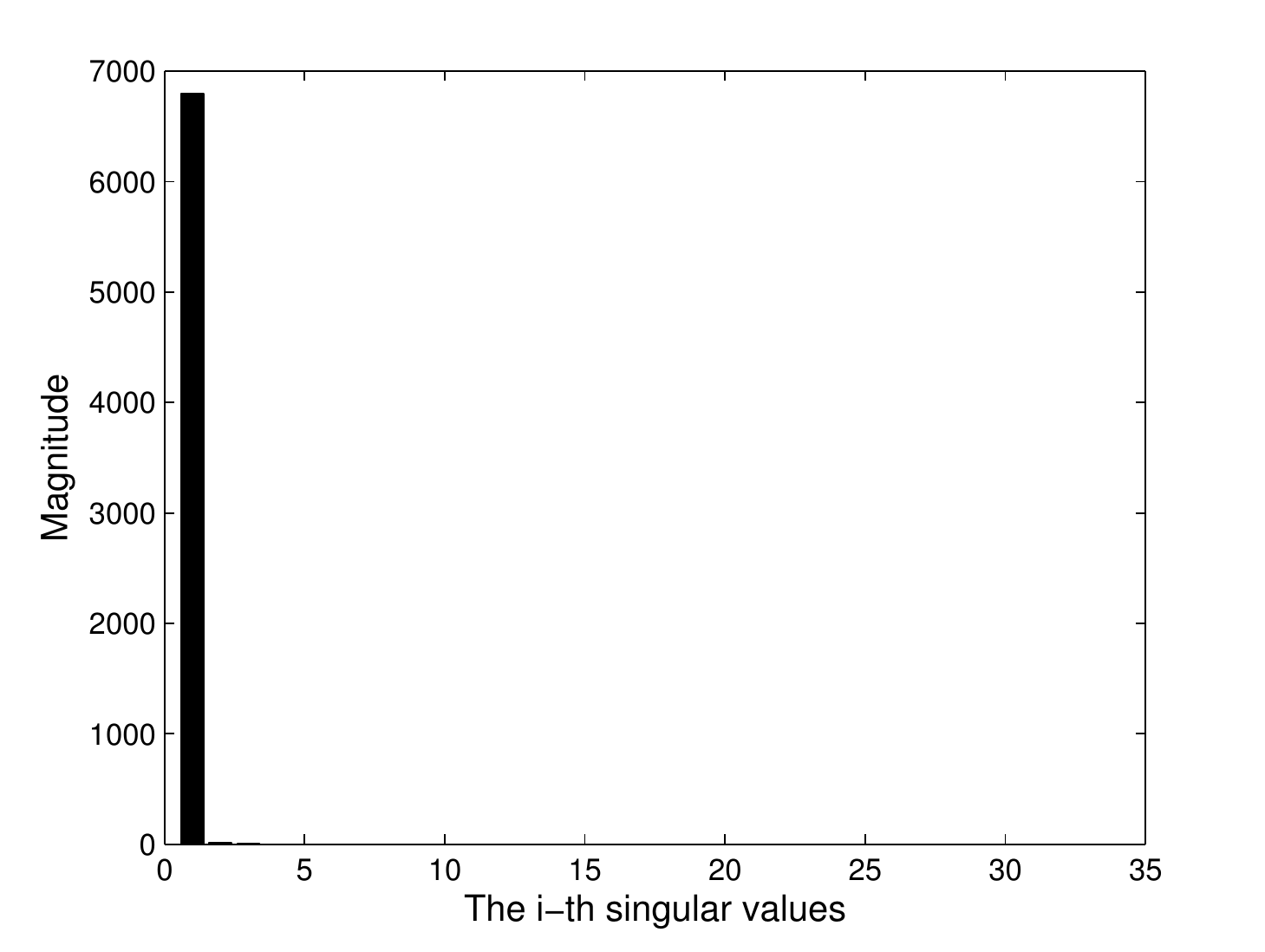}
		\caption{The singular values of $\bar A_2^{(1)}$}
	\end{subfigure}
\caption{The sampled frames in video and singular values of $\bar A_u^{(1)}$ for $u\in {\bf [3]}$ }
	\centering
	\label{fig:video}
\end{figure}

Motivated by this, we introduce multi-tubal rank, which is similar to Tucker rank. The introduced multi-tubal rank includes low rank structures on all three modes of third order tensor, which take full considerations of all low rank structures and will lead to promising performance for solving tensor completion problem.
%

%

%Based on multi-tubal rank, a new model is proposed for tensor completion problem in this section. For this new model, a rank-decreasing method is proposed \textbf{with global convergence.}

\section{Tensor completion problem based on multi-tubal rank}
In this section, we establish a low rank tensor completion based on multi-tubal rank and then apply a tensor factorization based method for solving it. For the method,  the convergence analysis will be presented.
\subsection{Tensor completion model based on multi-tubal rank and its tensor factorization based method}
Based on the introduced multi-tubal rank, the tensor completion problem can be modeled as
\begin{equation}\label{mt-vector-com}
	\mathop {\min }\limits_{{\mathcal C} \in {{\mathbb R}^{{n_1} \times {n_2} \times {n_3}}}} \operatorname{rank}_{mt}({\mathcal C}),\quad \mbox{\rm s.t.} \quad {P_\Omega }\left( {{\mathcal C} - {\mathcal M}} \right) = 0,
\end{equation}
which  is a vector optimization problem. To keep things simple, we consider the weighted multi-tubal rank minimization problem as
\[
\min \limits_{\mathcal{C} \in {\mathbb{R}^{{n_1} \times {n_2} \times {n_3}}}} \sum\limits_{u = 1}^3 {{\alpha _u}{r_u}(\mathcal{C})},\quad \mbox{\rm s.t.}\quad {P_\Omega }(\mathcal{C} - \mathcal{M}) = 0,
\]
where $\alpha_1,\alpha_2,\alpha_3\geq 0$ and $\sum\nolimits_{u = 1}^3 \alpha_u=1 $.
Note that $\C$ can be factorized as $\C=\X_u\ast_{u}\Y_u$ with $r_u(\C)\leq \min (r_u(\X_u), r_u(\Y_u))$ for $u\in {\bf [3]}$. Hence we consider the following  tensor factorization model
\begin{equation}\label{mt-tensor}
	\min\limits_{\C,\X_u,\Y_u} \sum\limits_{u = 1}^3
	\frac{\alpha_u}{2}\left\| \X_u \ast_u\Y_u - \C
	\right\|_F^2, \quad \mbox{\rm s.t.} \quad {P_\Omega }(\mathcal{C} - \mathcal{M}) = 0.
\end{equation}
To solve problem \eqref{mt-tensor} more conveniently, we introduce its regularized model as follows:
\begin{equation}\label{mt-tensor-xy}
	\min\limits_{\C,\X_u,\Y_u}  f(\C,\X_1,\X_2,\X_3,\Y_1,\Y_2,\Y_3),\quad \mbox{\rm s.t.} \quad {P_\Omega }(\mathcal{C} - \mathcal{M}) = 0.
\end{equation}
Here,
\begin{equation}\label{f-opt}
	f(\C,\X_1,\X_2,\X_3,\Y_1,\Y_2,\Y_3)
	=\sum\limits_{u = 1}^3 \left(\dfrac{\alpha_u}{2}\left\| \X_u \ast_u\Y_u - \C
	\right\|_F^2+\dfrac{\lambda}{2}\left(\|\X_u\|_F^2+\|\Y_u\|_F^2\right)\right).
\end{equation}
Now, we are ready to update $\C,\, \X_u$ and $\Y_u$ for all $u\in  {\bf [3]}$. %First of all, we update $\C$ at iteration $ t+1 $.
Note that
\begin{equation}\label{qeu:c}
	\begin{aligned}
		&\quad\sum\limits_{u = 1}^3 {{\alpha _u}\left\| {{\mathcal{X}_u}{ \ast _u}{\mathcal{Y}_u} - \mathcal{C}} \right\|_F^2}
		= \sum\limits_{u = 1}^3 {{\alpha _u}\left\langle {{\mathcal{X}_u}{ \ast _u}{\mathcal{Y}_u} - \mathcal{C},{\mathcal{X}_u}{ \ast _u}{\mathcal{Y}_u} - \mathcal{C}} \right\rangle }   \\
		&= \sum\limits_{u = 1}^3 {{\alpha _u}\left\langle {\mathcal{C},\mathcal{C}} \right\rangle }  - 2\sum\limits_{u = 1}^3
		{{\alpha _u}\left\langle {{\mathcal{X}_u}{ \ast _u}{\mathcal{Y}_u},\mathcal{C}} \right\rangle }+ \sum\limits_{u = 1}^3 {{\alpha _u}\left\langle {{\mathcal{X}_u}{ \ast _u}{\mathcal{Y}_u},{\mathcal{X}_u}
				\ast _u{\Y_u}} \right\rangle }   \\
		&= \left\langle {\mathcal{C},\mathcal{C}} \right\rangle  - 2\left\langle {\sum\limits_{u = 1}^3 {{\alpha _u}{\mathcal{X}_u}{ \ast _u}{\mathcal{Y}_u}} ,\mathcal{C}} \right\rangle  + \sum\limits_{u = 1}^3 {{\alpha _u}\left\| {{\mathcal{X}_u}{ \ast _u}\Y_u} \right\|_F^2}   \\
		&= \left\langle {\sum\limits_{u = 1}^3 {{\alpha _u}{\mathcal{X}_u}{ \ast _u}{\mathcal{Y}_u}}  - \mathcal{C},\sum\limits_{u = 1}^3 {{\alpha _u}{\mathcal{X}_u}{ \ast _u}{\mathcal{Y}_u}}  - \mathcal{C}} \right\rangle
		+ \sum\limits_{u = 1}^3 {{\alpha _u}\left\| {{\mathcal{X}_u}{ \ast _u}\mathcal{Y}_u} \right\|_F^2}  - \left\| {\sum\limits_{u = 1}^3 {{\alpha _u}{\mathcal{X}_u}{ \ast _u}{\mathcal{Y}_u}} } \right\|_F^2  \\
		&= \left\| {\sum\limits_{u = 1}^3 {{\alpha _u}{\mathcal{X}_u}{ \ast _u}{\mathcal{Y}_u}}  - \mathcal{C}} \right\|_F^2 + \sum\limits_{u = 1}^3 {{\alpha _u}\left\| {{\mathcal{X}_u}{ \ast _u}\mathcal{Y}_u} \right\|_F^2}- \left\| {\sum\limits_{u = 1}^3 {{\alpha _u}{\mathcal{X}_u}{ \ast _u}{\mathcal{Y}_u}} } \right\|_F^2.
	\end{aligned}
\end{equation}
Then $\C^{t+1}$ can be updated by
\begin{equation}\label{C-comp-com}
	\C^{t+1}=\mathop {\operatorname{argmin}} \limits_{{P_\Omega }(\C - \mathcal M) = 0} \frac{1}{2}\left\|
	\sum\limits_{u = 1}^3 \alpha _u {\X_u^t} { \ast _u} {\Y}_u^t  - \C \right\|_F^2\\
	=\sum\limits_{u = 1}^3 {{\alpha _u}{\X_u^t}{ \ast _u}{\Y_u^t}}  + {P_\Omega }\left( \mathcal{M} - \sum\limits_{u = 1}^3 {\alpha _u}{\X_u^t}{ \ast _u}{\Y_u^t} \right).
\end{equation}

Before we present how to update $\X_u^{t+1}$ and $\Y_u^{t+1}$, we rewrite (\ref{mt-tensor-xy}) as a corresponding matrix version.
Denote $r_u:=r_u(\C)$, $r_u^l:=r_u^l(\bar C_u^{(l)})$ with $\bar C_u^{(l)}\in {\mathbb{C}}^{n_{u_1}\times n_{u_2}}$, $u_1< u_2$ and $u_1, u_2\neq u$. Clearly, $r_u^l\leq r_u$ for all $l\in {\bf [n_u]}$. For each $u$ and $l$,  $\bar C_u^{(l)}$ can be factorized as a product of two matrices $\hat X_u^{(l)}$ and $\hat Y_u^{(l)}$ of smaller sizes,
where $\hat X_u^{(l)}\in {\mathbb{C}}^{n_{u_1}\times r_u^l}$ and $\hat Y_u^{(l)}\in {\mathbb{C}}^{r_u^l\times n_{u_2}}$ are the $l$th block diagonal matrices of
$\hat X_u\in {\mathbb{C}}^{n_{u_1}n_{u}\times \left(\sum\limits_{l=1}^{n_u}r_u^l\right)}$ and $\hat Y_u\in {\mathbb{C}}^{\left(\sum\limits_{l=1}^{n_u}r_u^l\right)\times n_{u}n_{u_2}}$. Let $\bar X_u^{(l)}=[\hat X_u^{(l)}, 0]\in \mathbb{C}^{n_{u_1}\times  r_u}$, $\bar Y_u^{(l)}=[\hat Y_u^{(l)};0]\in \mathbb{C}^{r_u\times n_{u_2}}$  and $\bar X_u, \bar Y_u$ be the block diagonal matrices with the $l$th block diagonal matrices $\bar X_u^{(l)},\bar Y_u^{(l)}$, respectively. Then $
\hat X_u\hat Y_u=\bar X_u\bar Y_u$. Together with Lemma \ref{lem:equ}, we have
\begin{equation*}
	\|\X_u*_u\Y_u-\C\|_F^2=\frac{1}{n_u}\|\bar X_u\bar Y_u-\bar C_u\|_F^2= \frac{1}{n_u}\|\hat X_u\hat Y_u-\bar C_u\|_F^2 =\frac{1}{n_u}\sum\limits_{l=1}^{n_u}\|\hat X_u^{(l)}\hat Y_u^{(l)}-\bar C^{(l)}_u\|_F^2,~u\in  {\bf [3]}.
\end{equation*}
Therefore, (\ref{mt-tensor-xy}) can be rewritten as
\begin{equation}\label{mt-finally}
	\begin{array}{rcl}
		&\min\limits_{\C,\X_u,\Y_u}&\sum\limits_{u = 1}^3 \sum\limits_{l = 1}^{n_u} \left(\dfrac{\alpha_u}{2n_u}\left\|\hat X_u^{(l)}\hat Y_u^{(l)} - \bar C_u^{(l)}\right\|_F^2\right)+\sum\limits_{u = 1}^3 \sum\limits_{l = 1}^{n_u} \left(\dfrac{\lambda }{2n_u}\left\|\hat X_u^{(l)} \right\|_F^2+\dfrac{\lambda }{2n_u}\left\| \hat Y_u^{(l)} \right\|_F^2\right) \\
		& \mbox{\rm s.t.} &\: {P_\Omega }(\C - \mathcal{M}) = 0.
	\end{array}
\end{equation}

To update $\hat X_u^{(l,t)}$, we consider its regularized version and have $\hat X_u^{(l,t+1)}$ as follows.
\begin{equation}\label{xl-comp}
	\begin{aligned}
		\hat X_u^{(l,t+1)} &= \mathop {\operatorname{argmin} }\limits_{\hat X_u^{(l)}} \frac{\alpha_{u}}{{2{n_u}}}\left\| {\hat X_u^{(l)}\hat Y_u^{(l,t)} - \bar C_u^{(l,t+1)}} \right\|_F^2 + \frac{\lambda }{{2{n_u}}}\left( \left\| {\hat X_u^{(l)}} \right\|_F^2 +\left\| {\hat X_u^{(l)}}-{\hat X_u^{(l,t)}} \right\|_F^2  \right) \\
		&= \left( \lambda{\hat X_u^{(l,t)}}+\alpha_u\bar C_u^{(l,t+1)}{\left( {\hat Y_u^{(l,t)}} \right)^*}\right) {\left( \alpha_u{\hat Y_u^{(l,t)}{{\left( {\hat Y_u^{(l,t)}} \right)}^*} + 2\lambda I} \right)^{ - 1}},~ \forall u\in  {\bf [3]},~ \forall l \in {\bf [n_u]}.
	\end{aligned}
\end{equation}

Similarly, $\hat Y_u^{(l,t+1)}$ can be updated by
\begin{equation}\label{yl-comp}
	\begin{aligned}
		\hat Y_u^{(l,t+1)} &= \mathop {\operatorname{argmin} }\limits_{\hat Y_u^{(l)}} \frac{\alpha_{u}}{{2{n_u}}}\left\| {\hat X_u^{(l,t+1)}\hat Y_u^{(l)} - \bar C_u^{(l,t+1)}} \right\|_F^2 + \frac{\lambda }{{2{n_u}}}\left( \left\| {\hat Y_u^{(l)}} \right\|_F^2+\left\| {\hat Y_u^{(l)}}-{\hat Y_u^{(l,t)}} \right\|_F^2 \right)
		\\&= {\left( {{{\alpha_{u}\left( {\hat X_u^{(l,t+1)}} \right)}^*}\hat X_u^{(l,t+1)} + 2\lambda I} \right)^{ - 1}}\left( \lambda{\hat Y_u^{(l,t)}}+\alpha_{u}{\left({\hat X_u^{(l,t+1)}} \right)^*}\bar C_u^{(l,t+1)}\right) ,~ \forall u\in  {\bf [3]},~ \forall l \in {\bf [n_u]}.
	\end{aligned}
\end{equation}

Based on above discussions, a tensor factorization algorithm can be outlined as Algorithm 3.1, denoted by  MTRTC.
\begin{table}[htbp]
	\centering
	% \begin{center}
	\begin{tabular}{l}
		\toprule
		\toprule
		{\bfseries Algorithm 3.1} Multi-Tubal Rank Tensor Completion (MTRTC)       \\
		\midrule
		{\bfseries Input:} The tensor data ${\mathcal M} \in {{\mathbb R}^{{n_1} \times {n_2} \times {n_3}}}$, the observed set
		$\Omega $,the initialized rank\\\qquad \quad \,\!  $ R^0$, parameters \,$\lambda$,\, $ \varepsilon$ and $\alpha_u,$\, $u\in [\bf 3]$.                                  \\
		{\bfseries Initialization:} $\hat X_u^0,\,\hat Y_u^0,\,u\in[\bf 3] $.                                           \\
		{\bfseries While not converge do}                \\
		\qquad  $\bm {1.}$ Fix $\hat X_u^t$ and $ \hat Y_u^t $ to compute $\mathcal C^{t+1} $ by \eqref{C-comp-com}.\\		
		\qquad  $ \bm{2.} $ Fix $\hat Y_u^t$ and $\mathcal C^{t+1} $ to update $\hat X_u^{t+1}$ by \eqref{xl-comp}.     \\
		\qquad  $ \bm{3.} $ Fix $\hat X_u^{t+1}$ and $\mathcal C^{t+1} $ to update $\hat Y_u^{t+1}$ by \eqref{yl-comp}.     \\
		\qquad  $ \bm{4.} $ Adopt the rank decreasing scheme  to adjust $ rank_{mt}(\C) $ and  adjust the sizes of \\
 \qquad \quad \,\!
 $\hat X_u^{t+1}$ and $ \hat Y_u^{t+1} $. \\
		%\qquad  $ \bm{5.} $ Adjust $ \alpha_{u}^{t+1} $ by \eqref{alpha}.\\
		\qquad  $ \bm{5.} $ Check the stop criterion: ${\left\| {\C_\Omega ^{t + 1} - {\M_\Omega }} \right\|_F}/{\left\| {{\M_\Omega }} \right\|_F} < \varepsilon $.  \\
		\qquad  $ \bm{6.} $ $t \leftarrow t + 1$.            \\
		{\bfseries end while}                             \\
		{\bfseries Output:}  $\mathcal C^{t + 1}  $.      \\		
		\bottomrule
		\bottomrule
	\end{tabular}
	% \end{center}
\end{table}
\begin{remark}In general, we do not know the true multi-tubal rank of optimal tensor $\C$ in advance. Thus, it is necessary to  estimate the multi-tubal rank of tensor $\C$. In this paper,  we adopt the same rank  estimation and rank decreasing strategy proposed in \cite{WYZ12,XHYS13, ZLLZ18}.
\iffalse	
	The parameters $ \alpha_1,\,\alpha_2 $ and $ \alpha_3 $ in \eqref{mt-tensor-xy}  were uniformly set to $ 1/3 $ at the beginning of algorithm. During the iterations, we either fixed them or dynamically updated them according to the fitting error
	\[{\pi_u ^{t+1}} = {\left\| {{P_\Omega }\left( {\M - \X_u^{t+1}*\Y_u^{t+1}} \right)} \right\|_F}.\]
	The smaller $ \pi_u ^{t+1} $ is, the larger $ \alpha_{u}^{t+1} $ should be. Specifically, we set
	\begin{equation}\label{alpha}
		\alpha _u^{t+1} = \frac{{{{\left( {\pi _u^{t+1}} \right)}^{ - 1}}}}{{\sum\nolimits_{u = 1}^3 {{{\left( {\pi _u^{t+1}} \right)}^{ - 1}}} }},\quad u \in [\bf 3].
	\end{equation}
\fi
	%Similar to complexity analysis in \cite{ZLLZ18}, the cost of each iteration in MTRTC is at most $\mathcal{O}(r(n_1+n_2)n_3\log n_3+	r(n_1+n_3)n_2\log r(n_1+n_3)n_2\log n_2+r(n_2+n_3)n_1\log n_1+ 3rn_1n_2n_3)$. We can see that MTRTC has lower computational complexity than some existing methods, other than TCTF. However, from experimental results reported in Section 4, the proposed method MTRTC performs better than TCTF, see Section 4 for details.
\end{remark}
\subsection{Convergence analysis}
In this subsection, we present the convergence of MTRTC. The following notation will be used in our analysis. In problem (\ref{mt-tensor-xy}), $\Omega$ is an index set which locates the observed data. We use $\Omega^c$ to denote the complement of the set $\Omega$ with respect to the set $\{(i,j,k): i\in {\bf [n_1]},j\in{\bf [n_2]},k\in {\bf [n_3]}\}$. To simply the notation, we denote $z^t=\left( \C^t, \X_1^t,\X_2^t,\X_3^t,\Y_1^t,\Y_2^t,\Y_3^t \right)$ in this subsection.

Before proceeding, we present the Kurdyka-Lojasiewicz (KL) property \cite{ABS13} with constraint defined as below.

\begin{definition}\label{KL}{\bf (Kurdyka-Lojasiewicz (KL) property)} Let $Z$ be an open set and $f:Z\to \re$ be a semi-algebra function. For every critical point $z^\star\in Z$ of $f$, there are a neighborhood of $z^\star$, denoted by $Z'\subset Z$, an exponent $\theta\in [0,1)$ and a positive constant $\mu$ such that
\begin{equation}\label{kl}
|f(z)-f(z^\star)|^\theta\leq \mu \left\| \prod\nolimits_\Omega \left(\nabla f(z)\right)\right\|_F	
\end{equation}	
for all $z\in Z'$, where $\prod_\Omega (\nabla f(z))$ denotes the projective gradient of $f$.
\end{definition}

Recall that $f(z)$ defined as in \reff{f-opt}, $f(z)$ is a quadratic function on $z$, and hence is a semi-algebra function. From Definition \reff{KL}, for any critical point $z^\star$, there exist $\theta$ and $\mu$ such that \reff{kl} is satisfied.
\begin{theorem}\label{th:ca}
	Suppose that $\{z^t\}$ is an infinite sequence generated by MTRTC. Then we have the following statements.
	\begin{itemize}
		\item [(1).] The sequence $\{z^t\}$ is bounded and any accumulation point of $\{z^t\}$ is a stationary point of problem \eqref{f-opt}.
		\item [(2).] There is a constant $\eta>0$ such that $\eta\|z^t-z^{t+1}\|_F\geq \|\prod_\Omega (\nabla f(z^t))\|_F$.
	\end{itemize}
\end{theorem}
\begin{proof} Since rank $r\geq 0$ in Algorithm MTRTC is non-increasing, we can assume that
 the rank $r$ is fixed for all $z^t$ when $t$ is sufficiently large. That is, the rank decreasing scheme is not adopted for all such big enough $t$. For simplicity, we assume that $t$ is big enough such that $r$ is fixed and denote  ${f^t} = f(z^t)$ in the following.

(1). By (\ref{C-comp-com}), it follows
$$\begin{array}{rl} \|\C^{t+1}-\C^t\|_F^2
	&=\left\|\sum\limits_{u = 1}^3 {{\alpha _u}\mathcal{X}_u^t}{ \ast _u}{\mathcal{Y}_u^t}  + {P_\Omega }\left( {\mathcal{M} - \sum\limits_{u = 1}^3 {{\alpha _u}{\mathcal{X}_u^t}{ \ast _u}{\mathcal{Y}_u^t}} } \right)-\C^t\right\|_F^2\\
	&=\left\|\sum\limits_{u = 1}^3 {\alpha _u}{\mathcal{X}_u^t}{ \ast _u}{\mathcal{Y}_u^t}-\C^t  + {P_\Omega }\left( {\mathcal{M} - \sum\limits_{u = 1}^3 {{\alpha _u}{\mathcal{X}_u^t}{ \ast _u}{\mathcal{Y}_u^t}} } \right)\right\|_F^2\\
	&=\left\|\left( {\sum\limits_{u = 1}^3 {{\alpha _u}\mathcal{X}_u^t{ * _u}\mathcal{Y}_u^t}  - {\mathcal{C}^t}} \right)_{\Omega ^c} \right\|_F^2.
\end{array}$$

According to Algorithm MTRTC, we have that
\begin{equation}\label{fk-fk+1}
	\begin{aligned}	
		f^t-f^{t+1}
		&= \sum\limits_{u=1}^3 \left(\frac{\alpha_u}{2}\left\| \X^t_u\ast_u \Y^t_u-\C^t\right\|_F^2+\frac{\lambda}{2}\left(\left\| \X^t_u\right\|_F^2+\left\| \Y^t_u\right\|_F^2\right)\right)\\
		&\quad-\sum\limits_{u=1}^3 \left(\frac{\alpha_u}{2} \left\| \X^{t+1}_u\ast_u \Y^{t+1}_u-\C^{t+1}\right\|_F^2+\frac{\lambda}{2}\left(\left\| \X^{t+1}_u\right\|_F^2+\left\|\Y^{t+1}_u\right\|_F^2\right)\right)\\
		&=\sum\limits_{u=1}^3 \frac{\alpha_u}{2}\left( \left\| \X^t_u\ast_u \Y^t_u-\C^t\right\|_F^2-\left\| \X^t_u\ast_u \Y^t_u-\C^{t+1}\right\|_F^2\right)\\
		&\quad +\sum\limits_{u=1}^3 \frac{\alpha_u}{2} \left(\left\| \X^t_u\ast_u \Y^t_u-\C^{t+1}\right\|_F^2-\left\| \X^{t+1}_u\ast_u \Y^t_u-\C^{t+1}\right\|_F^2\right) +\sum\limits_{u=1}^3\frac{\lambda}{2}\left(\left\| \X^{t}_u\right\|_F^2-\left\| \X^{t+1}_u\right\|_F^2\right) \\
		&\quad +\sum\limits_{u=1}^3\frac{\alpha_u}{2}\left(\left\|\X^{t+1}_u\ast_u \Y^t_u-\C^{t+1}\right\|_F^2-\left\| \X^{t+1}_u\ast_u \Y^{t+1}_u-\C^{t+1}\right\|_F^2\right) +\sum\limits_{u=1}^3\frac{\lambda}{2}\left(\left\| \Y^{t}_u\right\|_F^2-\left\|\Y^{t+1}_u\right\|_F^2\right) \\
		&\geqslant  \frac{1}{2}\left( \left\| \sum\limits_{u=1}^3\alpha_u\X^t_u\ast_u \Y^t_u-\C^t\right\|_F^2-\left\| \sum\limits_{u=1}^3\alpha_u\X^t_u\ast_u \Y^t_u-\C^{t+1}\right\|_F^2\right)\\&\quad+\sum\limits_{u=1}^3\sum\limits_{l=1}^{n_u}\frac{\lambda}{2n_u}\left(\left\|\hat X_u^{(l,t)}-\hat X_u^{(l,t+1)}\right\|_F^2+\left\|\hat Y_u^{(l,t)}-\hat Y_u^{(l,t+1)}\right\|_F^2\right)\\
		&=\frac{1}{2}\left\| \left(\sum\limits_{u=1}^3 \alpha_u \X_u^t*_u\Y_u^t-\C^t\right)_{\Omega^c}\right\| _F^2+\frac{\lambda}{2}\sum\limits_{u=1}^3\left(\left\|\X_u^t-\X_u^{t+1}\right\|_F^2+\left\| \Y_u^t-\Y_u^{t+1}\right\|_F^2\right)\\
		&=\frac{1}{2}\left\|\C^{t+1}-\C^t\right\|_F^2+\frac{\lambda}{2}\sum\limits_{u=1}^3\left(\left\|\X_u^t-\X_u^{t+1}\right\|_F^2+\left\|\Y_u^t-\Y_u^{t+1}\right\|_F^2\right)\\&\geq \frac{\min\{1,\lambda\}}{2}\left\|z^{t+1}-z^t\right\|_F^2,	
	\end{aligned}
\end{equation}
%	\end{small}
where the first inequality holds from (\ref{C-comp-com}), (\ref{xl-comp}) and (\ref{yl-comp}).
Therefore, $ \{f^{t}\} $ is monotonically decreasing. Together with the fact that $f\geq 0$,  the series $ \sum\limits_{t=1}^\infty (f^t- f^{t + 1})=f^1-\lim\limits_{t\to \infty}{f^t}$ converges. Hence, $$\sum\limits_{t=1}^\infty({f^t}- {f^{t + 1}})<\infty,\quad \sum\limits_{t=1}^\infty(\C^{t+1}-\C^t)<\infty, \quad \sum\limits_{t=1}^\infty\|z^t-z^{t+1}\|_F^2< \infty.$$

Since
$f^t \geqslant \sum\limits_{u = 1}^3 {\dfrac{\lambda}{{2{n_u}}}\left( {\left\| {{{\hat X}^t_u}} \right\|_F^2 + \left\| {\hat Y_u}^t \right\|_F^2}\right)}=\dfrac{\lambda}{2}\sum\limits_{u=1}^3 \left(\left\|\X_u^t\|_F^2+\|\Y_u^t\right\|_F^2\right)$, \,
$\{\X_u^t\}$, $\{\Y_u^t\}$ are bounded. Together with the expression of $\C^{t}$,
it is asserted that
$ \{\C^t\} $ is also bounded,  and hence $\{z^t\}$ is bounded.

Clearly, there exists a
convergent subsequence of $\{z^t\}$.
Without loss of generality, we assume that $\lim\limits_{k\to\infty}z^{t_k}=z^\star$. From $\sum\limits_{t=1}^\infty\|z^t-z^{t+1}\|_F^2< \infty$,
$\lim\limits_{t\to \infty }z^{t+1}-z^t=0$, and hence $\lim\limits_{k\to \infty} z^{t_k+1}=z^\star$.

Together with (\ref{C-comp-com}), (\ref{xl-comp}) and (\ref{yl-comp}), we have that
$$\left\{\begin{array}{rl}\C^\star&=\sum\limits_{u=1}^3 \alpha_u \X_u^\star*\Y_u^\star+P_{\Omega}\left(\M-\sum\limits_{u=1}^3\alpha_u \X_u^\star *_u\Y_u^\star\right) ,\\
	\hat X_u^{(l,\star)}&=\alpha_u\bar C_u^{(l,\star)}\left(Y^{(l,\star)}\right) ^*\left(\alpha_uY^{(l,\star)} \left(Y^{(l,\star)}\right)^*+\lambda I\right)^{-1},\,{u\in [\bf 3]}, \, l\in [\bf n_u],\\
	\hat Y_u^{(l,\star)}&=\alpha_u\left(\alpha_uX^{(l,\star)} \left(X^{(l,\star)}\right)^*+\lambda I\right)^{-1}\left( X_u^{(l,\star)}\right)^*\bar C_u^{(l,\star)}, \,{u\in [\bf 3]}, \, l\in [\bf n_u].
\end{array}\right.$$

By direct computation, the following system is asserted
$$\left\{\begin{array}{rl}&\alpha_u(\hat X_u^\star\hat Y_u^\star-\bar C_u^\star)(\hat Y^\star)^*+\lambda \hat X^\star=0, \quad \forall u\in  {\bf [3]}\\
	&\alpha_u(\hat X_u^\star)^*(\hat X_u^\star\hat Y_u^\star-\bar C_u^\star)+\lambda \hat Y_u^\star=0,\quad \forall u\in  {\bf [3]}\\
	&P_{\Omega}(\C^\star-\M)=0,\\
	&P_{\Omega ^c}\left( \sum\limits_{u = 1}^3 {{\alpha _u}\mathcal{X}_u^\star{ \ast_u}\mathcal{Y}_u^\star}  - \mathcal{C}^\star\right)=0.
	%P_\Omega(\C^\star-\sum\limits_{u=1}^3\alpha_u \X_u^\star \ast_u \Y_u^\star)+\mathcal{Q}=0,
\end{array}\right.$$
%where  $\mathcal{Q}=-P_\Omega(\C^\star-\sum\limits_{u=1}^3\alpha_u \X_u^\star \ast_u \Y_u^\star)$.
Therefore,  $z^\star$ is a stationary point  of problem (\ref{f-opt}).

(2). Since $\{z^t\}$ is bounded, there exists a compact convex set $Z$ such that $\{z^t\}\subset Z$.
Since $f$ is a quadratic polynomial in $z$, the gradient $\nabla f$ is Lipschitz in $Z$ with a Lipschitz
constant $L_f$, that is,  $$\|\nabla f(z)-\nabla f(z')\|_F\leq L_f\|z-z'\|_F, \quad \forall \, z,z'\in Z.$$

Clearly, \[\begin{array}{rl}&\left\| \prod_\Omega\left( \nabla_{\C} f\left( z^{t+1}\right)\right) \right\| _F\\\leq& \left\| \prod_\Omega\left( \nabla_{\C} f\left( \C^{t+1},\X_1^{t+1},\dots, \Y_3^{t+1}\right)\right) -\prod_\Omega\left( \nabla_{\C}f\left( \C^{t},\X_1^t,\dots, \Y_3^t\right)\right) \right\| _F\\
&+	\left\| \prod_\Omega\left( \nabla_{\C}f\left( \C^{t},\X_1^t,\dots, \Y_3^t\right)\right) \right\| _F\\
	\leq& L_f\left\| z^{t+1}-z^t\right\|_F+\left\|\left(\C^{t}-\sum\limits_{u=1}^3 \alpha_u \X_u^t*_u\Y_u^t\right)_{\Omega^c}\right\|_F\\
	=&L_f\left\| z^{t+1}-z^t\right\|_F+\left\| \C^{t+1}-\C^t\right\| _F\\
	\leq& \left( L_f+1\right)\left\| z^{t+1}-z^t\right\| _F.
\end{array}
\]
\iffalse
Clearly, \[\begin{array}{rl}&\left\| \nabla_{\C} f\left( z^{t+1}\right)\right\| _F\\\leq& \left\| \nabla_{\C} f\left( \C^{t+1},\X_1^{t+1},\dots, \Y_3^{t+1}\right)-\nabla_{\C}f\left( \C^{t+1},\X_1^t,\dots, \Y_3^t\right)\right\| _F+
	\left\| \nabla_{\C}f\left( \C^{t+1},\X_1^t,\dots, \Y_3^t\right)\right\| _F\\
	\leq& L_f\sum\limits_{u=1}^3\left(\left\|\X_u^{t+1}-\X_u^t\right\|_F+\left\|\Y_u^{t+1}-\Y_u^t\right\|_F\right)+\left\|\left(\C^{t+1}-\sum\limits_{u=1}^3 \alpha_u \X_u^t*_u\Y_u^t\right)_{\Omega^c}\right\|_F\\
	=&L_f\sum\limits_{u=1}^3\left( \left\| \X_u^{t+1}-\X_u^t\right\|_F+\left\| \Y_u^{t+1}-\Y_u^t\right\|_F\right)+\left\| \C^{t+1}-\C^t\right\| _F\\
	\leq& \left( L_f+1\right)\left\| z^{t+1}-z^t\right\| _F.
\end{array}
\]
\fi

Furthermore,
\[\begin{array}{rl}&\left\|\nabla_{\X_1} f\left(z^{t+1}\right)\right\|_F\\\leq& \left\|\nabla_{\X_1} f\left(\C^{t+1},\X_1^{t+1},\dots, \Y_3^{t+1}\right)-\nabla_{\X_1}f\left(\C^{t+1},\X_1^{t+1},\dots, \Y_3^t\right)\right\|_F+
	\left\|\nabla_{\X_1}f\left(\C^{t+1},\X_1^{t+1},\dots, \Y_3^t\right)\right\|_F\\
	\leq& L_f \|z^{t+1}-z^t\|_F+\lambda\|\X_u^{t+1}-\X_u^t\|_F\\
	\leq& \left(L_f+\lambda\right)\left\|z^{t+1}-z^t\right\|_F.
\end{array}
\]
Similarly, for any $u\in [\bf 3]$, we have
\[\begin{array}{rl}&\left\| \nabla_{\X_u} f(z^{t+1})\right\|_F
	\leq \left(L_f+\lambda\right)\left\| z^{t+1}-z^t\right\|_F,\\
	&\left\|\nabla_{\Y_u} f(z^{t+1})\right\|_F
	\leq \left(L_f+\lambda\right)\left\|z^{t+1}-z^t\right\|_F.
\end{array}
\]
Now we can assert that $\|\prod_\Omega\left(\nabla f(z^{t+1})\right)\|_F\leq (7L_f+6\lambda+1) \|z^{t+1}-z^t\|_F$ and the result (2) is arrived with $\eta:=7L_f+6\lambda+1$.
\end{proof}

\begin{theorem}\label{th33}
	Suppose that $z^\star$ is a limiting point of $ \{z^t\} $ generated by MTRTC. Assume that the starting point $z^{0}$ satisfies $z^{0} \in B\left(z^{\star}, \sigma\right):=\left\{z:\left\|z-z^\star\right\|_F<\sigma\right\} \subseteq Z',$  $\theta$ and $\mu$ are defined as in Definition \ref{KL}. Suppose that  $ \rho=\frac{\min\{1,\lambda\}}{2\eta} $ with $\eta$ and $\lambda$ being from Theorem \ref{th:ca} (2)  and
	\begin{equation*}
		\sigma>\frac{\mu}{\rho(1-\theta)}\left|f\left(z^{0}\right)-f\left(z^{\star}\right)\right|^{1-\theta}+\left\|z^{0}-z^{\star}\right\|_F.	
	\end{equation*}
Then \begin{itemize}
	\item [(1).]		$z^{t} \in B\left(z^{\star}, \sigma\right), \text { for } t=0,1,2, \cdots;$
\item[(2).]	
		$\sum_{t=0}^{\infty}\left\|z^{t+1}-z^{t}\right\|_F \leq \frac{\mu}{\rho(1-\theta)}\left|f\left(z^{0}\right)-f\left(z^{\star}\right)\right|^{1-\theta};$
	\item[(3).] The entire sequence $\{z^t\}$ converges.
	\end{itemize}	
\end{theorem}
\begin{proof}
	We show (1) by induction. Clearly, (1) is true for $t=0$ by assumption. Assume that (1) holds for all $t\leq \bar{t}$, then KL property holds for such $z^t$. Now
we show that (1) is true for $t=\bar{t}+1$.
	
	Let	$\theta\in (0,1)$ and $\phi(s):=\frac{\mu}{(1-\theta)}\left(s-f(z^{\star})\right)^{1-\theta}, s\geq f(z^{\star})$.
Then,  $\phi(s)$ is concave with its derivative
	$\phi'(s)=\frac{\mu}{\left|s-f(z^{\star})\right|^{\theta}}$
	for $s>f\left(z^{\star}\right)$. Since $\phi(s)$ is concave, we have
	$$
	\phi\left(f\left(z^{t}\right)\right)-\phi\left( f\left(z^{t+1}\right)\right)  \geq \phi^{\prime}\left(f\left(z^{t}\right)\right)\left[f\left(z^t\right)-f\left(z^{t+1}\right)\right]=\frac{\mu}{\left|f\left(z^{t}\right)-f\left(z^{\star}\right)\right|^{\theta}}\left[f\left(z^t\right)-f\left(z^{t+1}\right)\right].
	$$
Combining with \eqref{kl} \eqref{fk-fk+1} and Theorem \ref{th:ca} (2), we have
	\[
	\phi\left(f\left(z^{t}\right)\right)-\phi\left( f\left(z^{t+1}\right)\right) \geq \frac{1}{\left\| \prod_\Omega \left(\nabla f(z^t)\right)\right\|_F}\left[f\left(z^t\right)-f\left(z^{t+1}\right)\right]
\geq \rho\left\|z^{t+1}-z^{t}\right\|_F.
\]
	Hence,
	\begin{equation}\label{tk}
	\begin{array}{rl}	\sum_{p=0}^{t}\left\|z^{k+1}-z^{k}\right\|_F &\leq \frac{1}{\rho} \sum_{p=0}^{t}\left[\phi\left(f\left(z^{t}\right)\right)-\phi\left( f\left(z^{t+1}\right)\right)\right]=\frac{1}{\rho}\left[  \phi\left(f\left(z^{0}\right)\right)-\phi\left(f\left(z^{t+1}\right)\right)\right] \\
&\leq \frac{1}{\rho} \phi\left(f\left(z^{0}\right)\right).	\end{array}
	\end{equation}
	This implies that
	$$
	\left\|z^{t+1}-z^\star\right\|_F \leq \sum_{p=0}^{t}\left\|z_{t+1}-z_{t}\right\|_F+\left\|z^{0}-z^{\star}\right\|_F \leq \frac{1}{\rho} \phi\left(f\left(z^{0}\right)\right)+\left\|z^{0}-z^{\star}\right\|_F<\sigma.
	$$
	Then we have $z^{t+1} \in B\left(z^{\star}, \sigma\right)$, and hence (1) is asserted.

(2). Taking $t \rightarrow \infty$ in \eqref{tk},  (2) is arrived.

(3). From (2), for any $\epsilon>0$, there exists  $K_1>0$ such that for any $t\geq K_1$ such that
$\|z^{t}-z^{t_k}\|_F\leq \sum\limits_{i=1}^{t_k-t}\|z^{t+i}-z^{t+i-1}\|_F<\frac{\epsilon}{2}$. From $\lim\limits_{k\to\infty} z^{t_k}=z^\star$, there exists $K_2>0$ such that for all $k>K_2$,  $\|z^{t_k}-z^\star\|_F<\frac{\epsilon}{2}$. Hence, for any $t\geq \max\{K_1,K_2\}$,
\[\|z^t-z^\star\|_F\leq \|z^{t}-z^{t_k}\|_F+\|z^{t_k}-z^\star\|_F\leq \epsilon,\] which indicates that $z^t\to z^\star$.
\end{proof}

\begin{theorem}
Suppose that $\left\{z^t\right\}$ is an infinite sequence generated by MTRTC with an accumulating point $z^\star$ and $\theta,\mu$ are as in Definition \ref{KL}. Then
\begin{itemize}
	\item [(a).] If $\theta \in\left(0, \frac{1}{2}\right],$ then there exist $\gamma>0$ and $c \in(0,1)$ such that
	$$
	\left\|z^t-z^{\star}\right\|_F \leq \gamma c^{t};
	$$
	\item [(b).] If $\theta \in\left(\frac{1}{2}, 1\right),$ then there exists $\gamma>0$ such that
	$$
	\left\|z^t-z^\star\right\|_F \leq \gamma t^{-\frac{1-\theta}{2 \theta-1}}.
	$$
\end{itemize}
\end{theorem}

\begin{proof}
Assume that $z^0 \in B\left(z^{\star}, \sigma\right)$. Denote that
$$
\Delta_{t}:=\sum_{p=t}^{\infty}\left\|z^p-z^{p+1}\right\|_F.
$$	
Then
\begin{equation}\label{tstar}
\left\|z^t-z^\star\right\|_F \leq \Delta_{t}.	
\end{equation}
From Theorem \ref{th33} (2), we have
\[
\Delta_{t} \leq \frac{\mu}{\rho(1-\theta)}\left|f\left(z^{0}\right)-f\left(z^{\star}\right)\right|^{1-\theta}=\frac{\mu}{\rho(1-\theta)}\left[\left|f\left(z^{0}\right)-f\left(z^{\star}\right)\right|^\theta \right]^{\frac{1-\theta}{{\theta}}}.
\]
Combining with the KL inequality, there holds
$$
\Delta_{t} \leq \frac{\mu}{\rho(1-\theta)}\left(\mu\left\|\prod\nolimits_\Omega (\nabla f(z))\right\|_F\right)^{\frac{1-\theta}{\theta}}.
$$
From Theorem \ref{th:ca} (2), the above inequality implies that
\begin{equation}\label{c1}
\Delta_{t} \leq \frac{\mu}{\rho(1-\theta)}\left(\mu\eta\left\|z^t-z^{t+1}\right\|_F\right)^{\frac{1-\theta}{\theta}}=c_{1}\left(\Delta_{t}-\Delta_{t+1}\right)^{\frac{1-\theta}{\theta}}.	
\end{equation}
where $c_{1}=\frac{\mu}{\rho(1-\theta)}(\mu\eta)^{\frac{1-\theta}{\theta}}$ is a positive constant.

(a). If $\theta \in\left(0, \frac{1}{2}\right],$ then $\frac{1-\theta}{\theta} \geq 1$. For sufficiently large $t$, it holds
$$
\Delta_{t} \leq c_{1}\left(\Delta_{t}-\Delta_{t+1}\right).
$$
Hence
$$
\Delta_{t+1} \leq \frac{c_{1}-1}{c_{1}} \Delta_{t}.
$$

Together with (\ref{tstar}), result (a) ia arrived with  $c=\frac{c_1-1}{c_1}$.

(b). For case of $\theta \in\left(\frac{1}{2}, 1\right),$  let $h(s)=s^{-\frac{\theta}{1-\theta}}$. The function $h(s)$ is monotonically decreasing on $s$. By \eqref{c1}, we have
$$
c_{1}^{-\frac{\theta}{1-\theta}} \leq h\left(\Delta_{t}\right)\left(\Delta_{t}-\Delta_{t+1}\right)=\int_{\Delta_{t+1}}^{\Delta_{t}} h\left(\Delta_{t}\right) d s \leq \int_{\Delta_{t+1}}^{\Delta_{t}} h(s) d s=-\frac{1-\theta}{2 \theta-1}\left(\Delta_{t}^{-\frac{2 \theta-1}{1-\theta}}-\Delta_{t+1}^{-\frac{2 \theta-1}{1-\theta}}\right).
$$
Since $\theta \in\left(\frac{1}{2}, 1\right),\, \nu:=-\frac{2\theta-1}{1-\theta}<0$
and
$
\Delta_{t+1}^{\nu}-\Delta_{t}^{\nu} \geq-\nu c_{1}^{-\frac{\theta}{1-\theta}}>0.
$
Thus, there is a $\hat{t}$ such that for all $t \geq 2\hat{t}$,
$$
\Delta_{t}^\nu \ge\Delta_{\hat{t}}^{\nu}-\nu c_{1}^{-\frac{\theta}{1-\theta}}(t-\hat{t})\geq-\nu c_{1}^{-\frac{\theta}{1-\theta}}(t-\hat{t})\geq-\frac{\nu}{2} c_{1}^{-\frac{\theta}{1-\theta}}t,
$$
then we have
$$
\Delta_{t} \leq \gamma t^{\frac{1}{\nu}},
$$
for a certain positive constant $\gamma=\left(-\frac{\nu}{2} c_{1}^{-\frac{\theta}{1-\theta}}\right)^{\frac{1}{\nu}}$. Then result (b) is obtained.
\end{proof}

\section{Improvement with spatio-temporal characteristics}\label{secST}
In practical applications, some characteristics are included. For example, both the video data between two adjacent frames and  the internet traffic data of two adjacent days are temporal stability features. To characterize such properties, some constraint matrices are considered. %$let $ F $ and $ G $ are the spatial constraint matrices and $ H $ is the temporal constraint matrix.

%As suggested in \cite{RZWQ12,ZZXC15}, $F, G$ and $H$ are adopted as the spatio-temporal constraint matrices to characterize such stability, where $ F $ and $ G $ are the spatial constraint matrices and $ H $ is the temporal constraint matrix, which expresses our knowledge about the data spatio-temporal properties. Let $ F $ and $ G $ are the spatial constraint matrices and $ H $ is the temporal constraint matrix.
As in \cite{RZWQ12,ZZXC15}, the temporal constraint matrix $H $  captures the temporal stability feature, i.e., the data is similar at adjacent time slots in the tensor. Let $ H = Toeplitz(0, 1, -1) $ be a Toeplitz matrix of size $ (n_{3}-1) \times n_{3} $  with
\begin{align*}
	H=\left[\begin{array}{cccc}
		1 & -1 & 0 & \cdots \\
		0 & 1 & -1 & \ddots \\
		0 & 0 & 1 & \ddots \\
		\vdots & \ddots & \ddots & \ddots
	\end{array}\right]_{(n_{3}-1) \times n_{3}}.
\end{align*}
%The  matrix $ \mathcal C_3^{(k)} $ is the principal components in time dimension of tensor, i.e., temporal lower-dimensional representation. By minimizing
Let $n_3$ be the time dimension. Then the time stability is expressed by minimizing
$$ \left\| \C \times_3 H \right\|_F^2=\sum\limits_{k = 1}^{{n_3} - 1} \left\| {\mathcal{C}_{\text{3}}^{^{(k)}} - \mathcal{C}_{\text{3}}^{^{(k + 1)}}} \right\|_F^2. $$
%The temporal constraint matrix $ H $ on the latent space of time dimension of tensor, which approximate the property of having similar values at adjacent time slots.
\par Let the spatial constraint matrices $ F $ and $ G $ capture spatial correlation feature. We choose $ F $ and $ G $ according to  the similarity between $ \mathcal C_1^{(i)} $ and $ \mathcal C_1^{(j)}~(j \ne i) $, $ \mathcal C_2^{(i)} $ and $ \mathcal C_2^{(j)}~(j \ne i)$, respectively. For each $ \mathcal C_1^{(i)} $, we perform linear regression to find a set
of weights $w_{i}(j)$ such that the linear
combination of $ \mathcal C_1^{(j)} $ is a best approximation of $ \mathcal C_1^{(i)} $, i.e., $ \mathcal{C}_{1}^{(i)} = \sum_{j\ne i} {{w_i}(j)\mathcal{C}_{1}^{{(j)}}}  $. Then we set $ F(i, i) = 1 $ and $ F (i, j) = -w_{i}(j) $.  Matrix $ G $ can be obtained similarly. %The matrices $ F $ and $ G $ express spatial similar relationship of data.
%As the mode-1 and mode-2 are lower dimensional representation of the spatial dimension, by minimizing
Let $n_1$ and $n_2$ be the spatial dimensions. Then the spatial correlation features  can be expressed by minimizing
\[\left\| \C \times_1 F \right\|_F^2 = \sum\limits_{i = 1}^{{n_1}} {\left\| {C_1^{(i)} - \sum\limits_{j \ne i} {{w_i}(j)C_1^{(j)}} } \right\|} _F^2\]
and
\[\left\| \C \times_2 G \right\|_F^2 = \sum\limits_{i = 1}^{{n_2}} {\left\| {C_2^{(i)} - \sum\limits_{j \ne i} {{w_i}(j)C_2^{(j)}} } \right\|} _F^2.\]
%The spatial constraint $ F $ and $ G $ function on the underlying latent structure of spatial dimension, which approximate spatial correlation feature.
\par Before we get such matrices $ F $ and $ G $, it is necessary to  %are obliged to
estimate an initial tensor $\mathcal C $ without missing data and outlier because these factors may destroy spatial features. To this end, we first recover the missing entries and remove outlier by  using the temporal constraint (i.e., $ H $).
For the estimated tensor $ \mathcal C $, we analyze the similarities  and linear regression to find spatial constraints (i.e., $  F,\, G $). Then the obtained $F,G$ are used
 together with matrix $ H $ in algorithm to recovery the data.

\par Based on the three  matrices $F,\, G$ and $H$, the tensor factorization model (\ref{f-opt}) can be modified as
\begin{equation}\label{mt-st}
\begin{array}{ccl}
&\mathop {\min }\limits_{{\mathcal{X}_u},{\mathcal{Y}_u},\mathcal{C}} &\sum\limits_{u = 1}^3 \frac{\alpha_u}{2}{\left\| {{\mathcal{X}_u}{ * _u}{\mathcal{Y}_u} - \mathcal{C}} \right\|_F^2}  + \frac{\beta _1}{2}\left\| {\left( {{\mathcal{X}_2}{ * _2}{\mathcal{Y}_2}} \right){ \times _1}{F}} \right\|_F^2 \\
&&+ \frac{\beta _2}{2}\left\| {\left( {{\mathcal{X}_3}{ * _3}{\mathcal{Y}_3}} \right){ \times _2}{G}} \right\|_F^2 + \frac{\beta _3}{2}\left\| {\left( {{\mathcal{X}_1}{ * _1}{\mathcal{Y}_1}} \right){ \times _3}{H}} \right\|_F^2 \\
& \mbox{\rm s.t.} &{P_\Omega }(\mathcal{C} - \mathcal{M}) = 0.
\end{array}
\end{equation}
Let $\beta_u=0$ if there is no additional characteristics on the $u$th dimension of data. Hence, model (\ref{mt-finally}) can be regarded as a special case of  model (\ref{mt-st}).

With Lemma \ref{lem:martix-to-tensor},  \eqref{mt-st}  can be rewritten as
\begin{equation}\label{mt-st-tensor}
\begin{array}{ccl}
&\mathop {\min }\limits_{{\mathcal{X}_u},{\mathcal{Y}_u},\mathcal{C}} &\sum\limits_{u = 1}^3 {\frac{\alpha_u}{2}\left\| {{\mathcal{X}_u}{*_u}{\mathcal{Y}_u} - \mathcal{C}} \right\|_F^2}  + \frac{{{\beta _1}}}{2}\left\| {\mathcal{F}{ * _2}\left( {{\mathcal{X}_2}{*_2}{\mathcal{Y}_2}} \right)} \right\|_F^2 \\
&&+ \frac{{{\beta _2}}}{2}\left\| {\left( {{\mathcal{X}_3}{*_3}{\mathcal{Y}_3}} \right){ * _3}\mathcal{G}} \right\|_F^2 + \frac{{{\beta _3}}}{2}\left\| {\left( {{\mathcal{X}_1}{*_1}{\mathcal{Y}_1}} \right){ * _1}\mathcal{H}} \right\|_F^2 \\
& \mbox{\rm s.t.} &{P_\Omega }(\mathcal{C} - \mathcal{M}) = 0.
\end{array}
\end{equation}
Similar to solve \eqref{mt-tensor}, we consider the regularized version of problem (\ref{mt-st-tensor}), which can be written as
\begin{equation}\label{mt-st-reg}\min\limits_{\C,\X_u,\Y_u} g(\C,\X_1,\X_2,\X_3,\Y_1,\Y_2,\Y_3),\quad \mbox{\rm s.t.}\quad P_\Omega(\C-\M)=0,
\end{equation}
where
\begin{equation*}
\begin{aligned}
&g(\C,\X_1,\X_2,\X_3,\Y_1,\Y_2,\Y_3)\\
=&\sum\limits_{u = 1}^3 {\frac{\alpha_u}{2}\left\| {{\mathcal{X}_u}{*_u}{\mathcal{Y}_u} - \mathcal{C}} \right\|_F^2}  + \frac{{{\beta _1}}}{2}\left\| {\mathcal{F}{ * _2}\left( {{\mathcal{X}_2}{*_2}{\mathcal{Y}_2}} \right)} \right\|_F^2+ \frac{{{\beta _2}}}{2}\left\| {\left( {{\mathcal{X}_3}{*_3}{\mathcal{Y}_3}} \right){ * _3}\mathcal{G}} \right\|_F^2 + \frac{{{\beta _3}}}{2}\left\| {\left( {{\mathcal{X}_1}{*_1}{\mathcal{Y}_1}} \right){ * _1}\mathcal{H}} \right\|_F^2 \\&+ \frac{\lambda}{2} \left( 2\beta_1\|\F*_2\X_2\|_F^2+\alpha_2\|\X_2\|_F^2+\|\Y_2\|_F^2\right)
+\frac{\lambda}{2}\left( \|\X_3\|_F^2+2\beta_2\|\Y_3*_3\mathcal{G}\|_F^2+\alpha_3\|\Y_3\|_F^2 \right)
\\&+\frac{\lambda}{2} \left(  \|\X_1\|_F^2+2\beta_3\|\Y_1\ast_1 \mathcal{H}\|_F^2+\alpha_1\|\Y_1\|_F^2\right).
\end{aligned}
\end{equation*}

Clearly, $\C^{t+1}$ can be updated by \eqref{C-comp-com}. Hence it suffices to  consider how to update $\X_u^{t+1}$ and $\Y_u^{t+1}$ for all $u\in  {\bf [3]}$. From the structure of $ \hat X_u $ and $ \hat Y_u $ in section 2, we have
\begin{equation*}
	\begin{aligned}
	&\left\| \F*_2\left( \X_2*_2\Y_2 \right) \right\|_F^2=\frac{1}{n_2}\|\bar F_2\overline{\left( \X_2*_2\Y_2 \right)}\|_F^2= \frac{1}{n_2}\|\bar F_2(\bar X_2\bar Y_2)\|_F^2\\ =& \frac{1}{n_2}\|\bar F_2(\hat X_2\hat Y_2)\|_F^2= \frac{1}{n_2}\|\bar F_2\hat X_2\hat Y_2\|_F^2=\frac{{{1 }}}{{{n_2}}}\sum\limits_{j = 1}^{{n_2}} {\left\| {\bar F_2^{(j)}\hat X_2^{(j)}\hat Y_2^{(j)}} \right\|_F^2}.	
	\end{aligned}
\end{equation*}
Similarly, we have
\begin{equation*}
\begin{aligned}
\left\| {\left( {{\mathcal{X}_3}{*_3}{\mathcal{Y}_3}} \right){ * _3}\mathcal{G}} \right\|_F^2=\frac{1}{n_3}\sum\limits_{k=1}^{n_3}{\left\| {\hat X_3^{(k)}\hat Y_3^{(k)}\bar G_3^{(k)}} \right\|_F^2},\quad \left\| {\left( {{\mathcal{X}_1}{*_1}{\mathcal{Y}_1}} \right){ * _1}\mathcal{H}} \right\|_F^2=\frac{{{1}}}{{{n_1}}}\sum\limits_{i = 1}^{{n_1}} {\left\| {\hat X_1^{(i)}\hat Y_1^{(i)}\bar H_1^{(i)}} \right\|_F^2}.	
\end{aligned}
\end{equation*}

Based on these results, we can rewrite (\ref{mt-st-reg}) as the following  matrix version
\begin{footnotesize}
$$\begin{gathered}
\mathop {\min }\limits_{\C, \hat X_u,\hat Y_u} \sum\limits_{l = 1}^{n_u}
{\frac{\alpha _{u}}{{2{n_u}}}\left\| {\hat X_u^{(l)}\hat Y_u^{(l)} - \bar C_u^{(l)}} \right\|_F^2}
+\frac{{{\beta _1}}}{{2{n_2}}}\sum\limits_{j = 1}^{{n_2}} {\left\| {\bar F_2^{(j)}\hat X_2^{(j)}\hat Y_2^{(j)}} \right\|_F^2}  \hfill \\
+\frac{\beta_2}{2n_3}\sum\limits_{k=1}^{n_3}{\left\| {\hat X_3^{(k)}\hat Y_3^{(k)}\bar G_3^{(k)}} \right\|_F^2}+\frac{{{\beta _3}}}{{2{n_1}}}\sum\limits_{i = 1}^{{n_1}} {\left\| {\hat X_1^{(i)}\hat Y_1^{(i)}\bar H_1^{(i)}} \right\|_F^2} \\
+ \lambda \left( {\frac{{{\beta _1}}}{{{n_2}}}\sum\limits_{j = 1}^{{n_2}} {\left\| {\bar F_2^{(j)}\hat X_2^{(j)}} \right\|_F^2}  + \frac{{{\alpha _2}}}{{{2n_2}}}\sum\limits_{j= 1}^{{n_2}} {\left\| {\hat X_2^{(j)}} \right\|_F^2}  + \frac{1}{{2{n_2}}}\sum\limits_{j = 1}^{{n_2}} {\left\| {\hat Y_2^{(j)}} \right\|_F^2} } \right)    \\
+\lambda\left( \frac{1}{2n_3}\sum\limits_{k=1}^{n_3}\|X_3^{(k)}\|_F^2+\frac{\beta_2}{n_3}\sum\limits_{k=1}^{n_3}\|{\hat Y}_3^{(k)}\hat G_3^{(k)}\|_F^2+\frac{\alpha_3}{2n_3}\sum\limits_{k=1}^{n_3}\|\hat Y_3^{(k)}\|_F^2 \right)\\
+ \lambda \left( {\frac{1}{{2{n_1}}}\sum\limits_{i = 1}^{{n_1}} {\left\| {\hat X_1^{(i)}} \right\|_F^2}  + \frac{{{\beta _3}}}{{{n_1}}}\sum\limits_{i = 1}^{{n_1}} {\left\| {\hat Y_1^{(i)}\bar H_1^{(i)}} \right\|_F^2}  + \frac{{{\alpha _{\text{1}}}}}{{2{n_1}}}\sum\limits_{i = 1}^{{n_1}} {\left\| {\hat Y_1^{(i)}} \right\|_F^2} } \right)
.
\end{gathered}$$
\end{footnotesize}

To update $\X_1^{t+1}$ and $\Y_1^{t+1}$,  we consider the following problem
\begin{footnotesize}
$$\begin{gathered}
\mathop {\min }\limits_{{{\hat X}_{\text{1}}},{{\hat Y}_{\text{1}}}} \sum\limits_{i = 1}^{{n_1}} {\frac{{{\alpha _{\text{1}}}}}{{2{n_1}}}\left\| {\hat X_1^{(i)}\hat Y_{\text{1}}^{(i)} - \bar C_{\text{1}}^{(i)}} \right\|_F^2}  + \frac{{{\beta _3}}}{{2{n_1}}}\sum\limits_{i = 1}^{{n_1}} {\left\| {\hat X_1^{(i)}\hat Y_1^{(i)}\bar H_1^{(i)}} \right\|_F^2}  \hfill \\
+ \lambda \left( {\frac{1}{{2{n_1}}}\sum\limits_{i = 1}^{{n_1}} {\left\| {\hat X_1^{(i)}} \right\|_F^2}  + \frac{{{\beta _3}}}{{{n_1}}}\sum\limits_{i = 1}^{{n_1}} {\left\| {\hat Y_1^{(i)}\bar H_1^{(i)}} \right\|_F^2}  + \frac{{{\alpha _{\text{1}}}}}{{2{n_1}}}\sum\limits_{i = 1}^{{n_1}} {\left\| {\hat Y_1^{(i)}} \right\|_F^2} } \right).
\end{gathered}$$
\end{footnotesize}
For any $i\in[\bf n_1]$, $\hat X_1^{(i,t+1)}$ and $\hat Y_1^{(i,t+1)}$ are updated by
\begin{equation}\label{mt-st-x1}
\begin{aligned}
\hat X_1^{(i,t+1)}
=&\left(\lambda\hat X_1^{(i,t)}+\alpha _1\bar C_1^{(i,t+1)}\left( \hat Y_1^{(i,t)} \right)^*\right)
\bigg[\alpha _1 \hat Y_1^{(i,t)} \left( \hat Y_1^{(i,t)} \right)^* +\\
 &{\beta _3}\left(\hat Y_1^{(i,t)}\bar H_1^{(i)} \right) \left( {\hat Y_1^{(i,t)}\bar H_1^{(i)}} \right)^*+2\lambda I \bigg]^{ - 1} \hfill
\end{aligned}
\end{equation}
and
\begin{equation}\label{mt-st-y1}
\begin{aligned}
\hat Y_1^{(i,t+1)}=& {\alpha _1}{\left[ {{{\left( {\hat X_1^{(i,t+1)}} \right)}^*}\hat X_1^{(i,t+1)}}+2\lambda I \right]^{ - 1}}\left( \lambda\hat Y_1^{(i,t)}+{\left( {\hat X_1^{(i,t+1)}} \right)^*}\bar C_1^{(i,t+1)}\right)\\
 &\left[ \alpha _1 I + {\beta _3}\bar H_1^{(i)}  \left( \bar H_1^{(i)} \right)^* \right]^{ - 1}.
\end{aligned}
\end{equation}
To update  $\X_2^{t+1}$ and $\Y_2^{t+1}$, we consider the following problem
\begin{footnotesize}
$$\begin{gathered}
\mathop {\min }\limits_{{{\hat X}_{\text{2}}},{{\hat Y}_{\text{2}}}} \sum\limits_{j = 1}^{{n_2}} {\frac{{{\alpha _{\text{2}}}}}{{2{n_2}}}\left\| {\hat X_2^{(j)}\hat Y_{\text{2}}^{(j)} - \bar C_{\text{2}}^{(j)}} \right\|_F^2}  + \frac{{{\beta _1}}}{{2{n_2}}}\sum\limits_{j = 1}^{{n_2}} {\left\| {\bar F_2^{(j)}\hat X_2^{(j)}\hat Y_2^{(j)}} \right\|_F^2}  \hfill \\
+ \lambda \left( {\frac{{{\beta _1}}}{{{n_2}}}\sum\limits_{j = 1}^{{n_2}} {\left\| {\bar F_2^{(j)}\hat X_2^{(j)}} \right\|_F^2}  + \frac{{{\alpha _{\text{2}}}}}{{2{n_2}}}\sum\limits_{j = 1}^{{n_2}} {\left\| {\hat X_2^{(j)}} \right\|_F^2}  + \frac{1}{{2{n_2}}}\sum\limits_{j = 1}^{{n_2}} {\left\| {\hat Y_2^{(j)}} \right\|_F^2} } \right). \hfill \\
\end{gathered} $$
\end{footnotesize}
Therefore, for any $j\in[\bf n_2]$, $\hat X_2^{(j,t+1)}$ and $\hat Y_2^{(j,t+1)}$ are updated by
\begin{equation}\label{mt-st-x2}
\begin{aligned}	
\hat X_2^{(j,t+1)} =& {\alpha _2}{\left[ {{\alpha _2}I + {\beta _1}{{\left( {\bar F_2^{(j)}} \right)}^*}\bar F_2^{(j)}} \right]^{ - 1}}\left(\lambda X_2^{(j,t)}+\bar C_2^{(j,t+1)}{\left( {\hat Y_2^{(j,t)}} \right)^*}\right)\\
&{\left[ {\hat Y_2^{(j,t)}{{\left( {\hat Y_2^{(j,t)}} \right)}^*}}+2\lambda I \right]^{ - 1}}
\end{aligned}
\end{equation}
and
\begin{equation}\label{mt-st-y2}
\begin{aligned}	
\hat Y_2^{(j,t+1)}=& {\left[ {{\alpha _2}{{\left( {\hat X_2^{(j,t+1)}} \right)}^*}\hat X_2^{(j,t+1)} + {\beta _1}{{\left( {\bar F_2^{(j)}\hat X_2^{(j,t+1)}} \right)}^*}\bar F_2^{(j)}\hat X_2^{(j,t+1)}}+2\lambda I \right]^{ - 1}}\\
&\left( \lambda\hat Y_2^{(j,t)}+{\alpha _2}{\left( {\hat X_2^{(j,t+1)}} \right)^*}\bar C_2^{(j,t+1)}\right) .
\end{aligned}
\end{equation}
To update $\X_3^{t+1}$ and $\Y_3^{t+1}$, we consider the following problem
\begin{footnotesize}
$$\begin{gathered}
\mathop {\min }\limits_{{{\hat X}_{\text{3}}},{{\hat Y}_{\text{3}}}} \sum\limits_{k = 1}^{n_3}
{\frac{\alpha _{3}}{{2{n_3}}}\left\| {\hat X_3^{(k)}\hat Y_3^{(k)} - \bar C_3^{(k)}} \right\|_F^2}
+\frac{\beta_2}{2n_3}\sum\limits_{k=1}^{n_3}{\left\| {\hat X_3^{(k)}\hat Y_3^{(k)}\bar G_3^{(k)}} \right\|_F^2} \\
+\lambda\left( \frac{1}{2n_3}\sum\limits_{k=1}^{n_3}\|X_3^{(k)}\|_F^2+\frac{\beta_2}{n_3}\sum\limits_{k=1}^{n_3}\|{\hat Y}_3^{(k)}\hat G_3^{(k)}\|_F^2+\frac{\alpha_3}{2n_3}\sum\limits_{k=1}^{n_3}\|\hat Y_3^{(k)}\|_F^2 \right).
\end{gathered}$$
\end{footnotesize}

Then $\hat X_3^{(k,t+1)}$ and $\hat Y_3^{(k,t+1)}$ for any $k\in[\bf n_3]$ are updated by
\begin{equation}\label{mt-st-x3}
\begin{aligned}	
\hat X_3^{(k,t+1)} =&\left(\lambda\hat X_3^{(k,t)}+  {\alpha _3}\bar C_3^{(k,t+1)}{\left( {\hat Y_3^{(k,t)}} \right)^*}\right) \bigg[ {\alpha _3}\hat Y_3^{(k,t)}{{\left( {\hat Y_3^{(k,t)}} \right)}^*}+\\
 &{\beta _2}\left( {\hat Y_3^{(k,t)}\bar G_3^{(k)}} \right){{\left( {\hat Y_3^{(k,t)}\bar G_3^{(k)}} \right)}^*}+2\lambda I \bigg]^{ - 1}
\end{aligned}
\end{equation}
and
\begin{equation}\label{mt-st-y3}
\begin{aligned}
\hat Y_3^{(k,t+1)} =& {\alpha _3}{\left[ {{{\left( {\hat X_3^{(k,t+1)}} \right)}^*}\hat X_3^{(k,t+1)}}+2\lambda I \right]^{ - 1}}\left( \lambda\hat Y_3^{(k,t)}+{\left( {\hat X_3^{(k,t+1)}} \right)^*}\bar C_3^{(k,t+1)}\right)\cdot \\
&{\left[ {{\alpha _3}I + {\beta _2}{\bar G_3^{(k)}{\left( {\bar G_3^{(k)}} \right)}^*}} \right]^{ - 1}}. \hfill \\
\end{aligned}
\end{equation}
Based on above analysis,
the alternating minimization method can be outlined as Algorithm 4.1, denoted by ST-MTRTC for convenience.
\begin{table}[htbp]
	\centering
	% \begin{center}
	\begin{tabular}{l}
		\toprule
		\toprule
		{\bfseries Algorithm 4.1} Spatio-Temporal Multi-Tubal Rank Tensor Completion (ST-MTRTC)       \\
		\midrule
		{\bfseries Input:} The tensor data $\M \in {{\mathbb C}^{{n_1} \times {n_2} \times {n_3}}}$, $\H\in \mathbb{R}^{n_1\times n_3\times n_3}$, the observed set
		$\Omega $, the initialized \\ \qquad \quad \,\,\,\! rank  $ R^0$, parameters $\lambda$, $ \varepsilon $ and $\alpha_u$, $u\in [\bf 3]$.                                  \\
		{\bfseries Initialization: } $\hat X_u^0,\,\hat Y_u^0,\, \,u\in[\bf 3] $.                                                                               \\
		{\bfseries While not converge do}                \\
		\qquad  $ \bm{1.} $ Fix $\hat X_u^t$ and $ \hat Y_u^t $, compute $\mathcal C^{t+1} $ by \eqref{C-comp-com}.             \\
		\qquad  $ \bm{2.} $ Compute  $ \mathcal F$ and $ \mathcal G $ based on $\C^1$.                                \\
		\qquad  $ \bm{3.} $ Compute $\hat X_u^{t+1}$ by \eqref{mt-st-x1},\eqref{mt-st-x2} and \eqref{mt-st-x3} by fixing $\hat Y_u^t$ and $ \C^{t+1} $.     \\
		\qquad  $ \bm{4.} $ Obtain $\hat Y_u^{t+1}$ by \eqref{mt-st-y1},\eqref{mt-st-y2} and \eqref{mt-st-y3} based on  $\hat X_u^{t+1}$ and $ \C^{t+1} $.     \\
		\qquad  $ \bm{5.} $ Adopt the rank decreasing scheme to adjust $ rank_{mt}(\A) $ and  the sizes  of\\\qquad \quad \,\!
  $\hat X_u^{t+1}$ and $ \hat Y_u^{t+1} $. \\
		%\qquad  $ \bm{6.} $ Adjust $ \alpha_{u}^{t+1} $ by \eqref{alpha}.\\
		\qquad  $ \bm{6.} $ Check the stop criterion ${\left\| {\C_\Omega ^{t + 1} - {\M_\Omega }} \right\|_F}/{\left\| {{\M_\Omega }} \right\|_F} < \varepsilon $.  \\
		\qquad  $ \bm{7.} $ $t \leftarrow t + 1$.            \\
		{\bfseries end while}                             \\
		{\bfseries Output:}  $ \C^{t + 1}  $.      \\		
		\bottomrule
		\bottomrule
	\end{tabular}
	% \end{center}
\end{table}

The convergence is similar to that of Algorithm MTRTC and hence we omit it here.
%\begin{theorem}
%	Suppose that $ \left\{\left(\C^t, \X_1^t,\X_2^t,\X_3^t,\Y_1^t,\Y_2^t,\Y_3^t\right)  \right\} $ is an infinite sequence generated by ST-MTRTC. Then $ %\left\{\left(\C^t, \X_1^t,\X_2^t,\X_3^t,\Y_1^t,\Y_2^t,\Y_3^t\right) \right\} $ is bounded and any accumulation point of \\
%$ \left\{\left(\C^t, \X_1^t,\X_2^t,\X_3^t,\Y_1^t,\Y_2^t,\Y_3^t \right)\right\} $ is a KKT point of \eqref{mt-st-reg}.
%\end{theorem}
%The proof is similar to that of Theorem \ref{th:ca}, and hence we omit it here.

\section{Numerical Experiments}

In this section, we report some numerical results of our proposed algorithms  MTRTC and ST-MTRTC  to show the validity. We adopt the relative error and the peak signal-to-noise ratio (PSNR) as evaluation metrics, which are defined by
$$\operatorname{RSE}:=\frac{\left\|\mathcal{\hat C}-\mathcal{M}\right\|_{F}}{\|\mathcal{M}\|_{F}},\quad
\operatorname{PSNR}:=10 \log _{10}\left(\frac{n_{1} n_{2} n_{3}\|\mathcal{M}\|_{\infty}^{2}}{\|\hat{\mathcal{C}}-\mathcal{M}\|_{F}^{2}}\right),
$$
where $ \M $ and $ \hat{\C} $ are the observed tensor and estimated tensor, respectively. The parameter $\lambda$ is set as $0.1 $ in both MTRTC and ST-MTRTC. We conduct extensive experiments to evaluate our methods, and then compare the results with those by some other existing methods, including TMac \cite{XHYS13} and TCTF \cite{ZLLZ18}. All the methods are implemented on the platform of Windows 10 and Matlab (R2014a) with an Intel(R) Core(TM) i7-7700 CPU at 3.60GHz and 8 GB RAM.

\subsection{Numerical Simulation}

\par In this subsection, we test MTRTC on synthetic data to evaluate the efficiency by comparing MTRTC with TCTF. In experiments, the maximum iteration number is set to be 300 and the termination precision $ \varepsilon $ is set to be 1e-5.

The tested tensor $\mathcal M \in {\mathbb R^{{100} \times {100} \times {100}}}$ is constructed in the following way. Use Matlab command $randn(r_1,r_2,r_3)$ to generate tensor $\mathcal{B}\in{\mathbb R^{{r_1} \times {r_2} \times {r_3}}}$. Generate matrices $U^i\in {\mathbb R^{100\times r_i}}$ with $i\in  {\bf [3]}$ such that the multi-rank of tensor  $\mathcal{M}:=\mathcal{B}\times_1 U^1\times_2 U^2\times_3 U^3$ is $(r_1,r_2,r_3)$. Select $ pn_{1}n_{2}n_{3} $ positions of $\M $ uniformly to construct $\Omega$, where $p$ is the sampling ratio. If $ RSE < 1e$-3, $ \mathcal{\hat C} $ is regarded as a successful recovery to $ \M $. For fairness, we run these procedures for 30 times.

First, we test TCTF and MTRTC for the problems of different sample rates. Let $r_1=r_2=r_3=20$, the initial rank $(r_{u}^{l})^0=20,\, u \in  {\bf [3]},\, l \in {\bf [n_u]}$ in MTRTC and the initial rank $ (20,20,20) $ in TCTF. We set sampling ratio $ p $ varying from $ 0.1 $ to $ 0.9 $ with increment $0.1$. The numerical results are reported in Figure \ref{fig:NS_p_RSE} (a).

\begin{figure}[htbp]
	\centering
	\begin{subfigure}[t]{0.49\linewidth}
		\centering
		\includegraphics[width=3in]{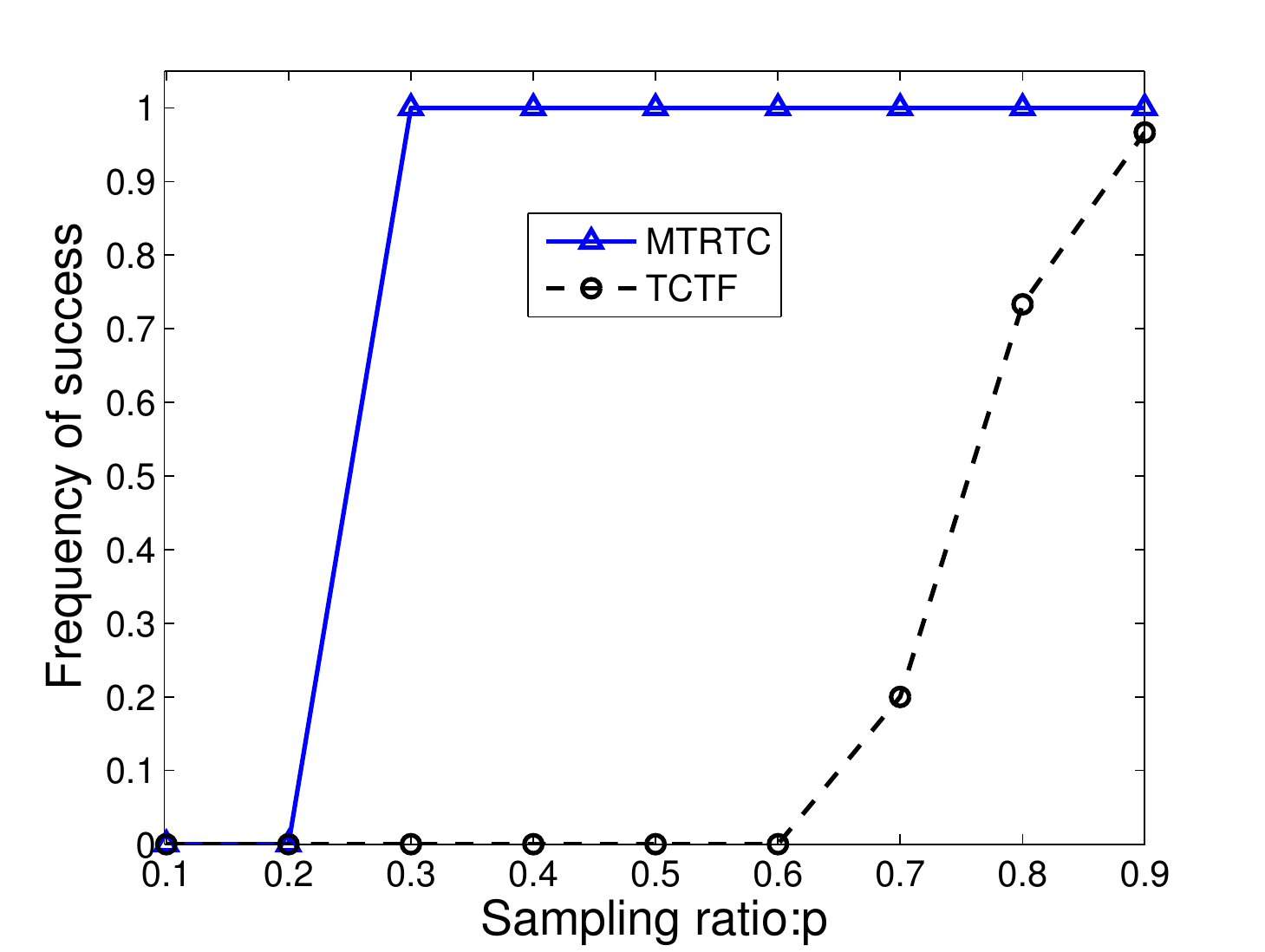}
		\caption{Comparison on frequency of success of different sampling ratios}
	\end{subfigure}
	\begin{subfigure}[t]{0.49\linewidth}
		\centering
		\includegraphics[width=3in]{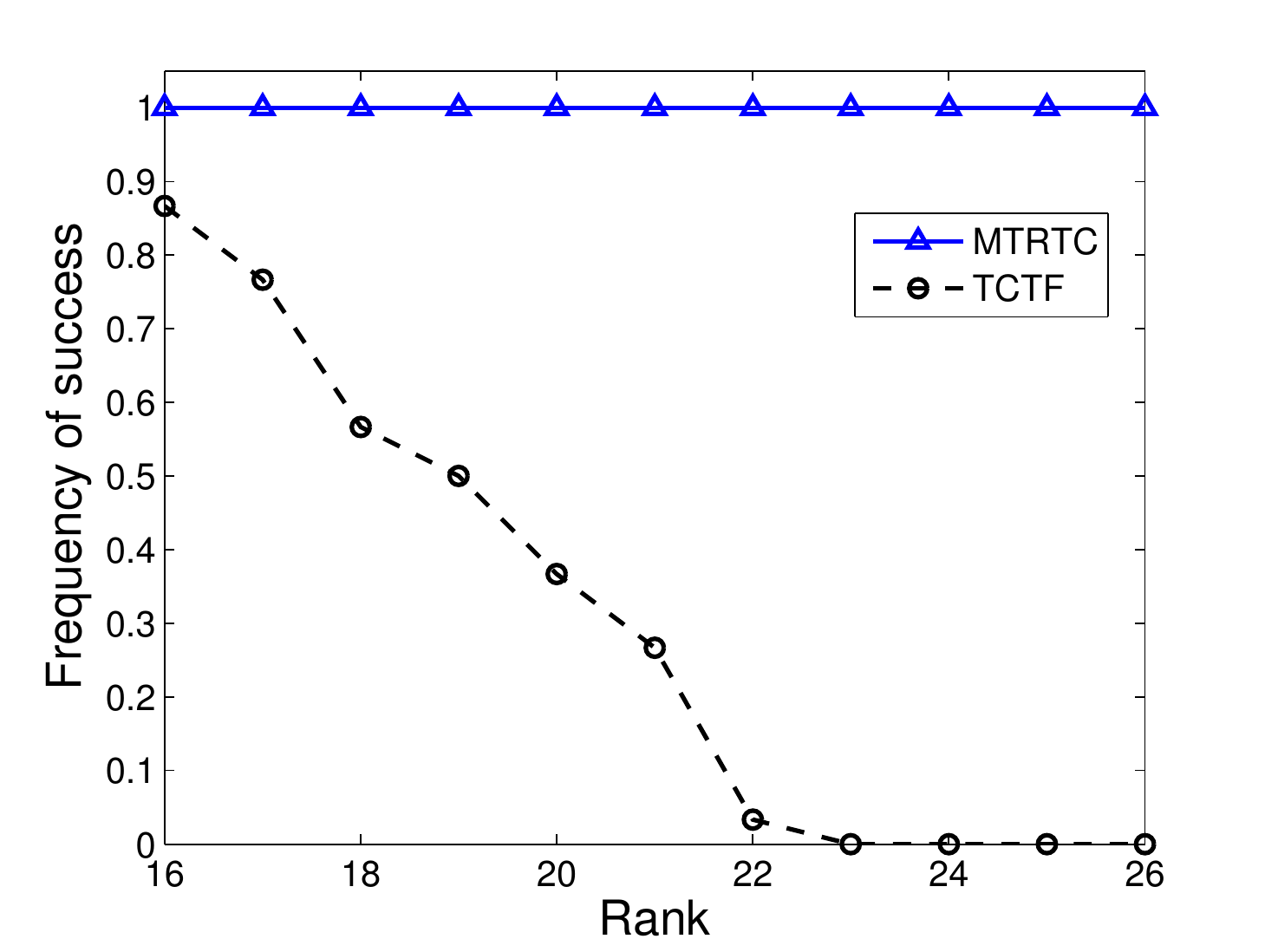}
		\caption{Comparison on frequency of success of different ranks}
	\end{subfigure}
	\centering
	\caption{Comparison on frequency of success obtained by MTRTC and TCTF}
	\label{fig:NS_p_RSE}
\end{figure}

In Figure \ref{fig:NS_p_RSE} (a), the frequency of success of these two methods are reported. Our proposed method MTRTC performs much better than TCTF. We find that the lower the sampling ratio $ p $, the more difficult it is to recover the tensor successfully. Form Figure \ref{fig:NS_p_RSE} (a), it is clear that our method MTRTC can complete the tensor successfully  when the sample rate is bigger than 0.2; while the tensor can not be completed by TCTF when the sample rate is less than 0.6.

On the other hand, we test TCTF and MTRTC for the tested tensors of sampling ratio $p=0.7$ with different ranks. We set the rank $r=r_1=r_2=r_3$ varying from $ 16 $ to $ 26 $ with increment $ 1 $. We set the initialized rank $(r_{u}^{l})^0=r, u \in  {\bf [3]}, l \in {\bf [n_u]}$ in MTRTC and the initial rank $ (r,r,r) $ in TCTF. The frequences of success are reported in Figure \ref{fig:NS_p_RSE} (b).

Figure \ref{fig:NS_p_RSE} (b) indicates that tensor can be completed by MTRTC  for all estimated rank from $(16,16,16)$. With the increase of rank, the success rate of TCTF in restoring tensors gradually decreases. Moreover, TCTF cannot successfully restore tensors when the rank is bigger than 22.

From accuracy and efficiency, we know that MTRTC performs better than TCTF  for all sizes of the sampling ratios and  tensor ranks.

\subsection{Image Simulation}
In this subsection, we apply MTRTC to color image inpainting.
Note that  color images can be expressed as third order tensors. When the tensor data is of low rank, or numerical low rank, the image inpainting problem can be modeled as a tensor completion problem. We use the Berkeley Segmentation database \cite{MFTM01} to evaluate our method for image inpainting. It has a total of $ 200 $ color images, of size $ 321 \times 481 \times 3 $. In these experiments, we compare our results with those from the state-of-the-art methods (TMac, TCTF).

In the test, all $ 200 $ images are chosen from the Berkeley Segmentation database. For each chosen image, we randomly sample by sampling ratio $p=0.7$.  We set the initial multi-tubal rank $ (r_{u}^{l})^0=2, u \in  {\bf [2]}, \,(r_{3}^{l})^0=30, l \in {\bf [n_u]} $ in MTRTC,  the initial tubal rank $ (30,30,30) $ in TCTF and the initial Tucker rank $ (30,30,3) $ in TMac. In experiments, the maximum iteration number is set to be 300 and the termination precision $ \varepsilon $ is set to be 1e-5.

\begin{figure}[htbp]
 \centering
 \begin{subfigure}[t]{1\linewidth}
  \centering
  \includegraphics[width=6in]{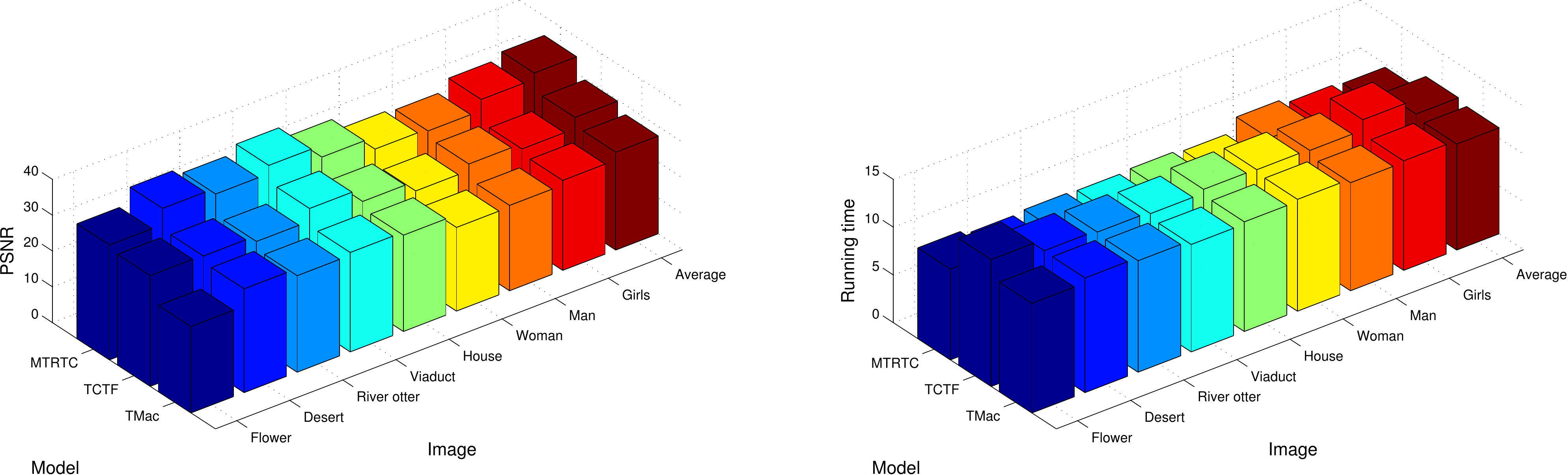}
 \end{subfigure}
\iffalse
 \begin{subfigure}[t]{0.49\linewidth}
  \centering
  \includegraphics[width=3in]{picture/picture_bar3_rse}
 \end{subfigure}
\fi
 \centering
 \caption{ Comparison on the PSNR  and the running time by MTRTC, TCTF and TMac}
 \label{fig:image_bar}
\end{figure}

\begin{table*}
\caption{Comparison of the PSNR, the RSE  and the running time by MTRTC, TCTF and TMac}\label{tab:image}
\centering
		\begin{tabular}{|c|ccc|ccc|ccc|}
		\hline
			{\multirow{2}{*}{\diagbox{Image}{Method}}}&\multicolumn{3}{|c|}{MTRTC} &\multicolumn{3}{c|}{TCTF}&\multicolumn{3}{c|}{TMac}
			\\\cline{2-10}
&  PSNR      & RSE     &time
			&PSNR         & RSE         &time
			&PSNR         & RSE         &time
			\\   \hline
Flower & \textbf{32.07 } & \textbf{0.079 } & 9.51  & 30.95  & 0.090  & 13.28  & 24.09  & 0.199  & 11.47  \\
Desert & \textbf{36.61 } & \textbf{0.031 } & 9.02  & 30.72  & 0.060  & 12.08  & 28.96  & 0.074  & 12.00  \\
River otter & \textbf{34.95 } & \textbf{0.052 } & 9.23  & 29.08  & 0.102  & 11.89  & 26.95  & 0.131  & 11.67  \\
Viaduct & \textbf{37.23 } & \textbf{0.031 } & 9.11  & 32.63  & 0.053  & 11.69  & 27.74  & 0.093  & 11.24  \\
House & \textbf{33.88 } & \textbf{0.038 } & 9.26  & 28.91  & 0.067  & 12.11  & 26.89  & 0.084  & 11.48  \\
Man   & \textbf{30.41 } & \textbf{0.056 } & 9.00  & 26.15  & 0.092  & 12.03  & 23.37  & 0.127  & 11.70  \\
Human & \textbf{29.80 } & \textbf{0.064 } & 10.41  & 27.97  & 0.079  & 11.91  & 23.84  & 0.127  & 11.31  \\
Girl  & \textbf{32.92 } & \textbf{0.049 } & 9.88  & 26.49  & 0.103  & 12.99  & 25.12  & 0.121  & 11.46  \\
Average & \textbf{34.60 } & \textbf{0.043 } & 9.65  & 29.59  & 0.079  & 11.46  & 27.31  & 0.101  & 11.09  \\
			\hline
		\end{tabular}

\end{table*}

\begin{figure}[htbp]
	\centering
	\begin{subfigure}[b]{1\linewidth}
%		\subfloat[]{
			\begin{subfigure}[b]{0.19\linewidth}
			\centering
			\includegraphics[width=\linewidth]{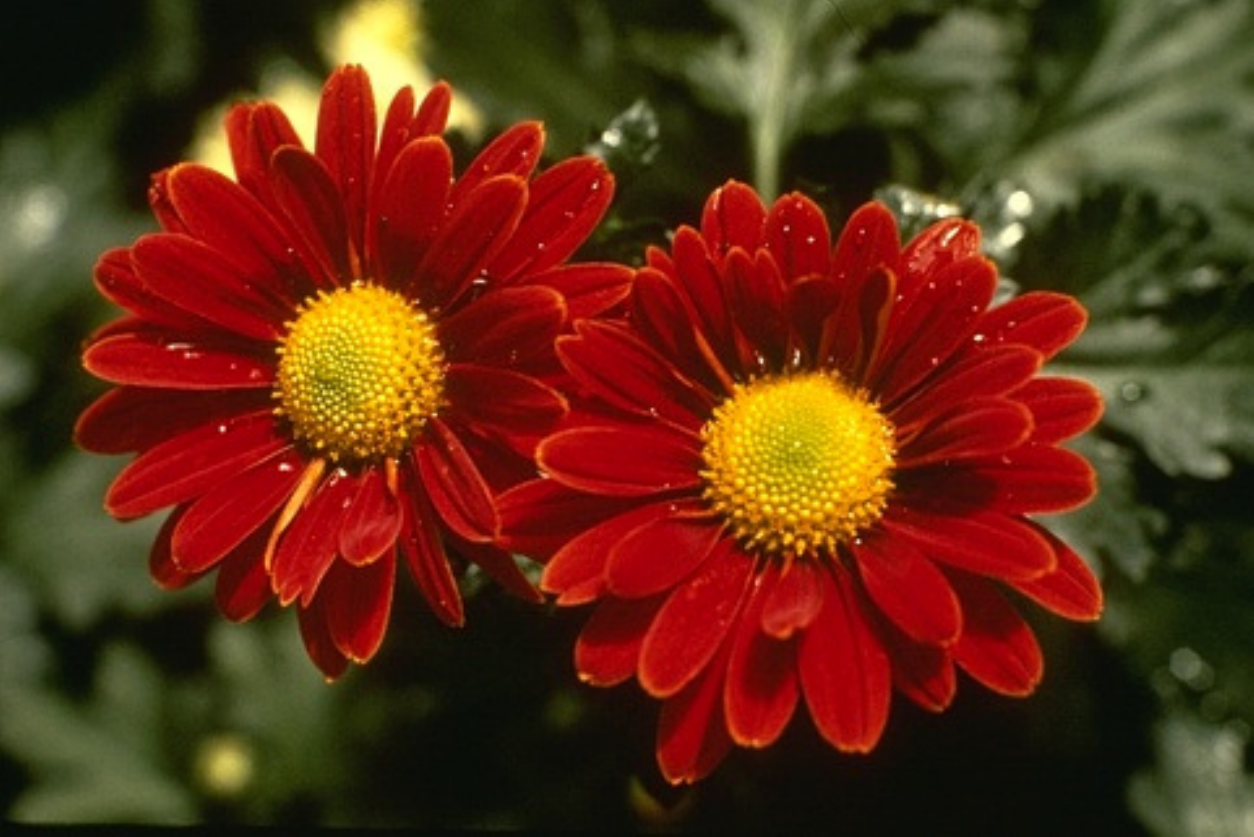}\vspace{0pt}
			\includegraphics[width=\linewidth]{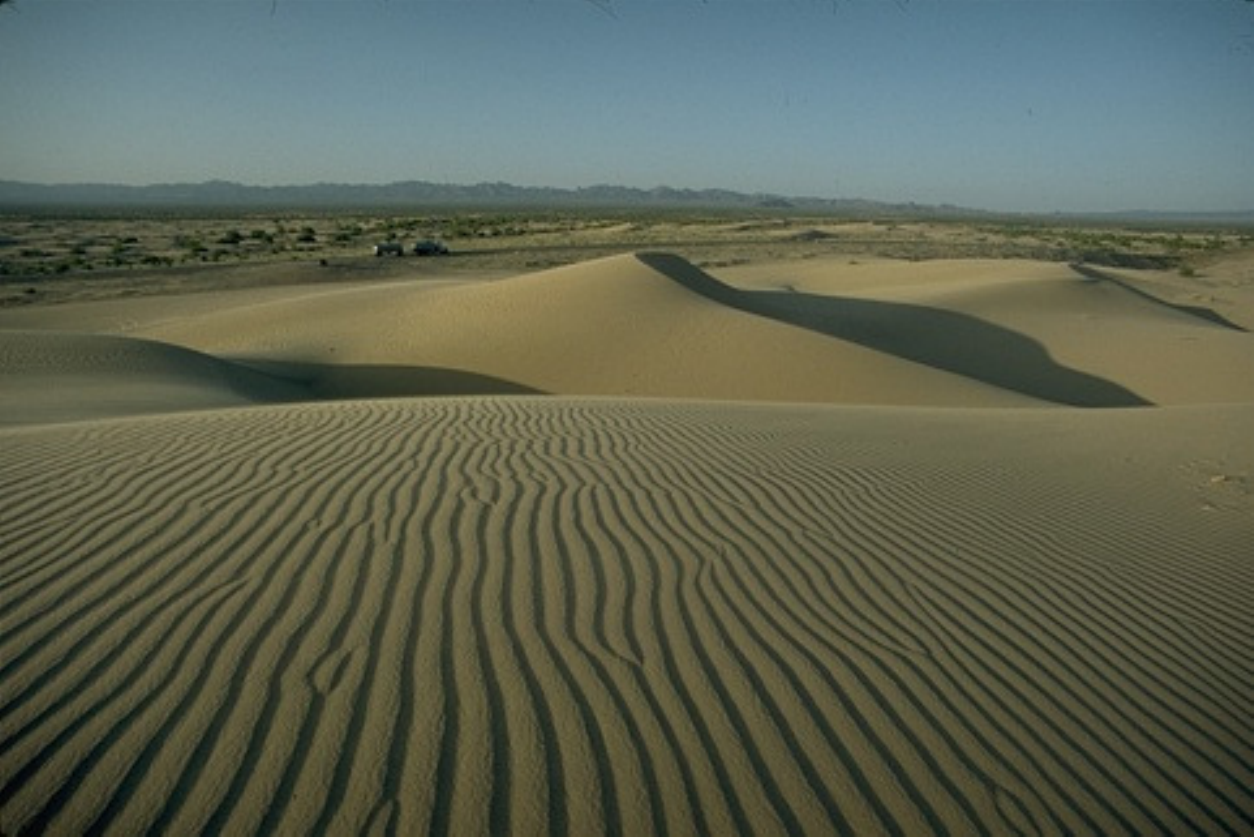}\vspace{0pt}
			\includegraphics[width=\linewidth]{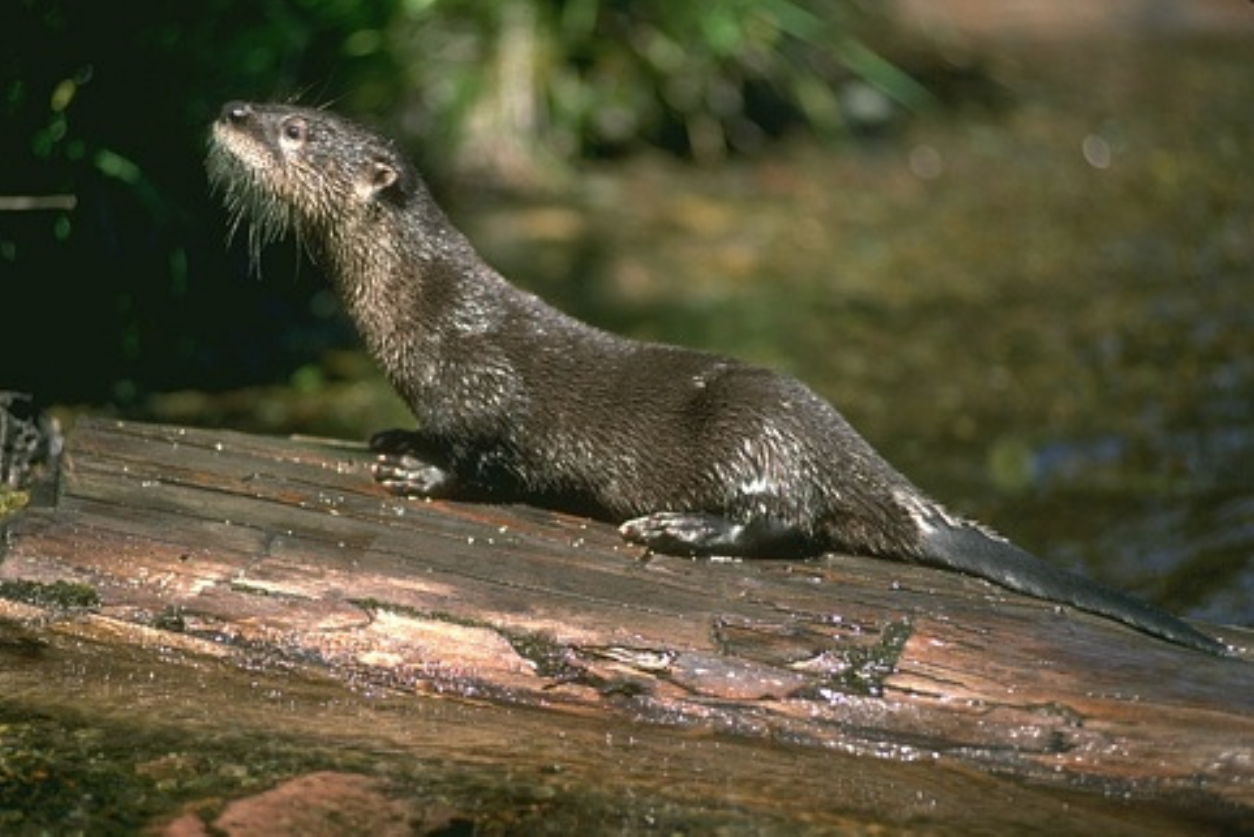}\vspace{0pt}
			\includegraphics[width=\linewidth]{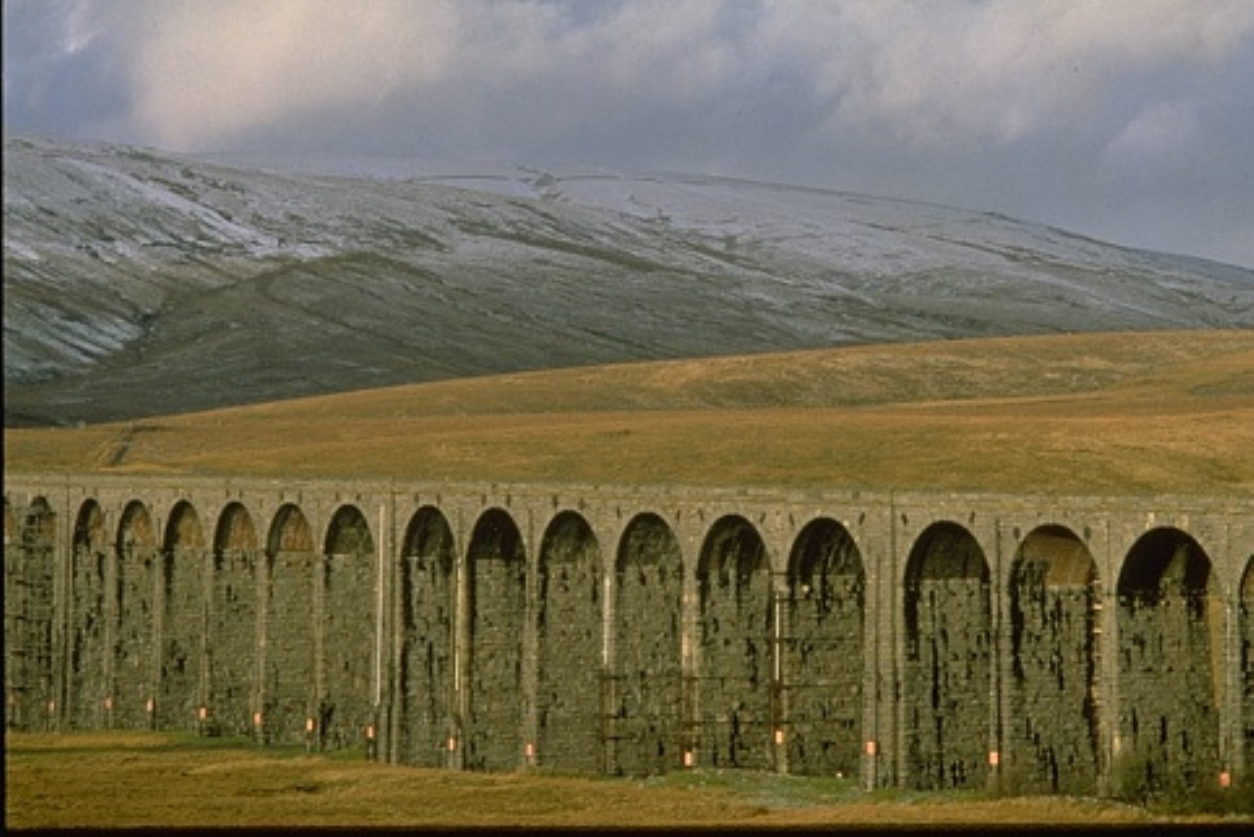}\vspace{0pt}
			\includegraphics[width=\linewidth]{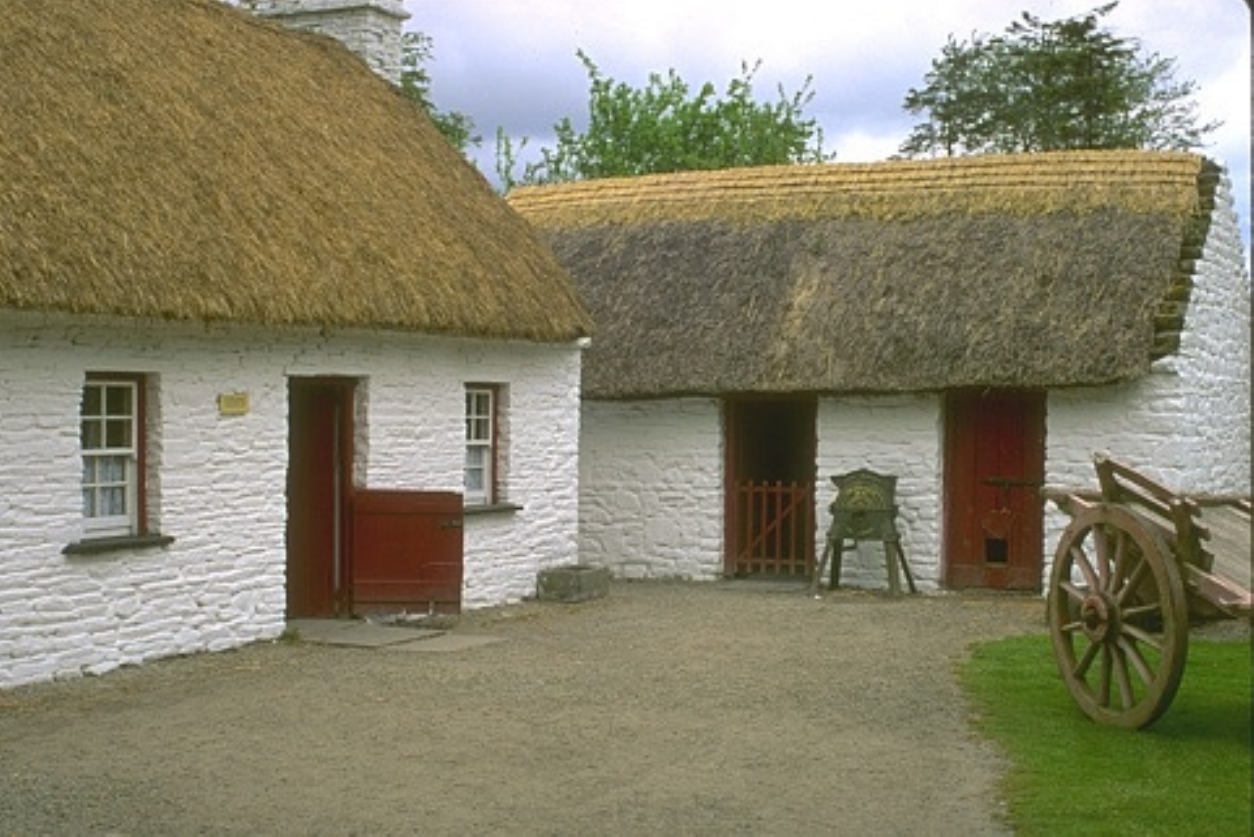}\vspace{0pt}
			\includegraphics[width=\linewidth]{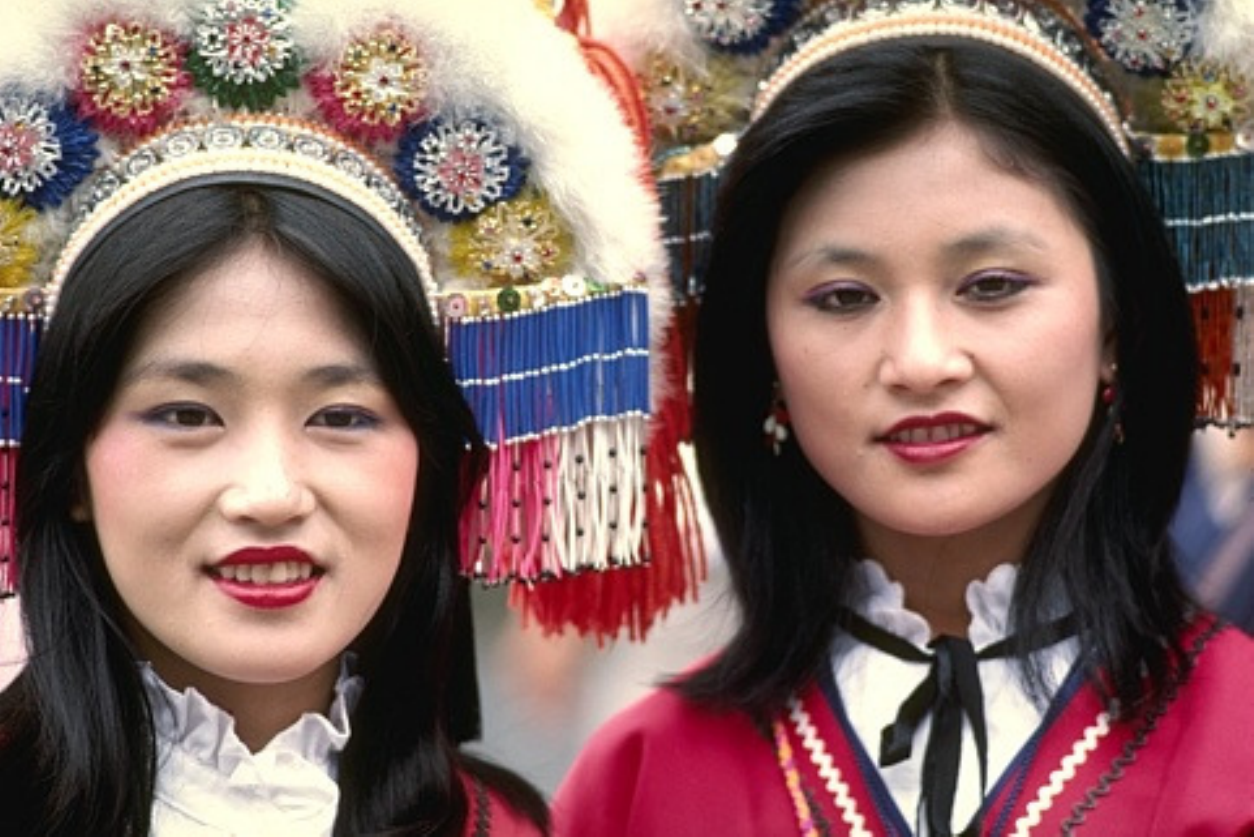}\vspace{0pt}
			\includegraphics[width=\linewidth]{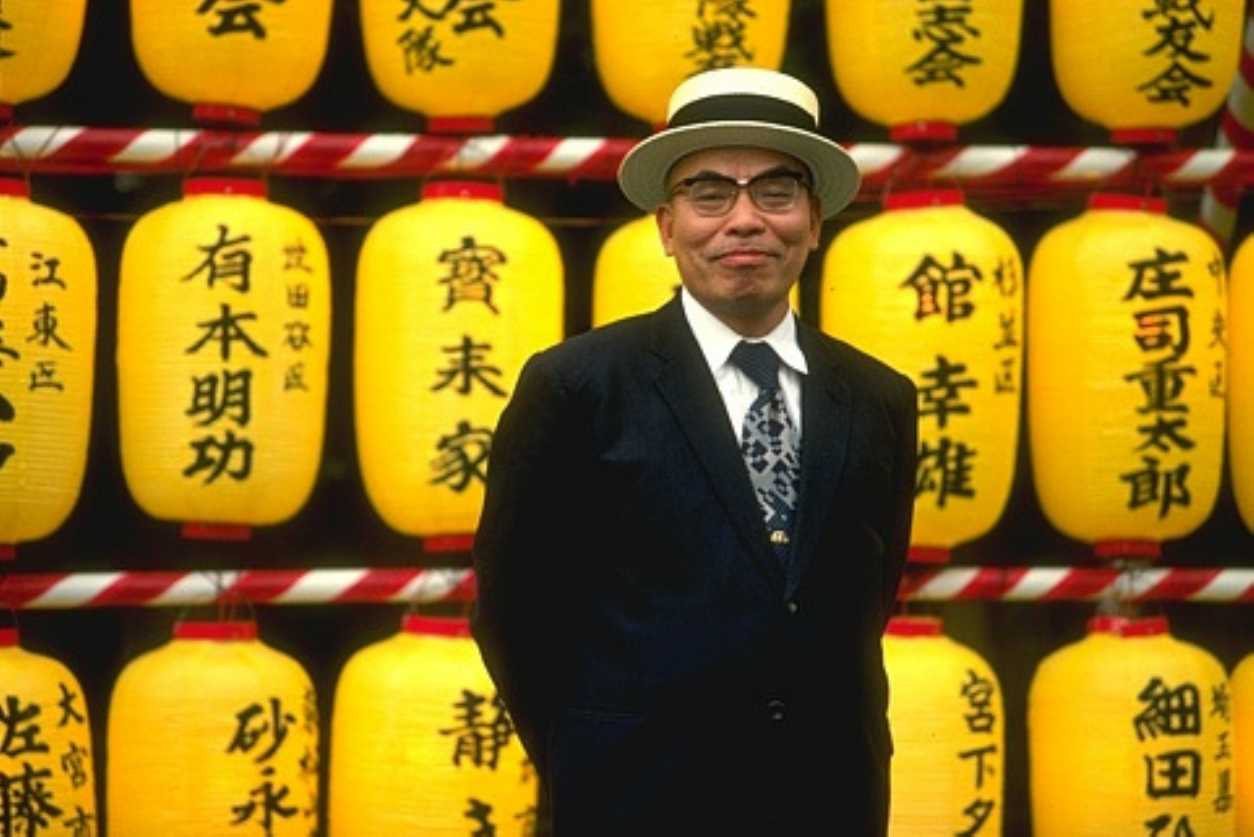}\vspace{0pt}
			\includegraphics[width=\linewidth]{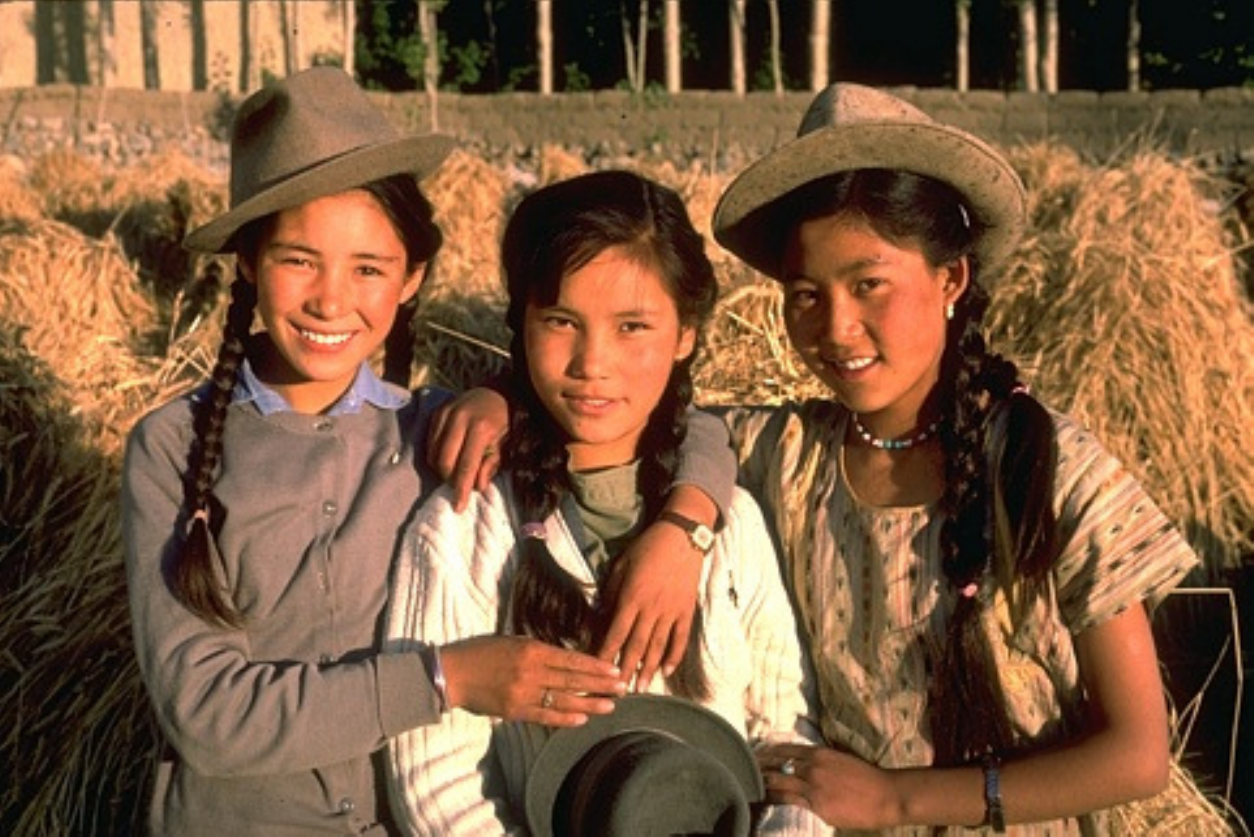}
			\caption{Original}
		    \end{subfigure}
%	\subfloat[Observation]{    	
	\begin{subfigure}[b]{0.19\linewidth}
		\centering
		\includegraphics[width=\linewidth]{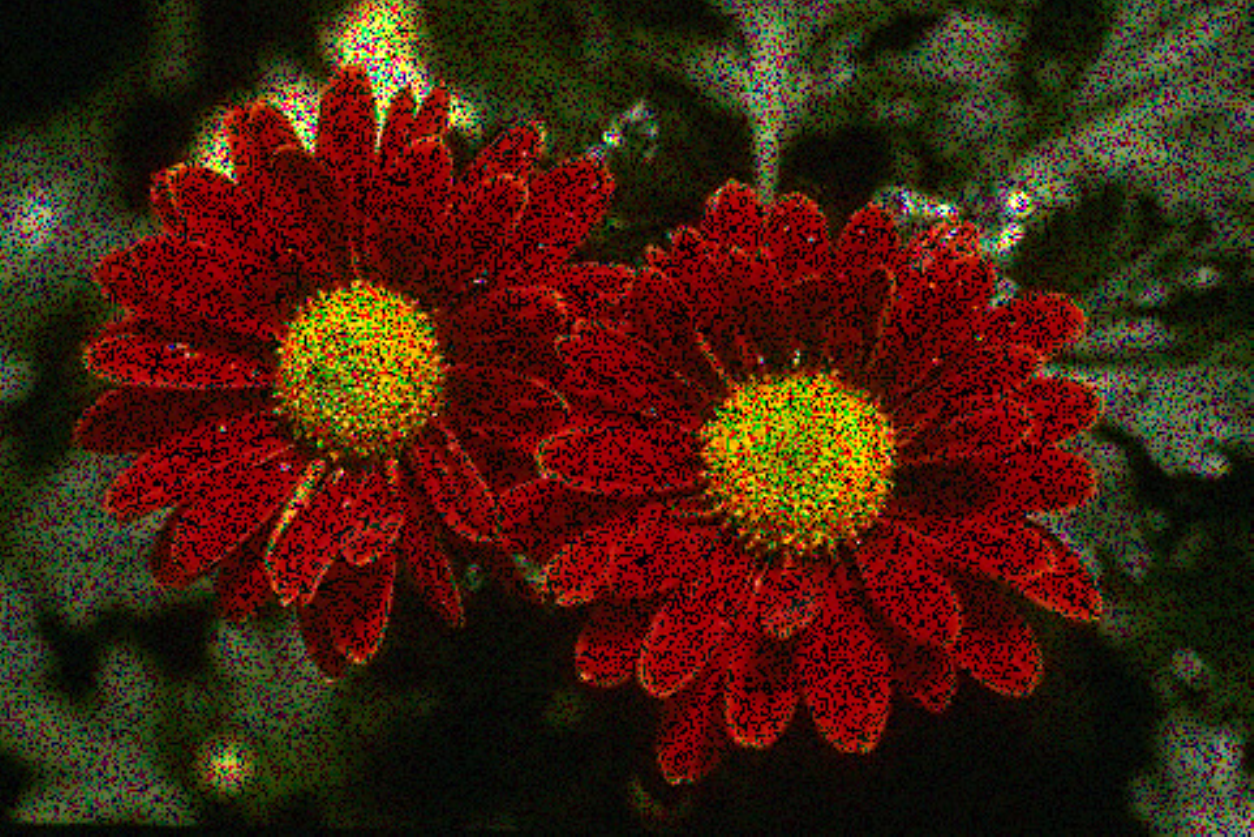}\vspace{0pt}
		\includegraphics[width=\linewidth]{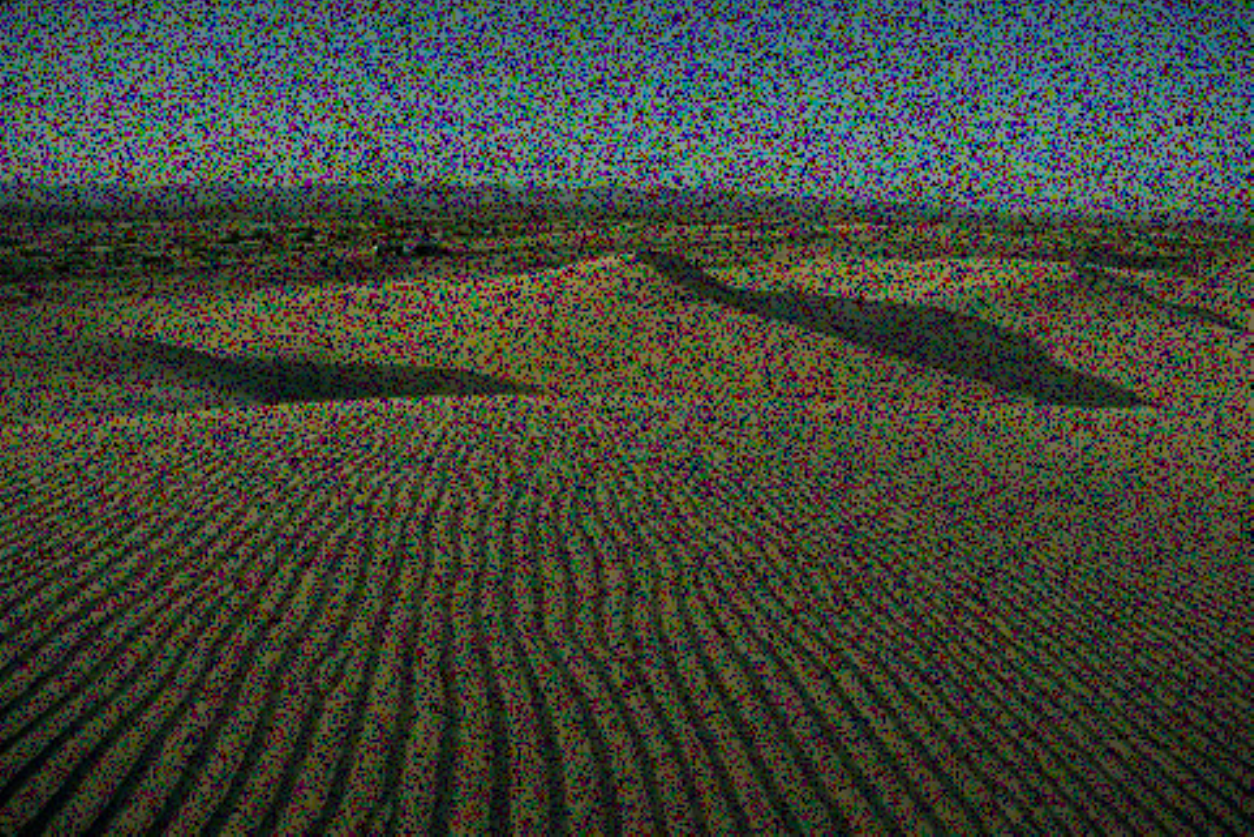}\vspace{0pt}
		\includegraphics[width=\linewidth]{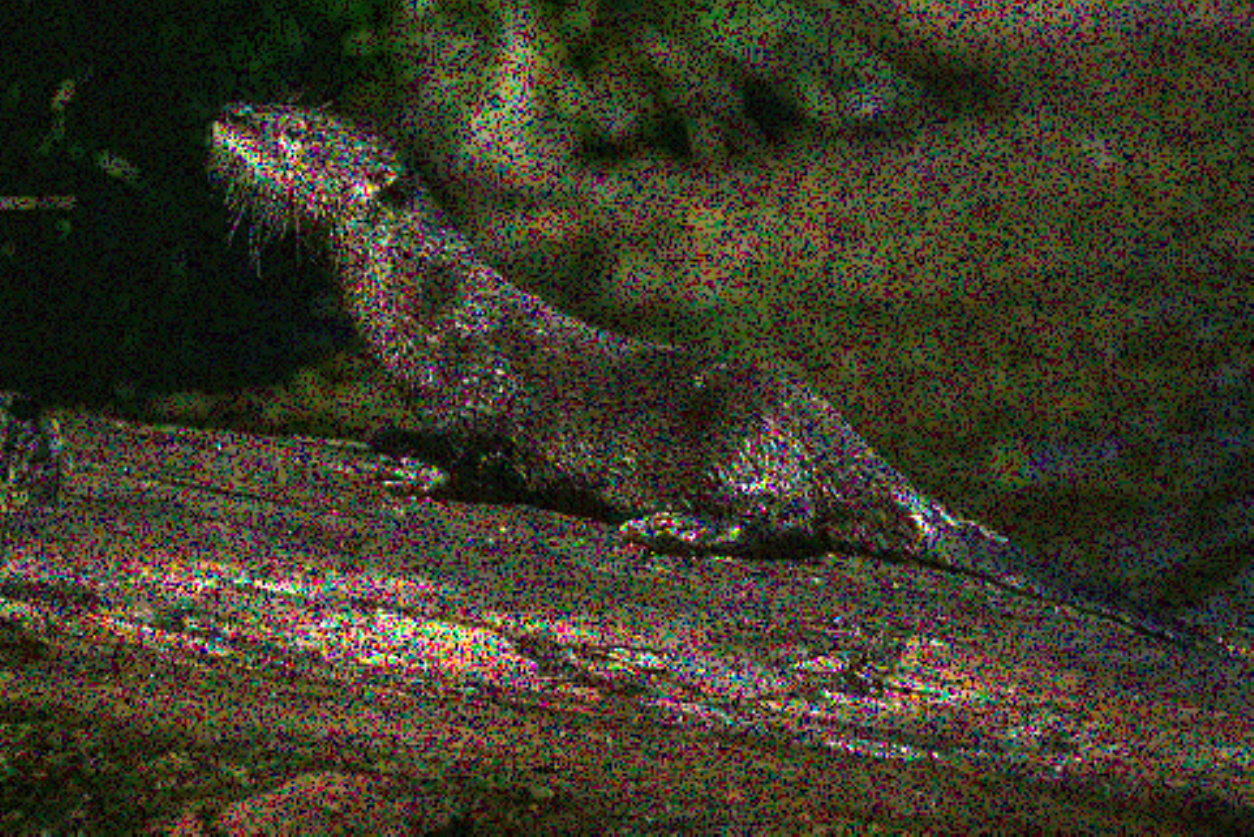}\vspace{0pt}
		\includegraphics[width=\linewidth]{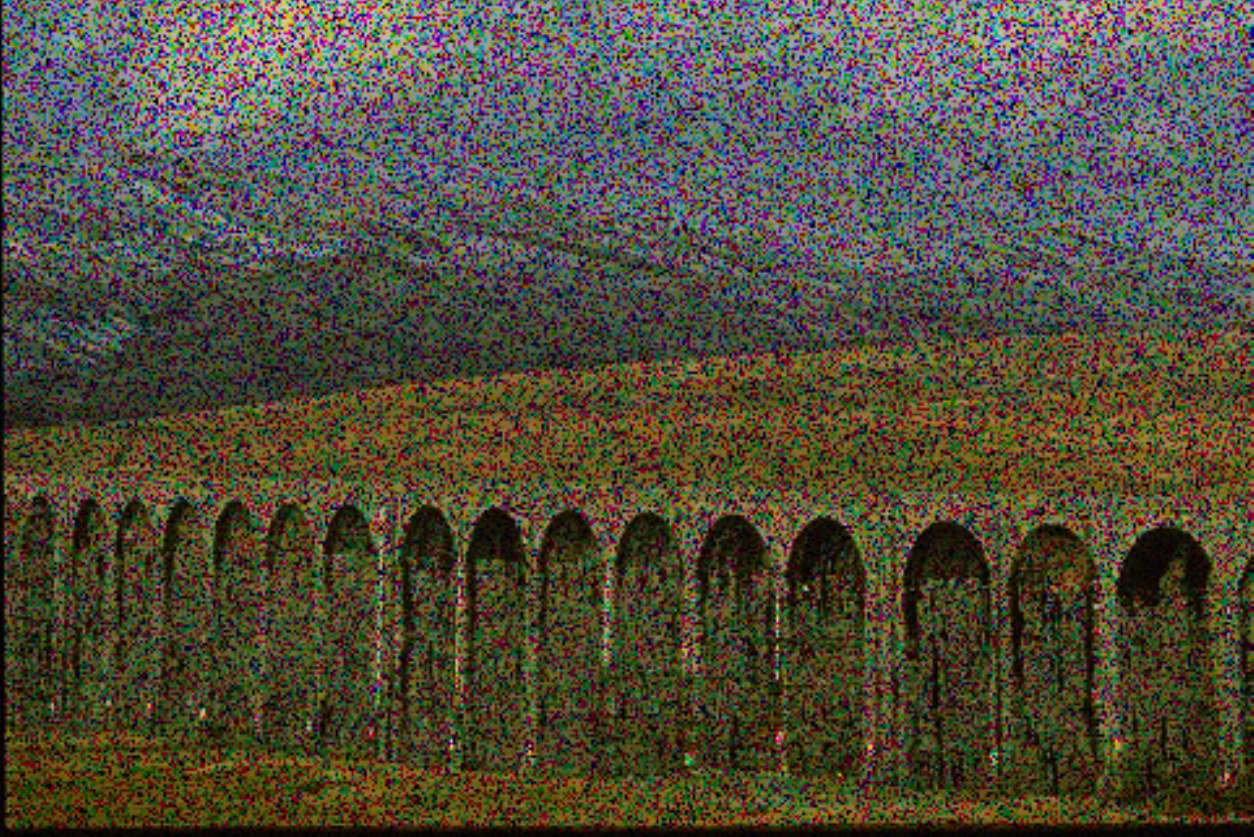}\vspace{0pt}
		\includegraphics[width=\linewidth]{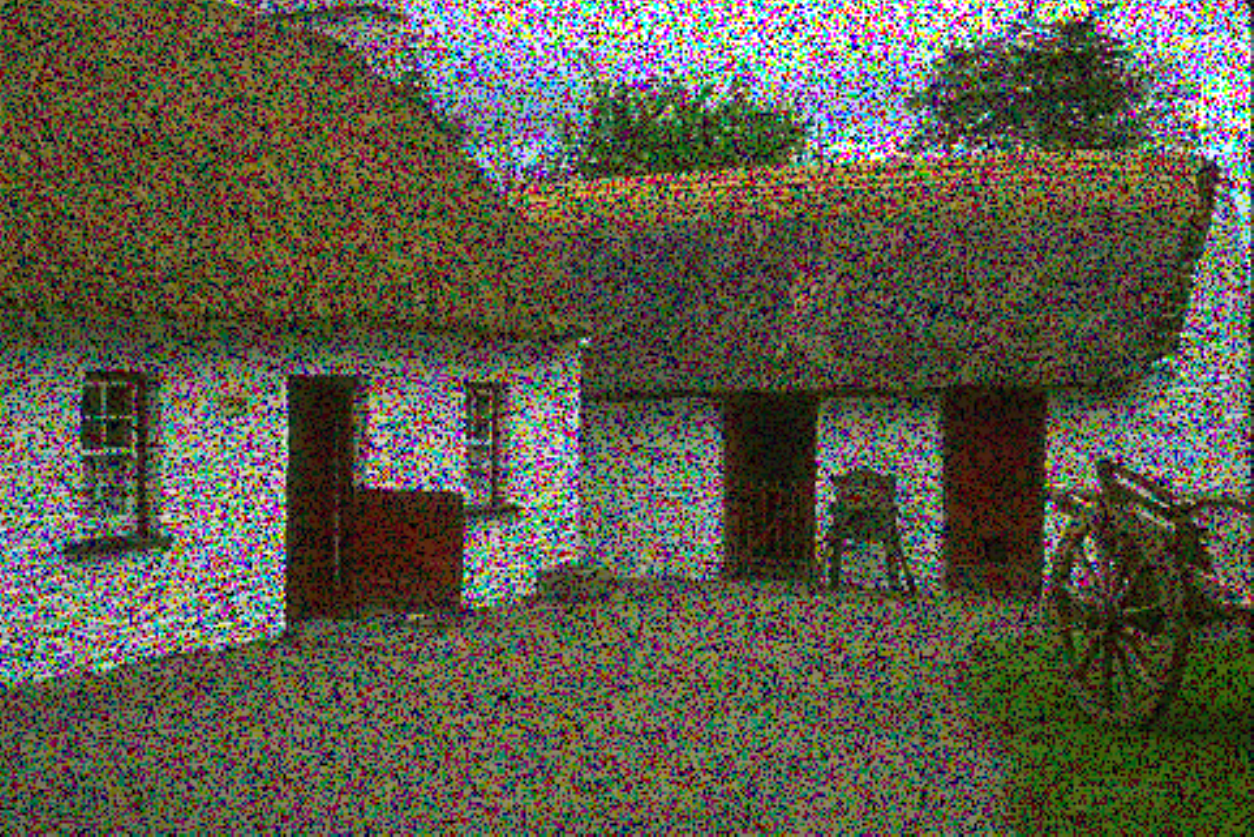}\vspace{0pt}
		\includegraphics[width=\linewidth]{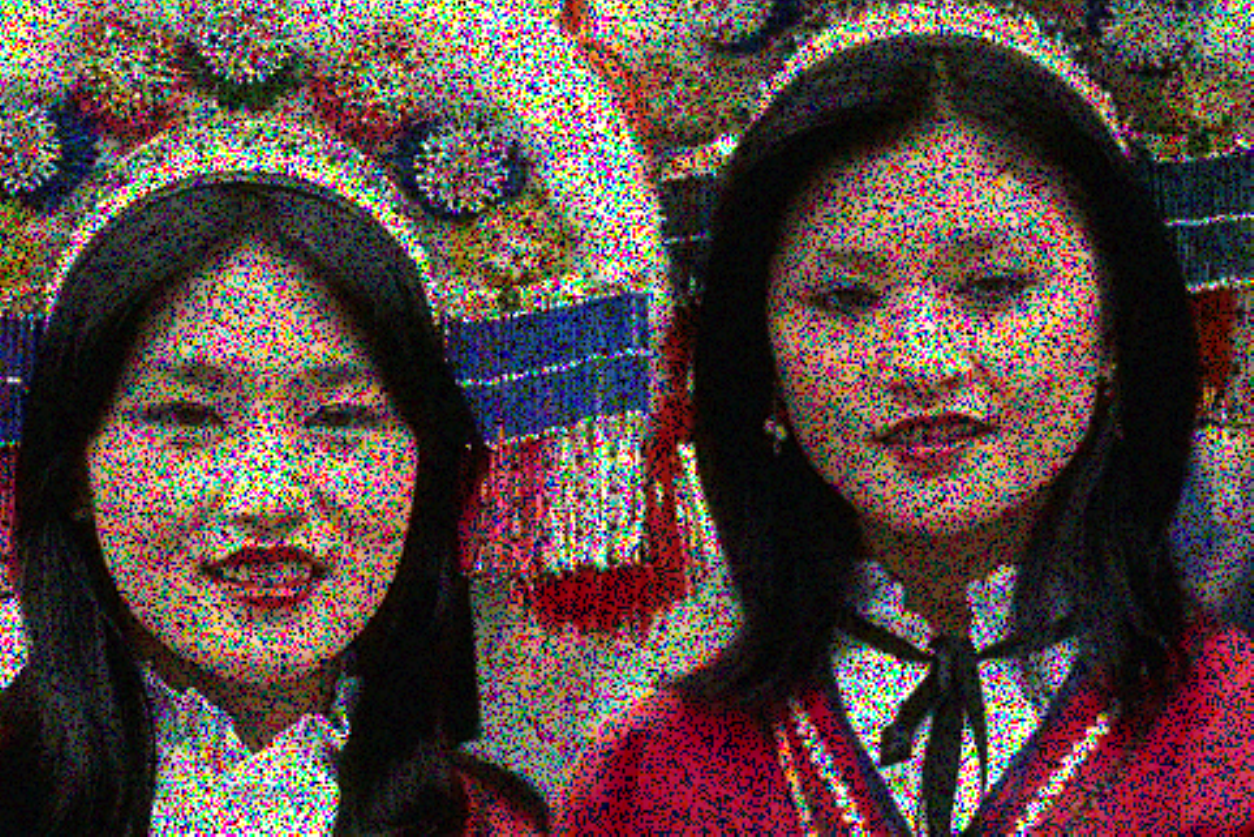}\vspace{0pt}
		\includegraphics[width=\linewidth]{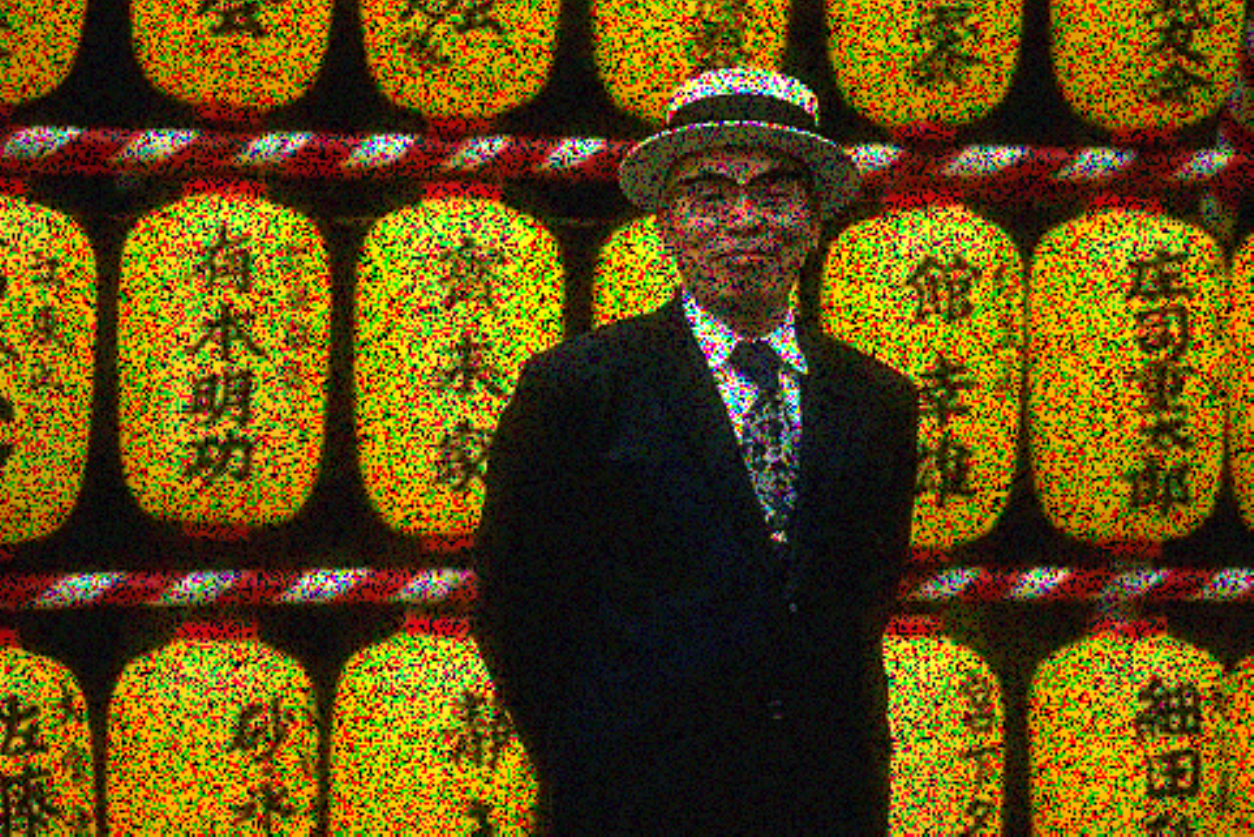}\vspace{0pt}
		\includegraphics[width=\linewidth]{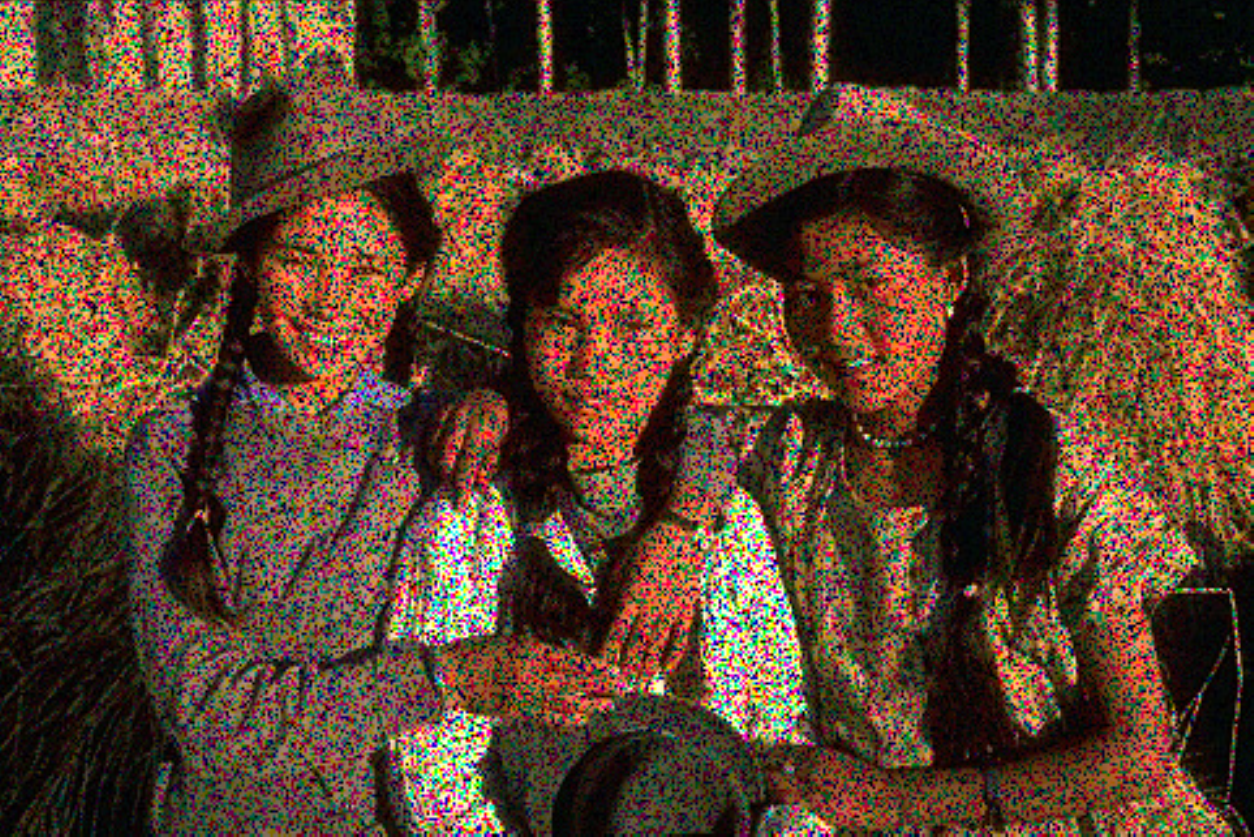}
		\caption{Observation}
	\end{subfigure}
%\subfloat[TRTF]{
	\begin{subfigure}[b]{0.19\linewidth}
	\centering
	\includegraphics[width=\linewidth]{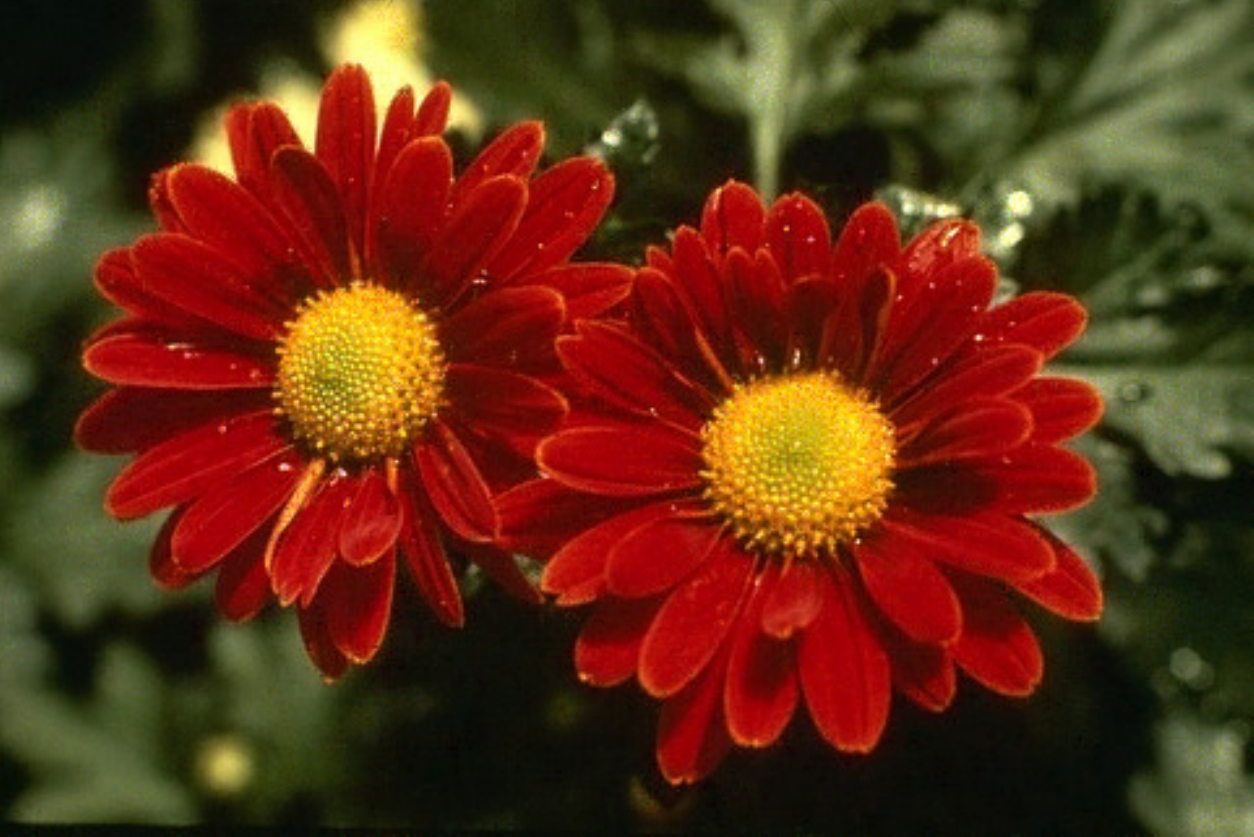}\vspace{0pt}
	\includegraphics[width=\linewidth]{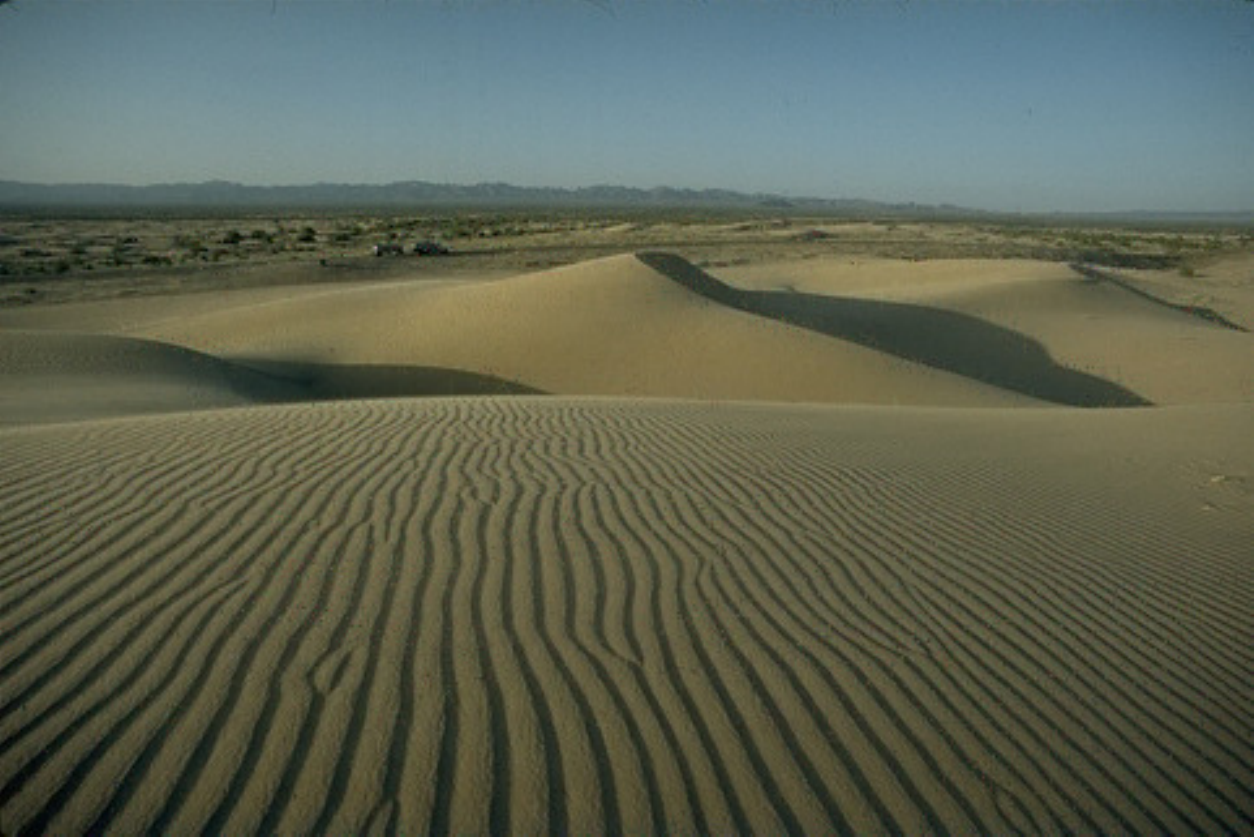}\vspace{0pt}
	\includegraphics[width=\linewidth]{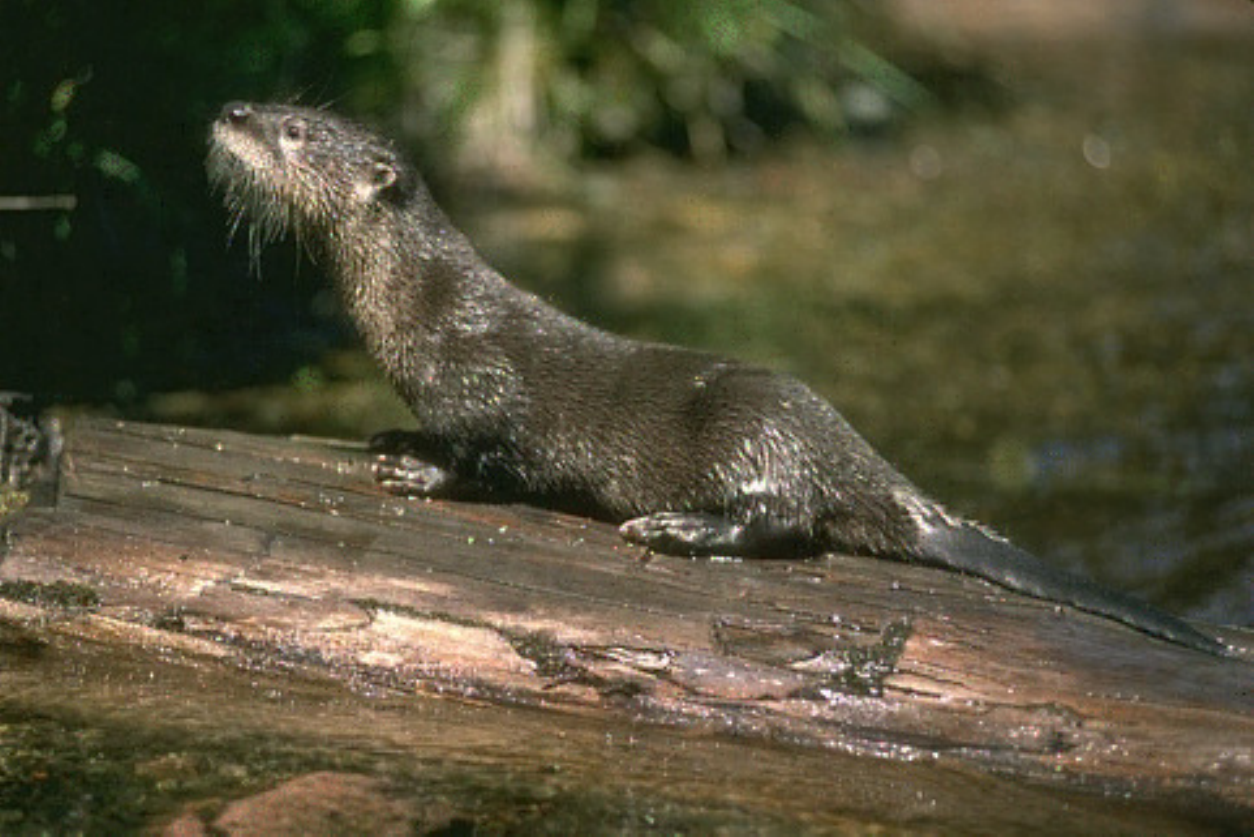}\vspace{0pt}
	\includegraphics[width=\linewidth]{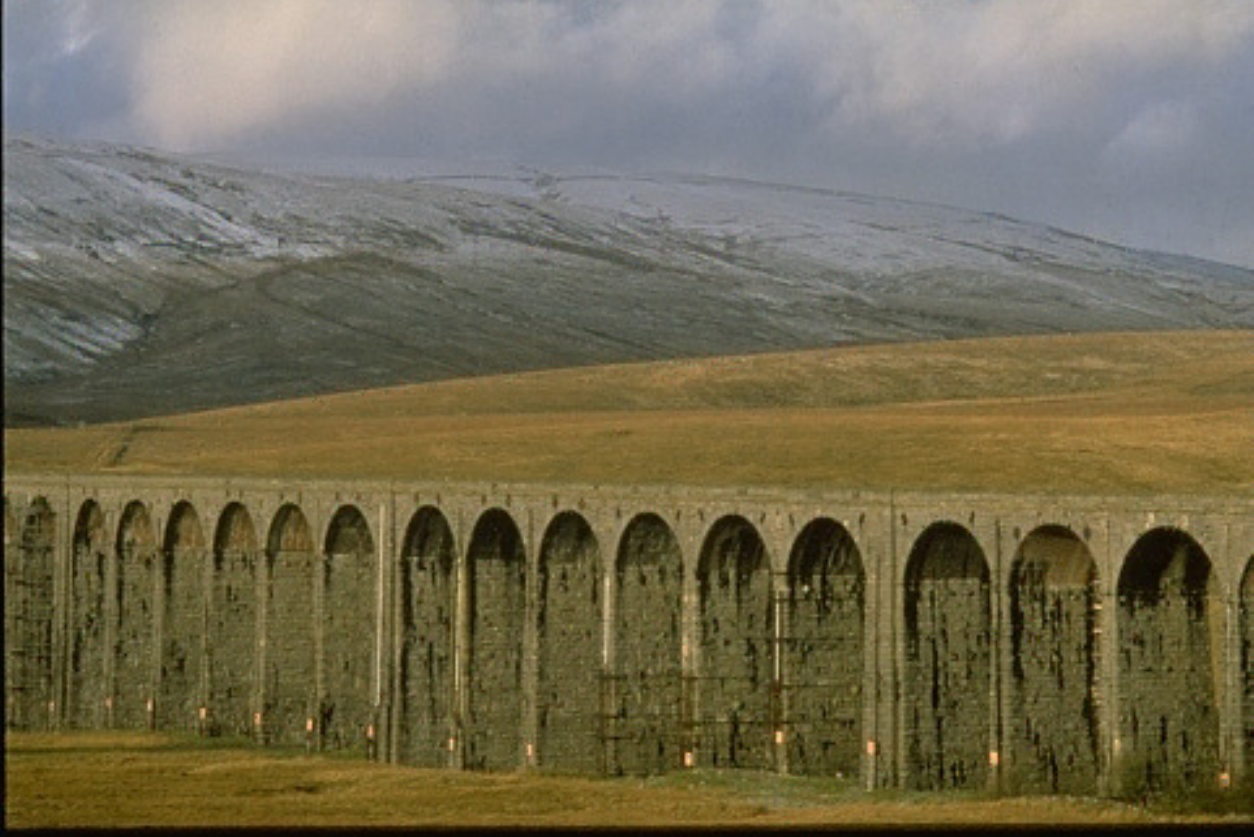}\vspace{0pt}
	\includegraphics[width=\linewidth]{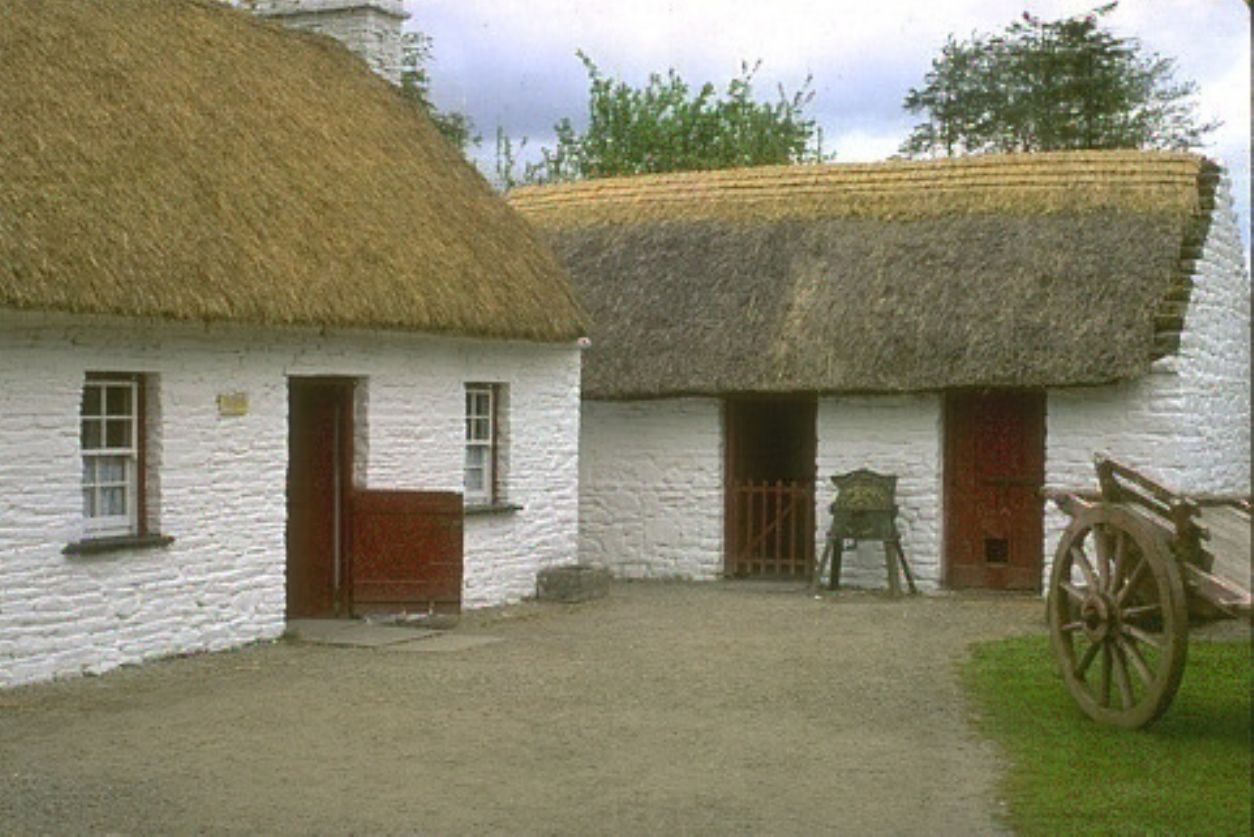}\vspace{0pt}
	\includegraphics[width=\linewidth]{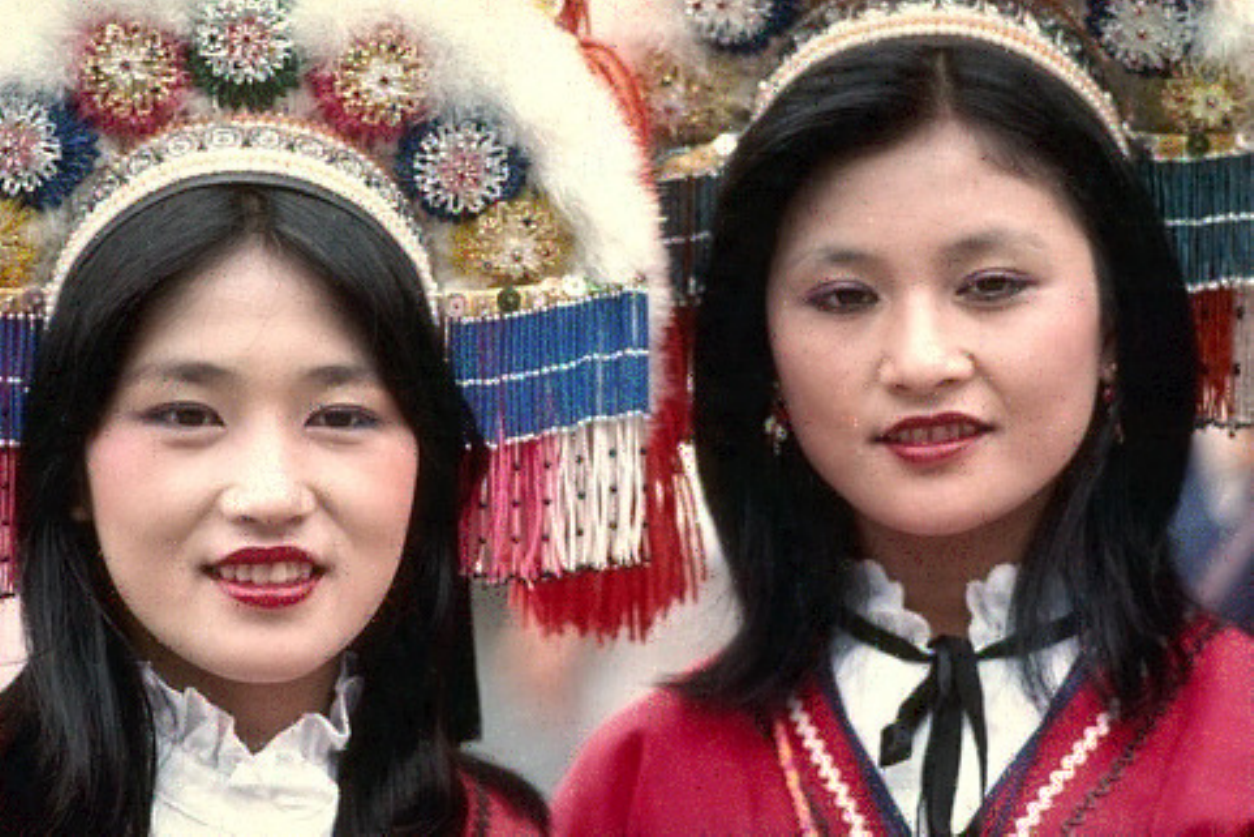}\vspace{0pt}
	\includegraphics[width=\linewidth]{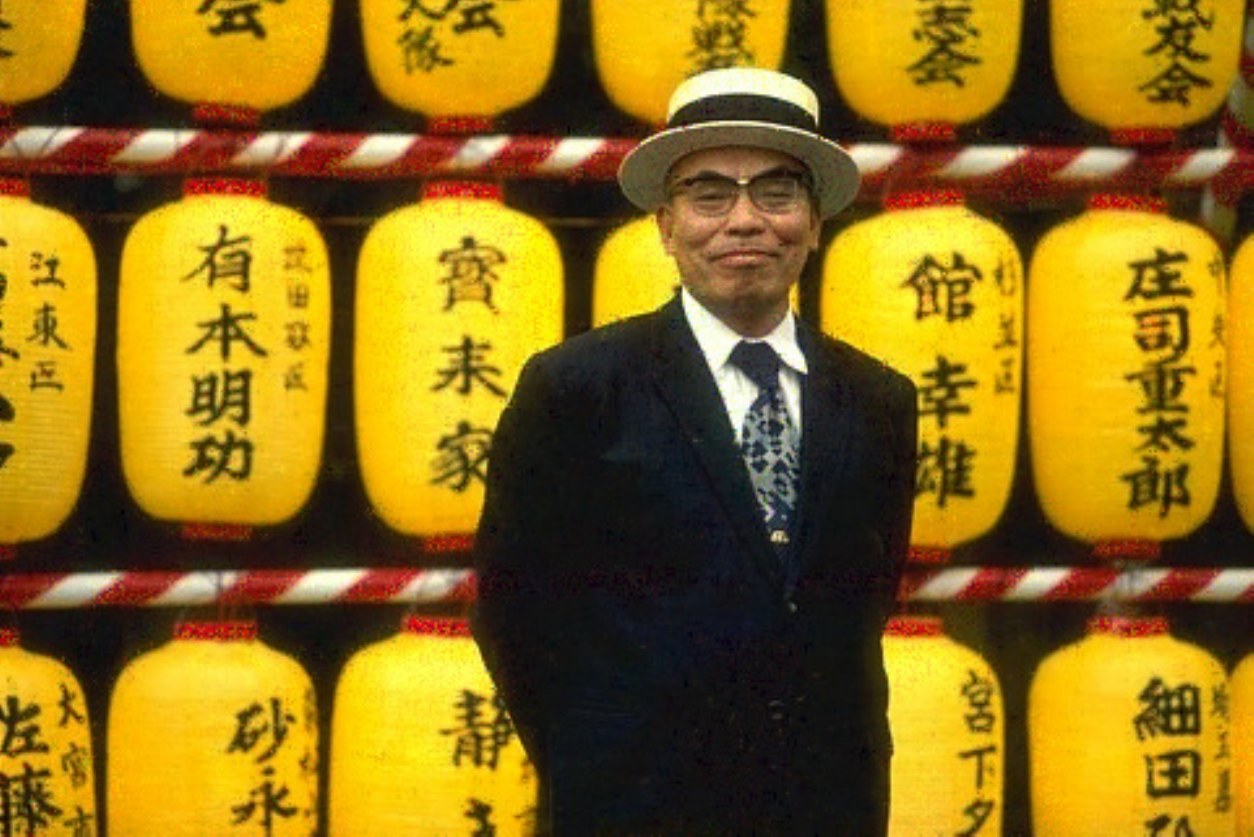}\vspace{0pt}
	\includegraphics[width=\linewidth]{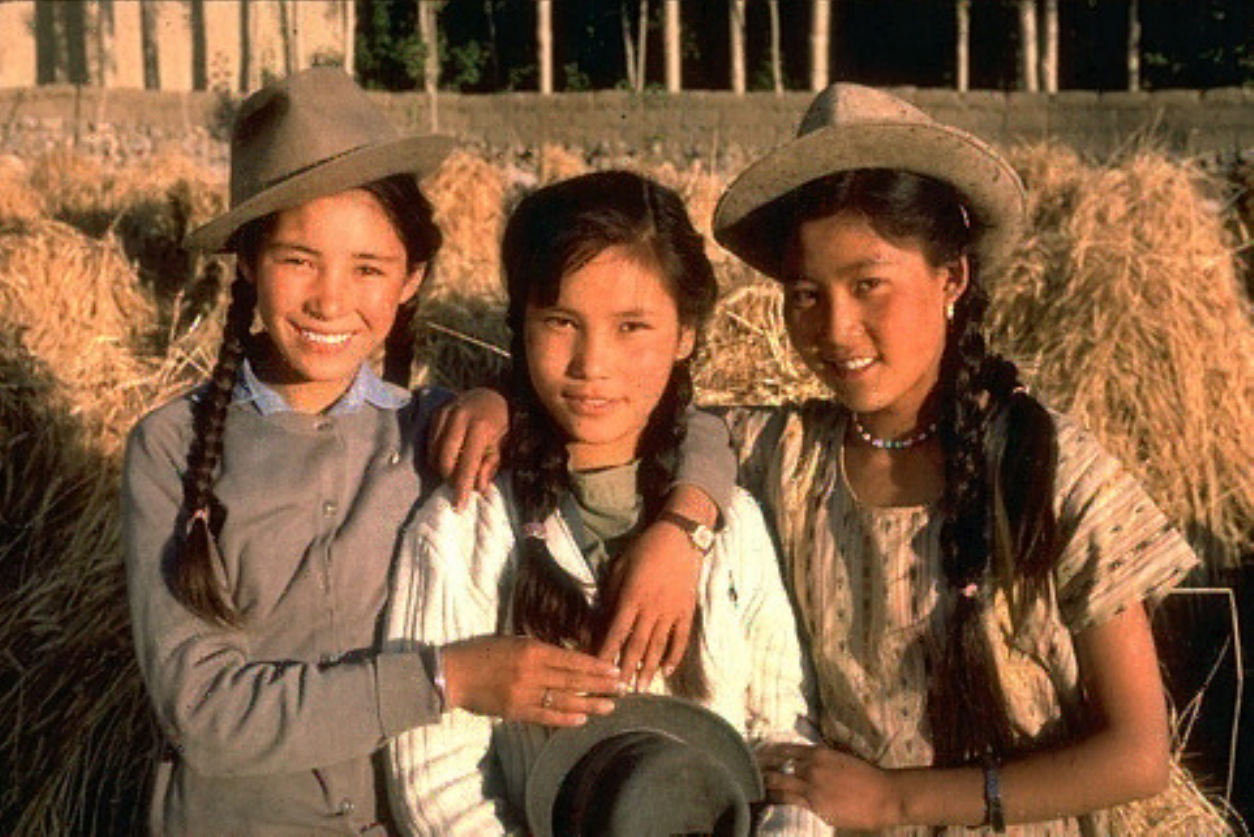}
	\caption{MTRTC}
\end{subfigure}
%\subfloat[TCTF]{
	\begin{subfigure}[b]{0.19\linewidth}
	\centering
	\includegraphics[width=\linewidth]{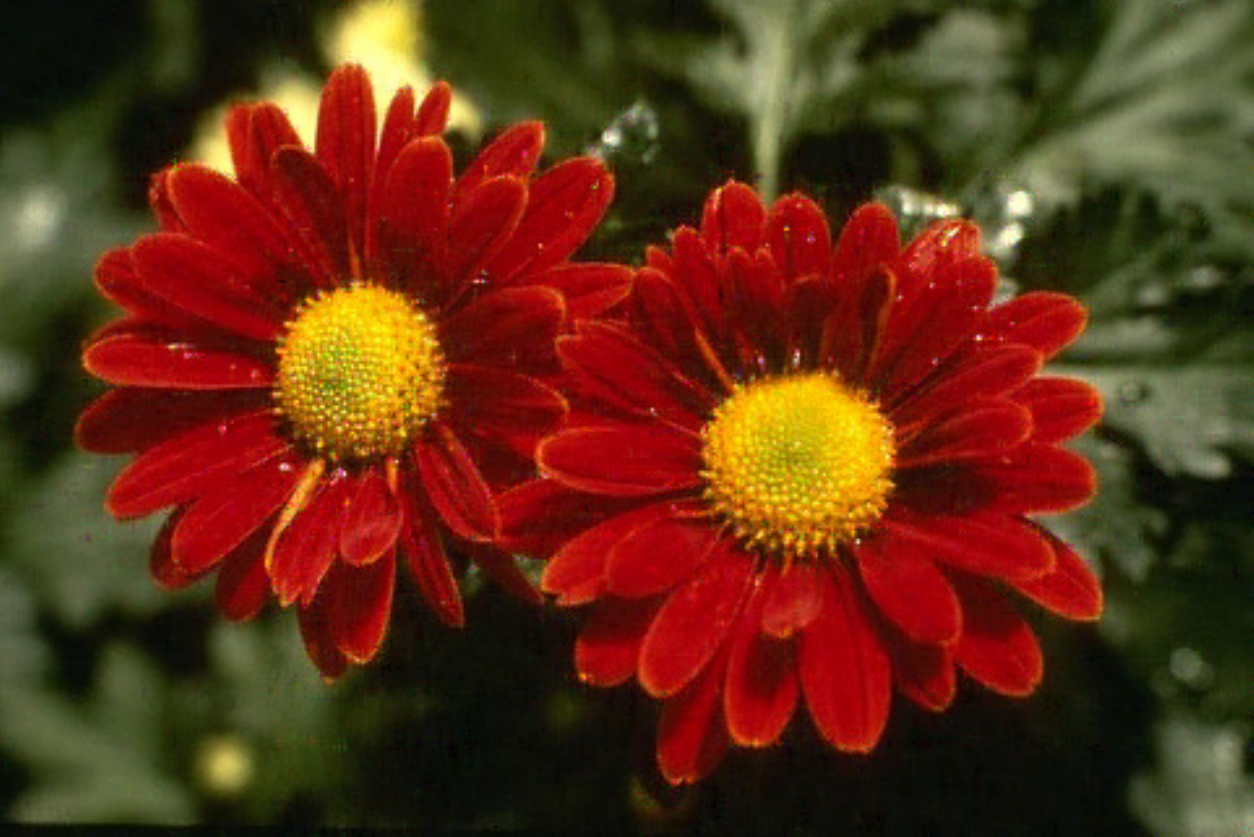}\vspace{0pt}
	\includegraphics[width=\linewidth]{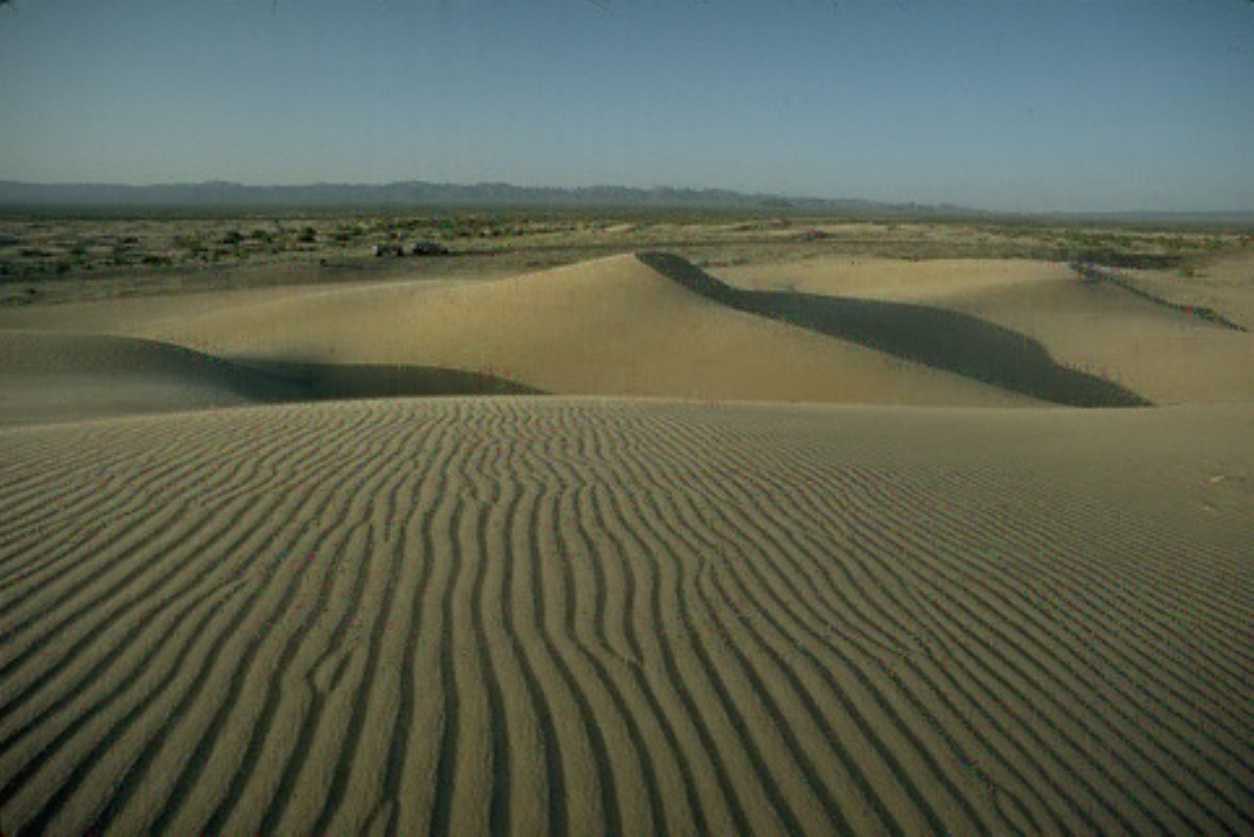}\vspace{0pt}
	\includegraphics[width=\linewidth]{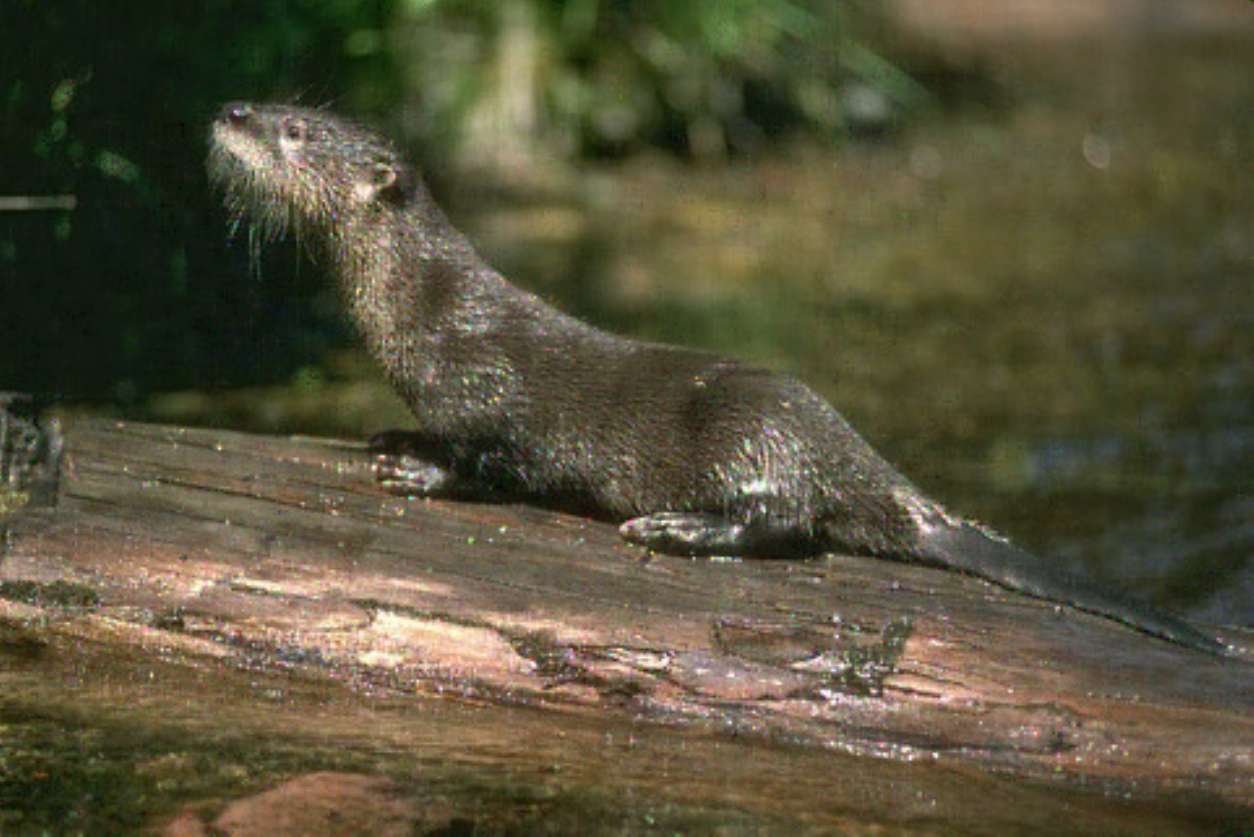}\vspace{0pt}
	\includegraphics[width=\linewidth]{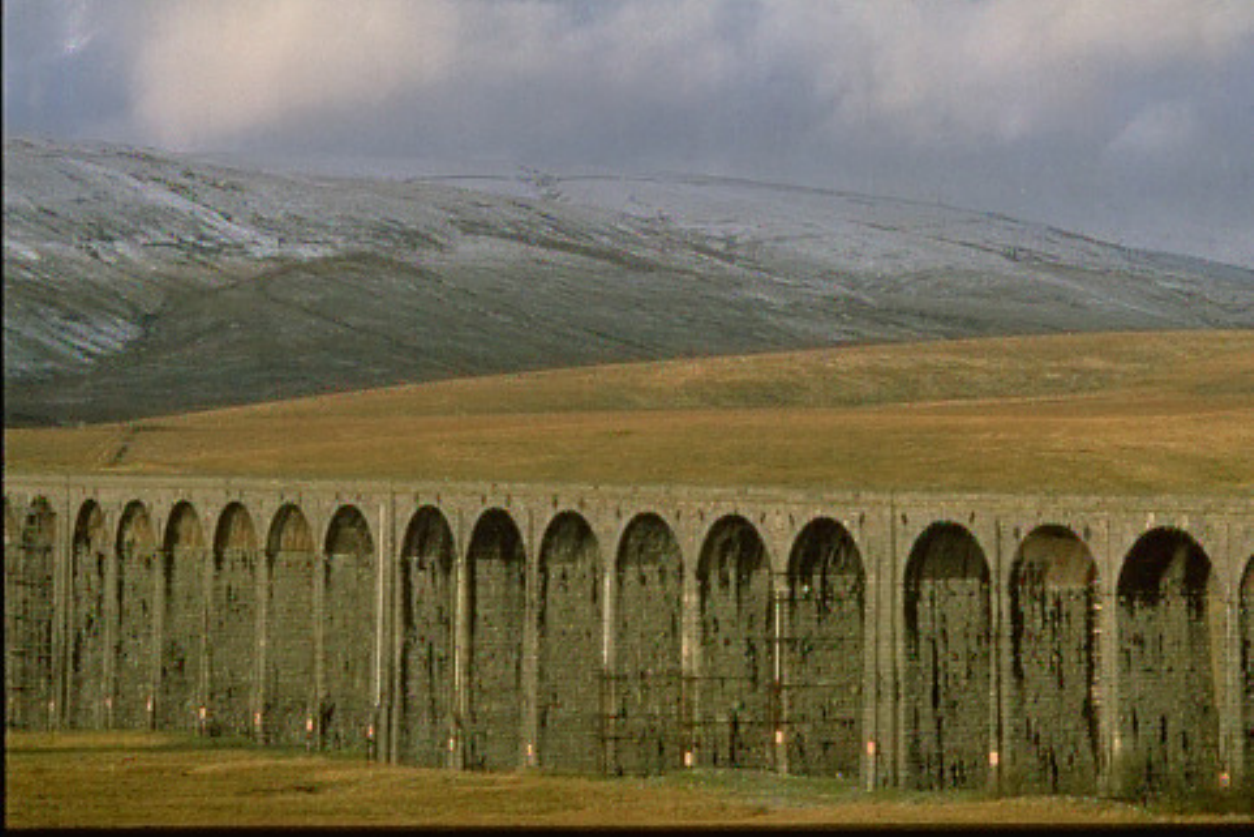}\vspace{0pt}
	\includegraphics[width=\linewidth]{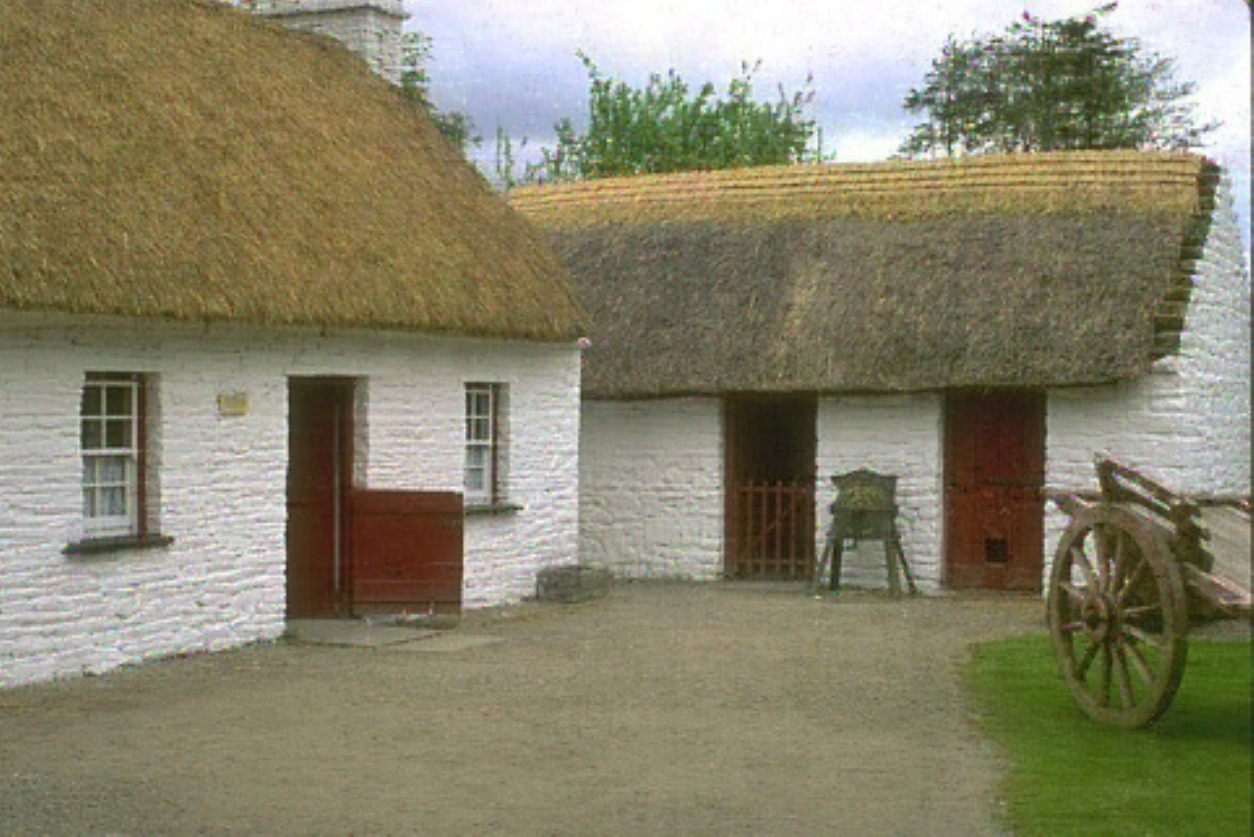}\vspace{0pt}
	\includegraphics[width=\linewidth]{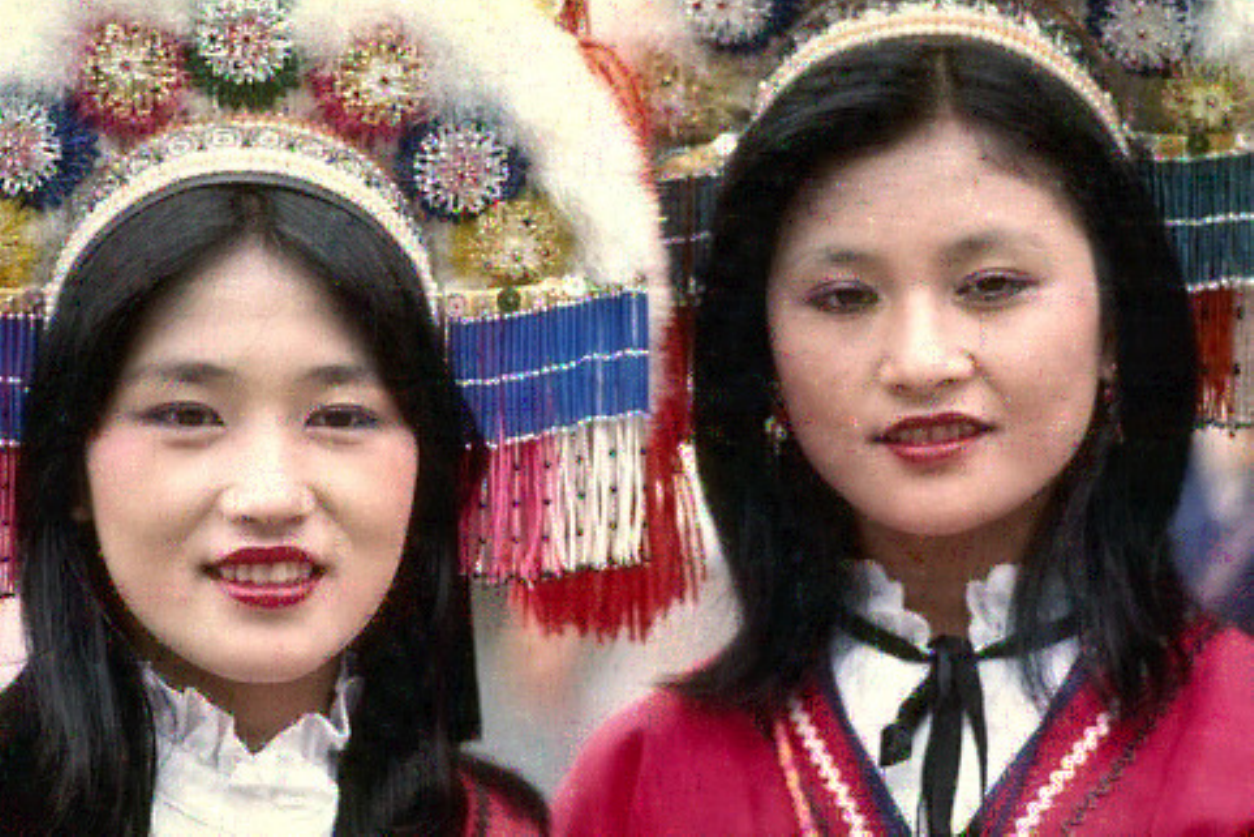}\vspace{0pt}
	\includegraphics[width=\linewidth]{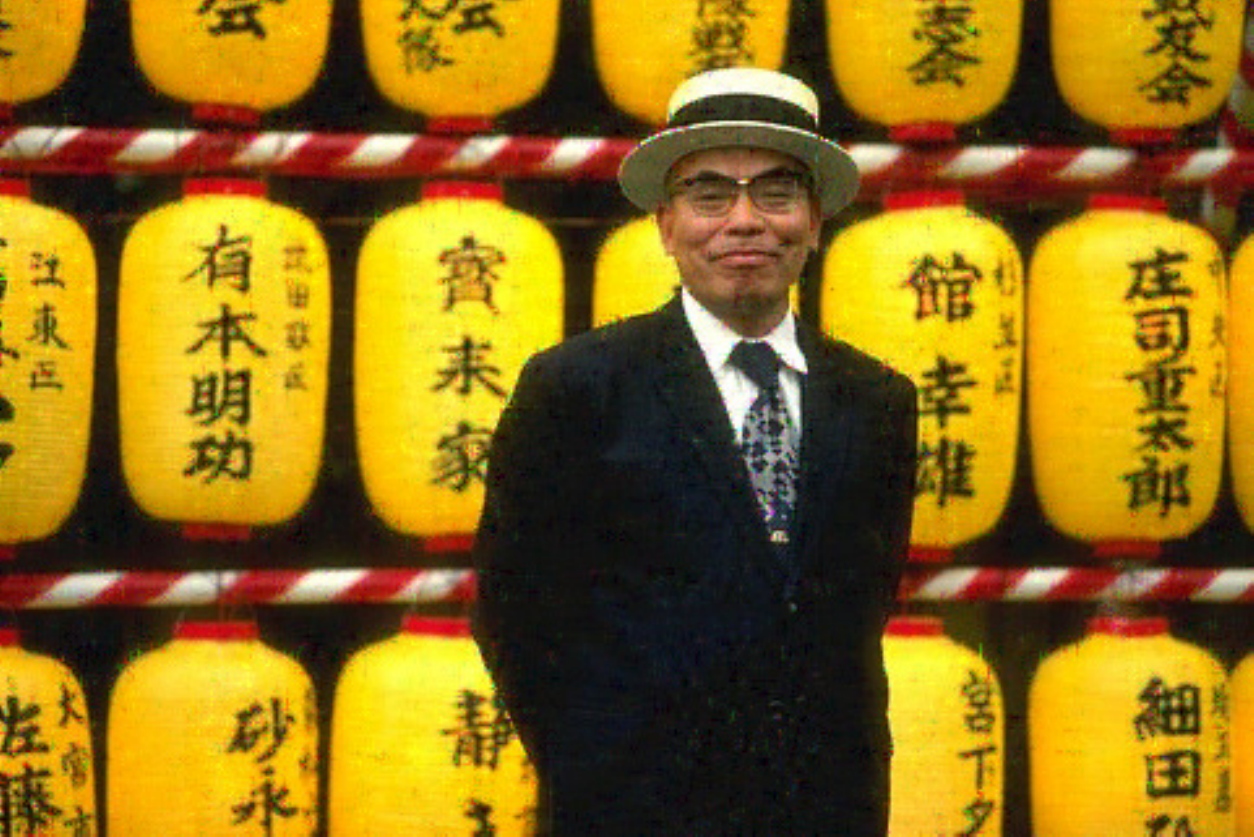}\vspace{0pt}
	\includegraphics[width=\linewidth]{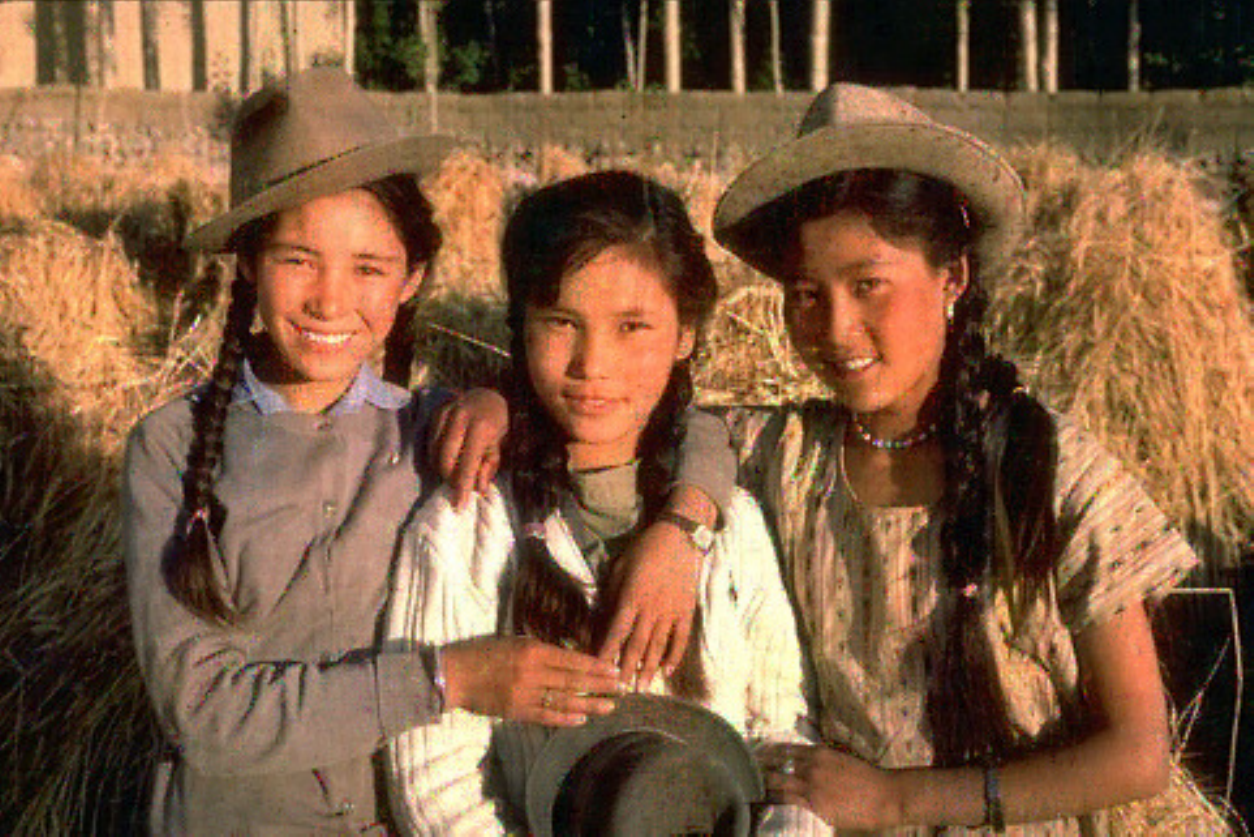}
	\caption{TCTF}
    \end{subfigure}
%\subfloat[TMac]{
	\begin{subfigure}[b]{0.19\linewidth}
	\centering
	\includegraphics[width=\linewidth]{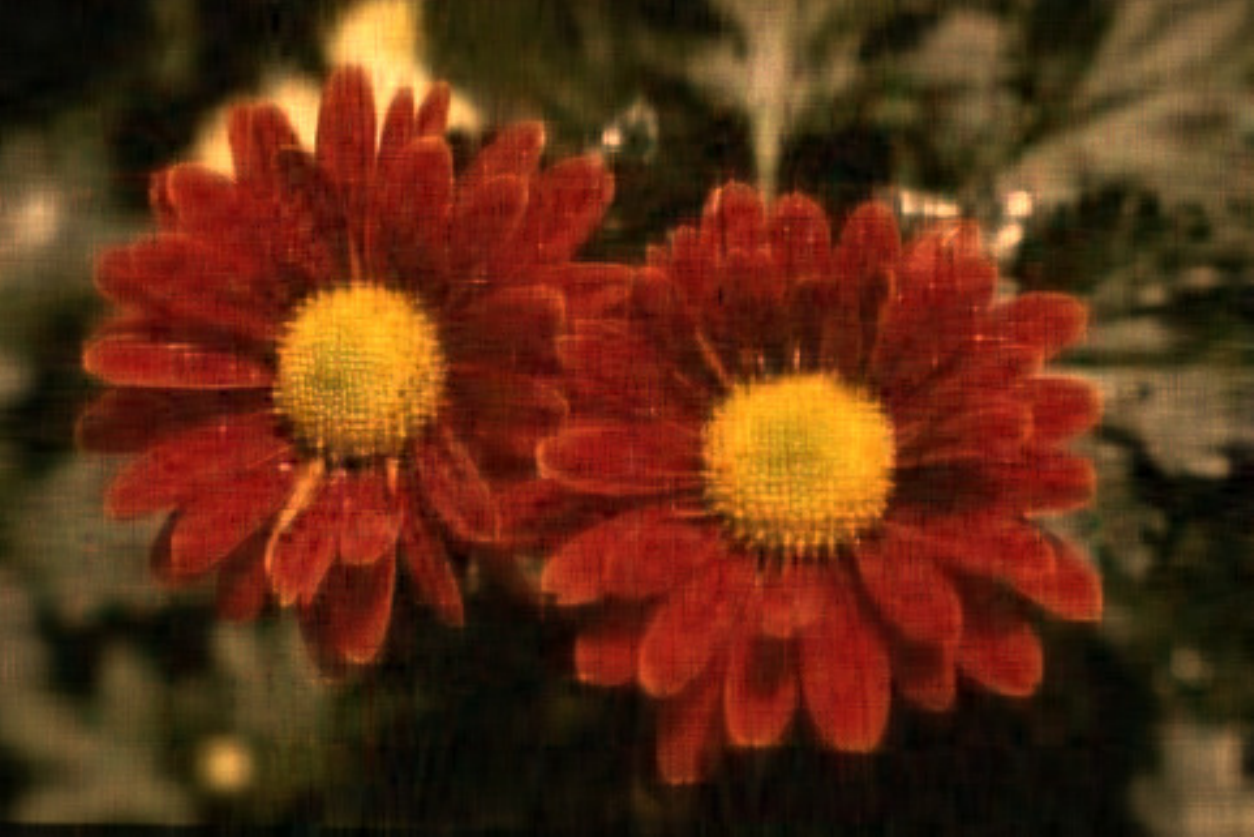}\vspace{0pt}
	\includegraphics[width=\linewidth]{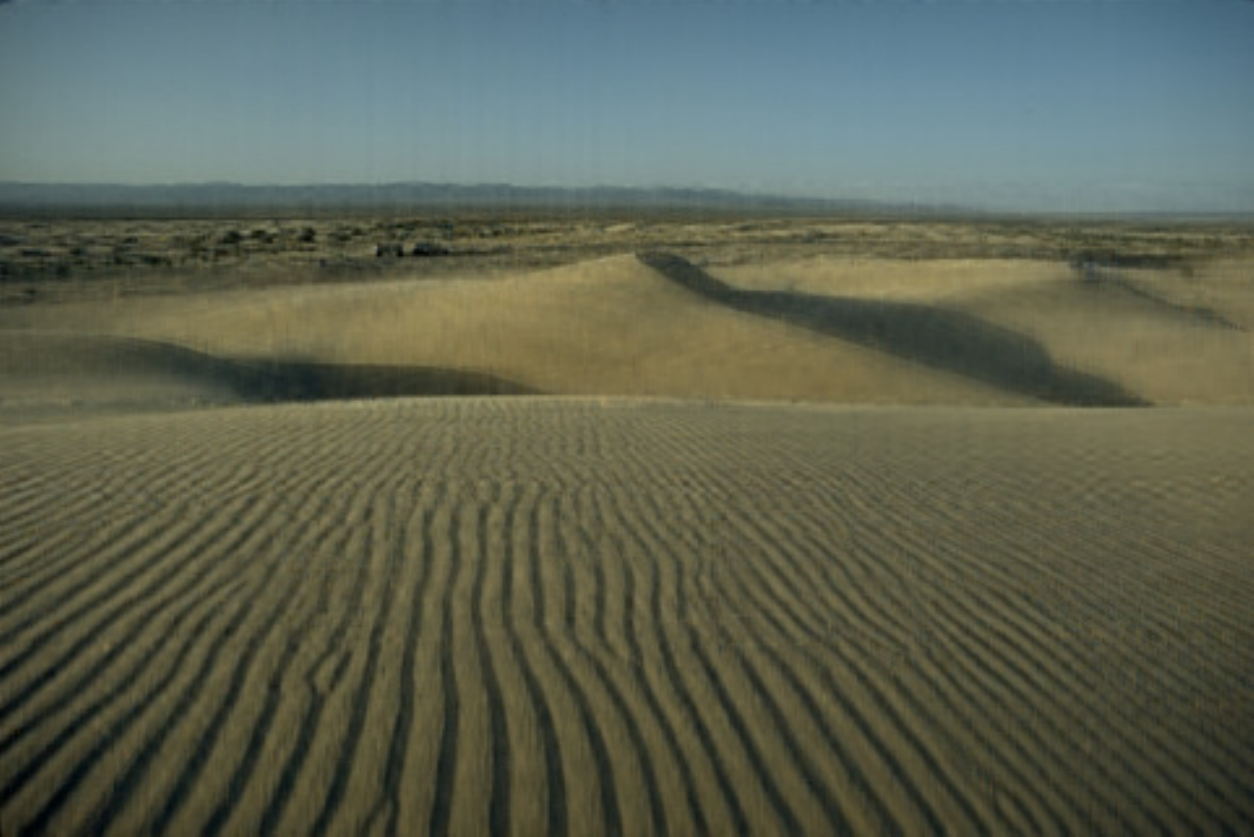}\vspace{0pt}
	\includegraphics[width=\linewidth]{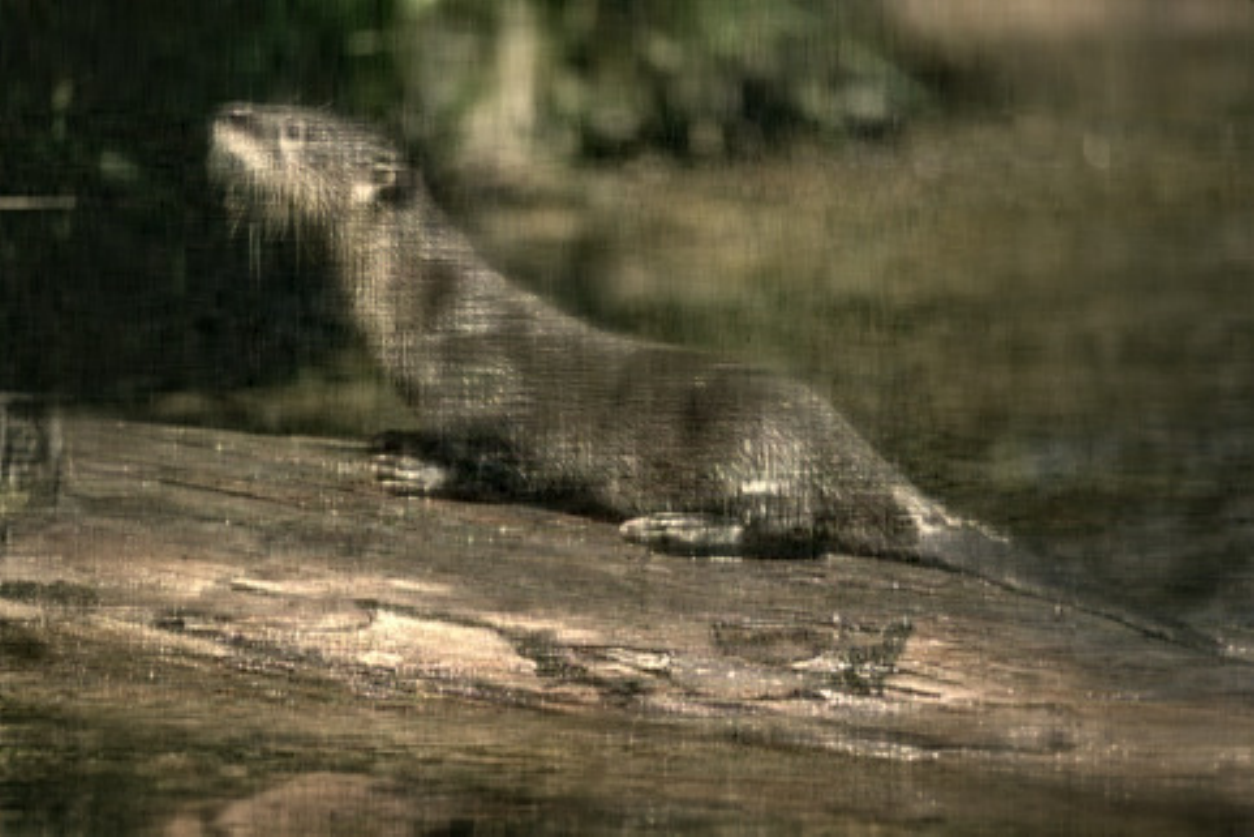}\vspace{0pt}
	\includegraphics[width=\linewidth]{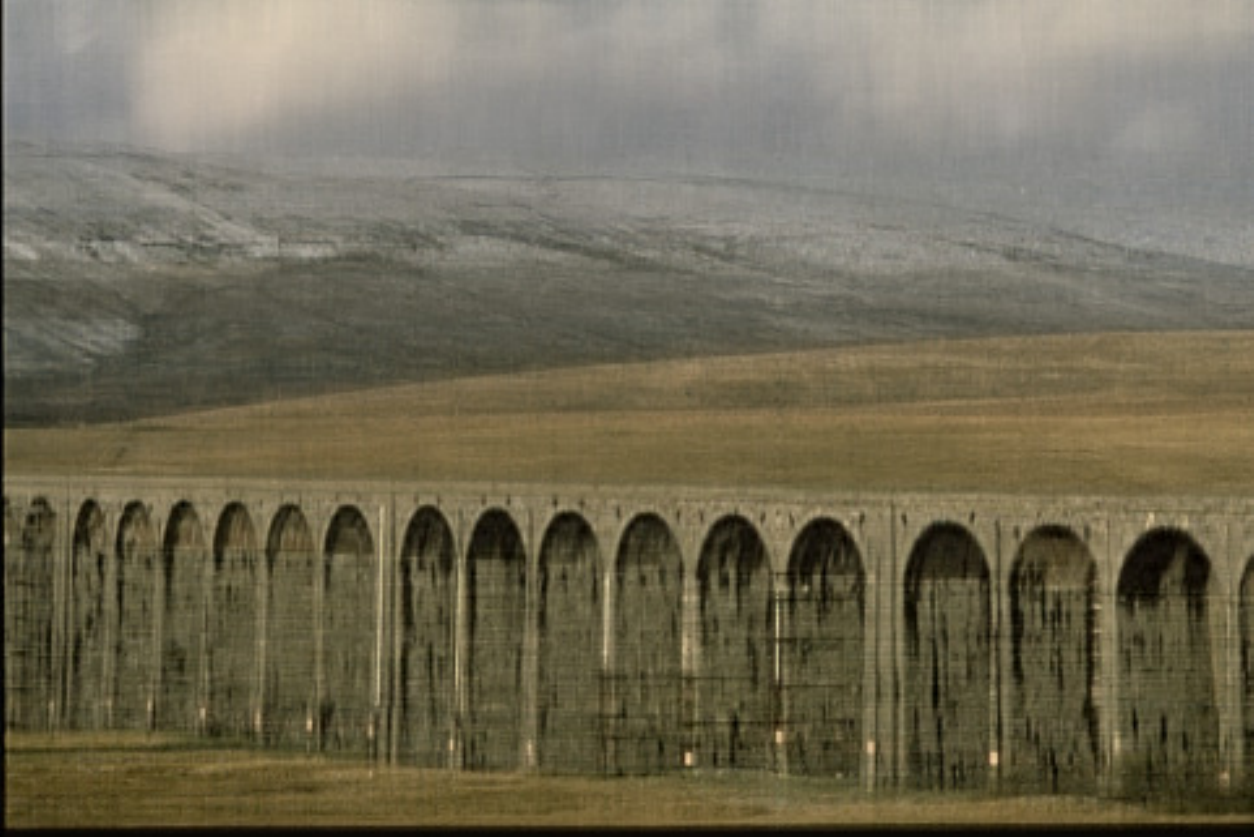}\vspace{0pt}
	\includegraphics[width=\linewidth]{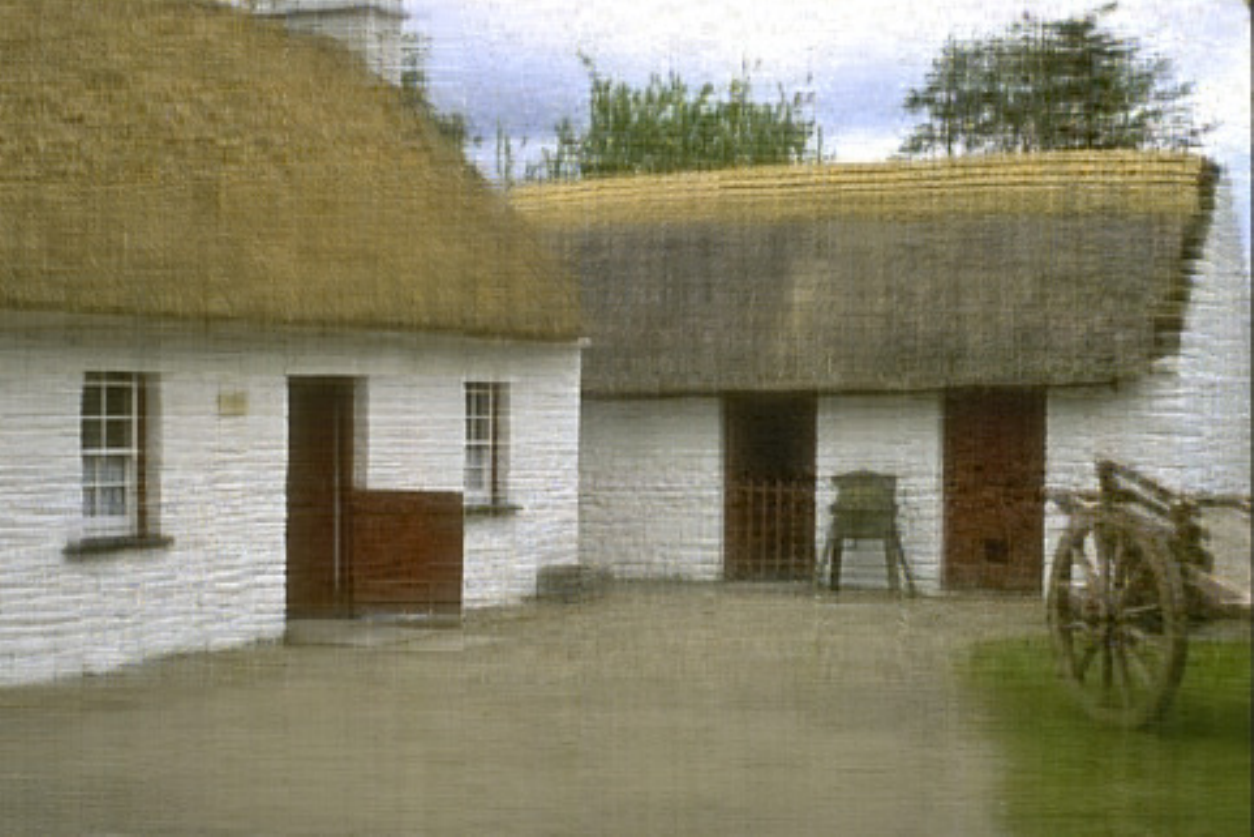}\vspace{0pt}
	\includegraphics[width=\linewidth]{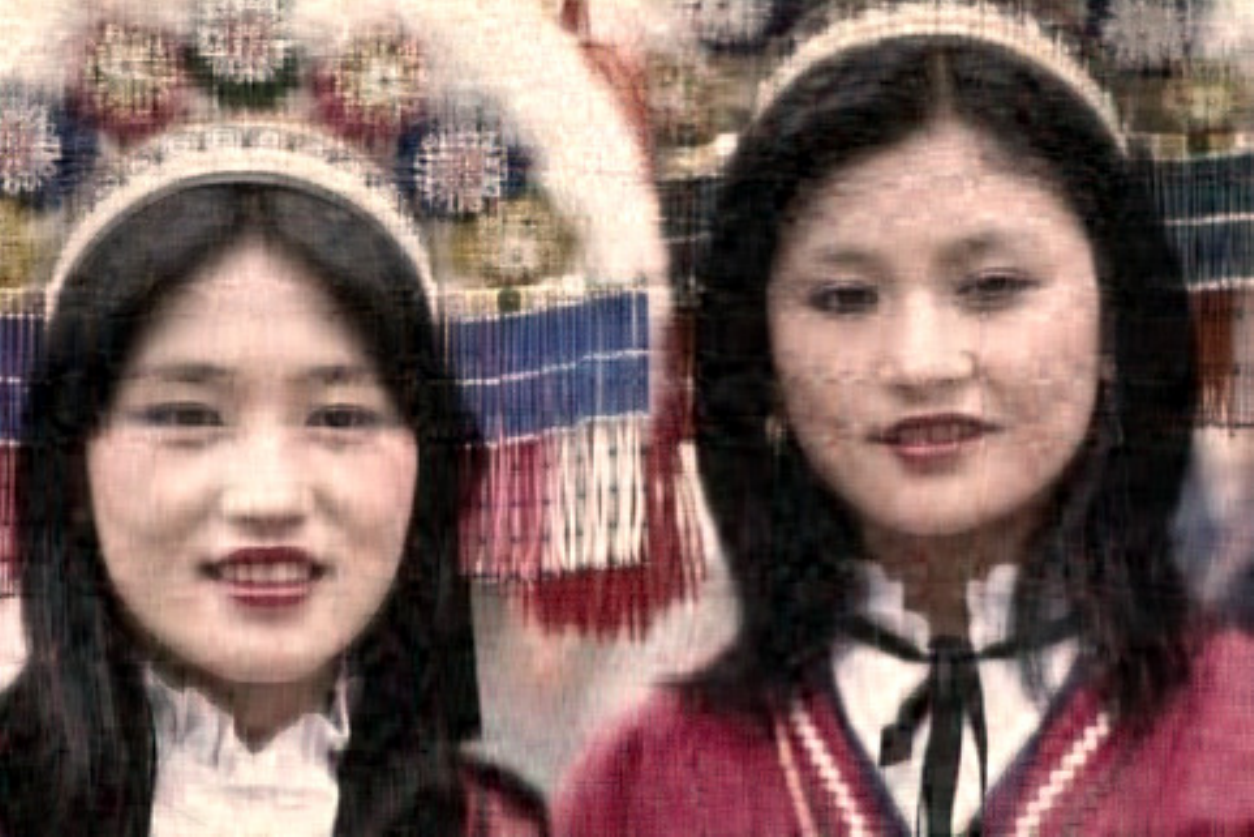}\vspace{0pt}
	\includegraphics[width=\linewidth]{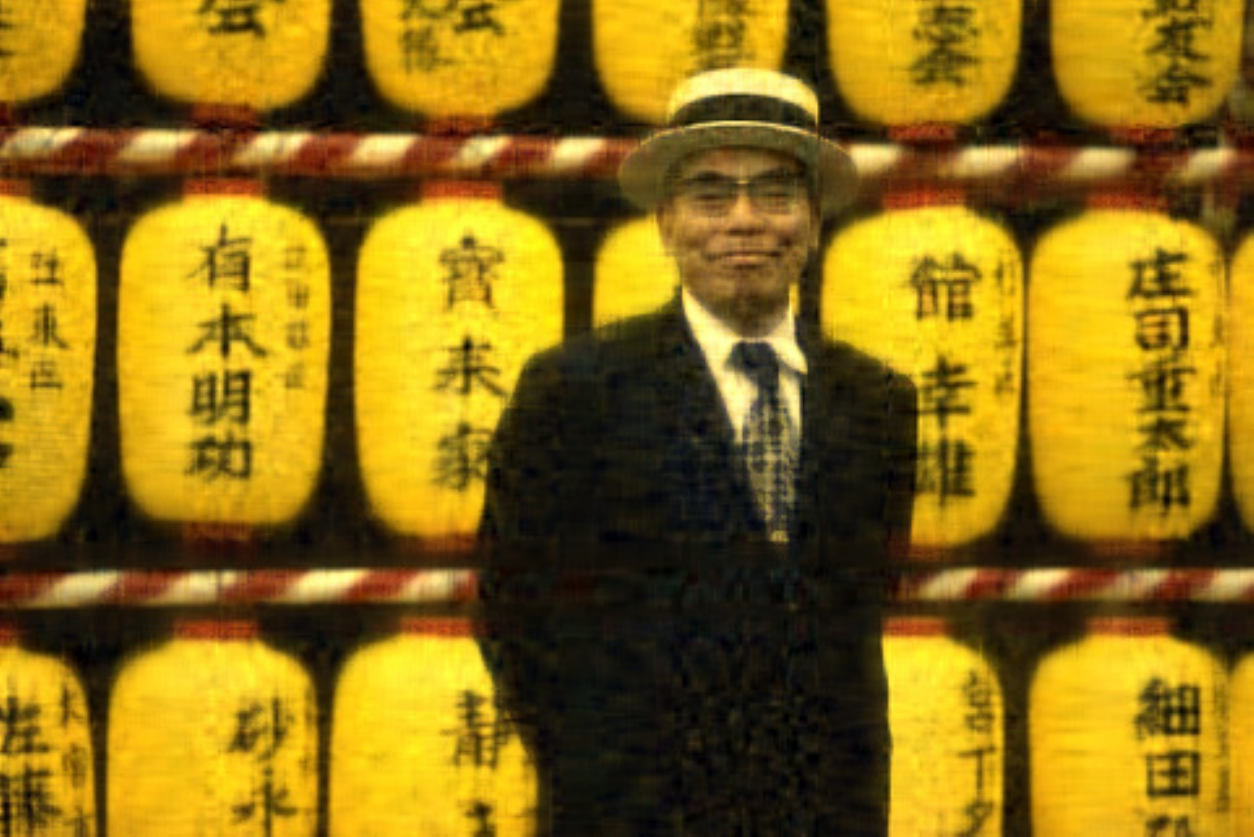}\vspace{0pt}
	\includegraphics[width=\linewidth]{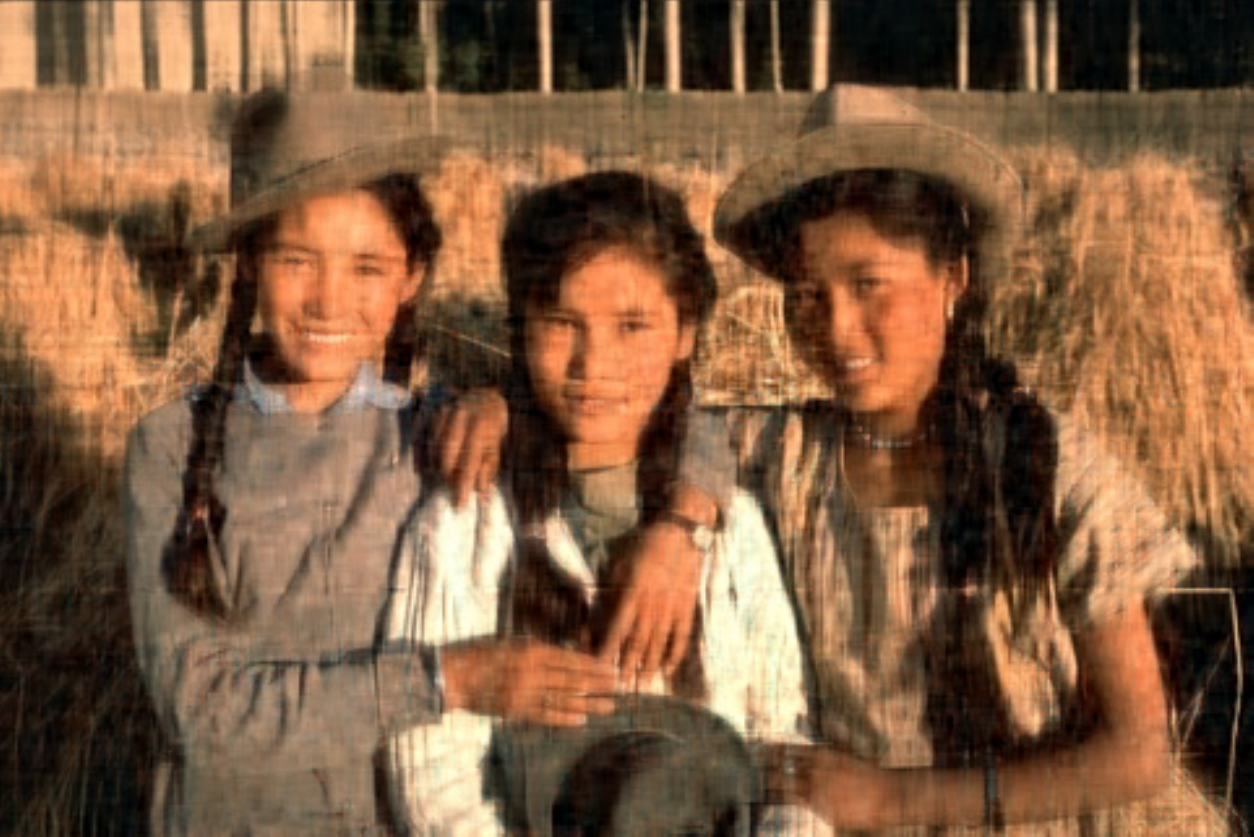}
	\caption{TMac}
    \end{subfigure}

	\end{subfigure}
	\vfill
	\caption{Recovery performance comparison on the $8$ images by MTRTC, TCTF and TMac }
	\label{fig:image_inpainting}
\end{figure}

\par We present the image inpainting  results of the eight tested images in Table \ref{tab:image}, Figure  \ref{fig:image_bar} and Figure \ref{fig:image_inpainting}, in which ``Average'' denotes the average inpainting  results of all 200 images. ``Average'' indicates that MTRTC outperforms TCTF and TMac. As stated in \cite{KBHH13,ZEAHK14}, TMac expands the tensor data directly into matrices and applies matrix nuclear norm to approximate matrix rank, which may destroy multi-data structures and cause performance degradation. Based on tensor factorization, TCTF and MTRTC avoid the loss of tensor structure information \cite{KBHH13,ZEAHK14}, thus obtain better inpainting results. Although TCTF requires less time in each iteration, it takes more iterations to converge, see the running time in Table \ref{tab:image}.
Furthermore, MTRTC takes account of all the modes,  which is more comprehensive to preserve all low rank structure of tensor data. %From the reported results,  MTRTC outperforms the existing methods for all considered images.
From Figure \ref{fig:image_bar},  MTRTC is the fastest one, which needs about 2/3 times running time of TCTF and TMAC.

\begin{figure}[htbp]
	\centering
	\begin{subfigure}[t]{\linewidth}
		\centering
		\includegraphics[width=6in]{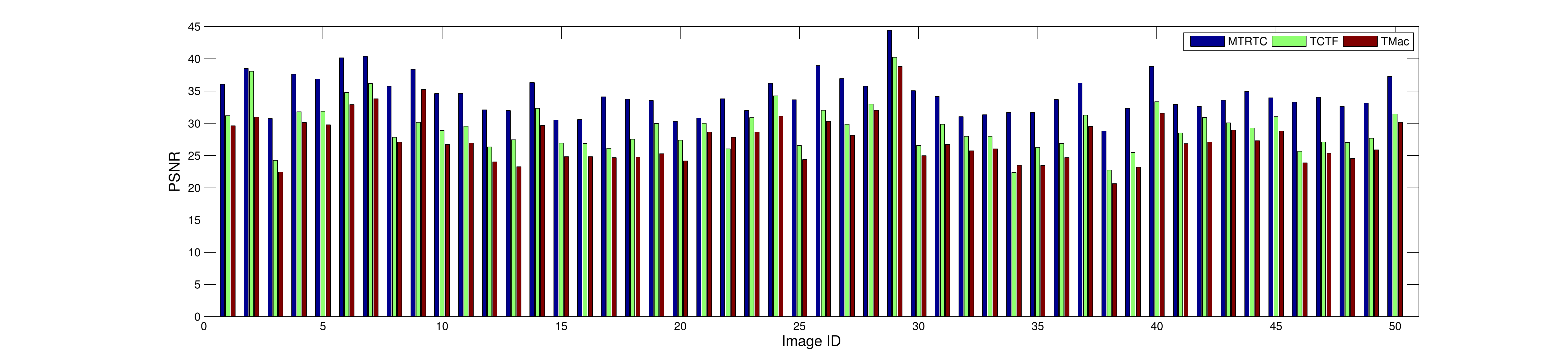}
		\caption{Comparison of the PSNR values on $50$ images}
	\end{subfigure}

	\begin{subfigure}[t]{\linewidth}
		\centering
		\includegraphics[width=6in]{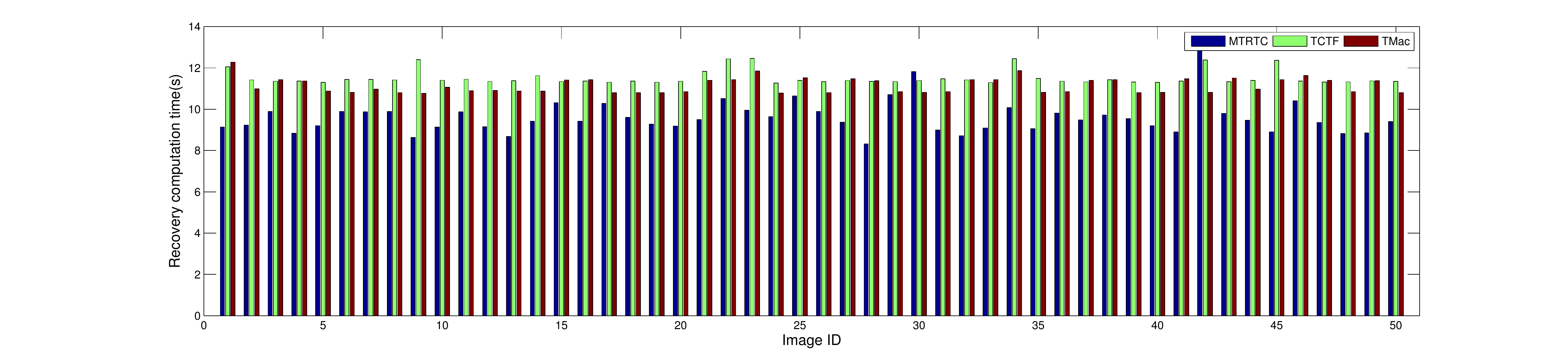}
		\caption{Comparison of the running time on $50$ images}
	\end{subfigure}
	\centering
	\caption{Comparison of the PSNR and the running time on $50$ images}
	\label{fig:allimage_bar}
\end{figure}

In Figure \ref{fig:allimage_bar}, we report the PSNR values and the running time of all methods on the first 50 images. MTRTC performs the best with at least 1.2 times improvement upon the PSNR metric on all 50 images, verifying its advantages and robustness. From Figure \ref{fig:allimage_bar} (b),  MTRTC is much faster than other compared methods. In conclusion, it not only achieves the best
inpainting results but also runs within least running time.

For further comparison, we also  recover  images of the deterministically masked images by grids, leaves and letters, respectively. In experiments, the maximum iteration number is set to be 500 and the termination precision $ \varepsilon $ is set to be $1e$-5. Clearly, the masked images  are no-mean-sampling. The results are displayed in  Figure \ref{fig:building} and Table \ref{tab:building}, which show that TCTF and MTRTC have better performance than TMac. Furthermore, the effect of MTRTC is much better than that of TCTF. Table \ref{tab:building} reports all numerical results of three methods. We can assert that MTRTC is the best one in MTRTC, TCTF and TMac.
\begin{figure*}[htbp]
	\centering	
		\begin{subfigure}[b]{0.16\linewidth}
	\includegraphics[width=\linewidth]{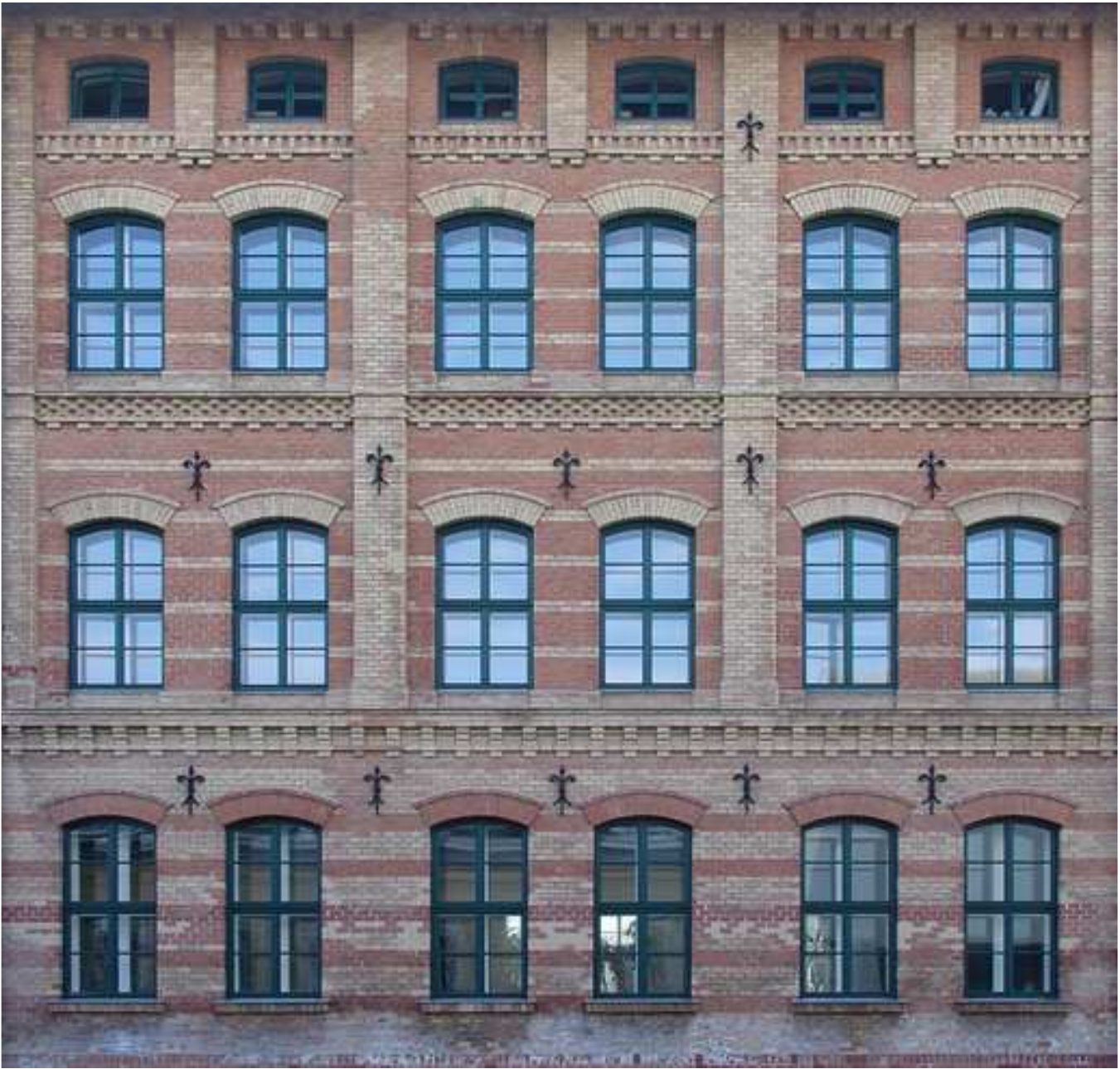}	
\end{subfigure}
   	\begin{subfigure}[b]{0.16\linewidth}
	\includegraphics[width=\linewidth]{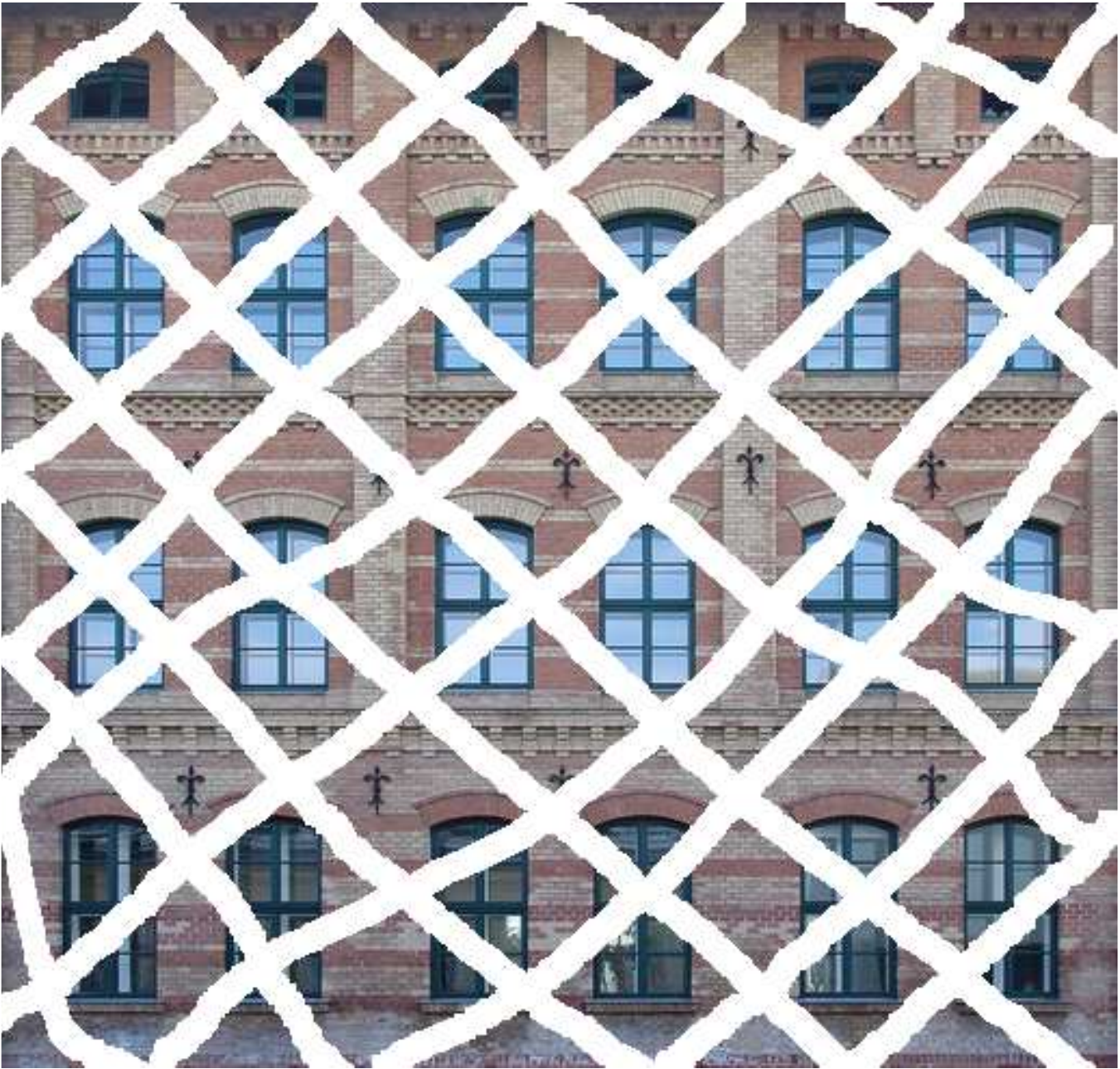}
\end{subfigure}
			\begin{subfigure}[b]{0.16\linewidth}
	\includegraphics[width=\linewidth]{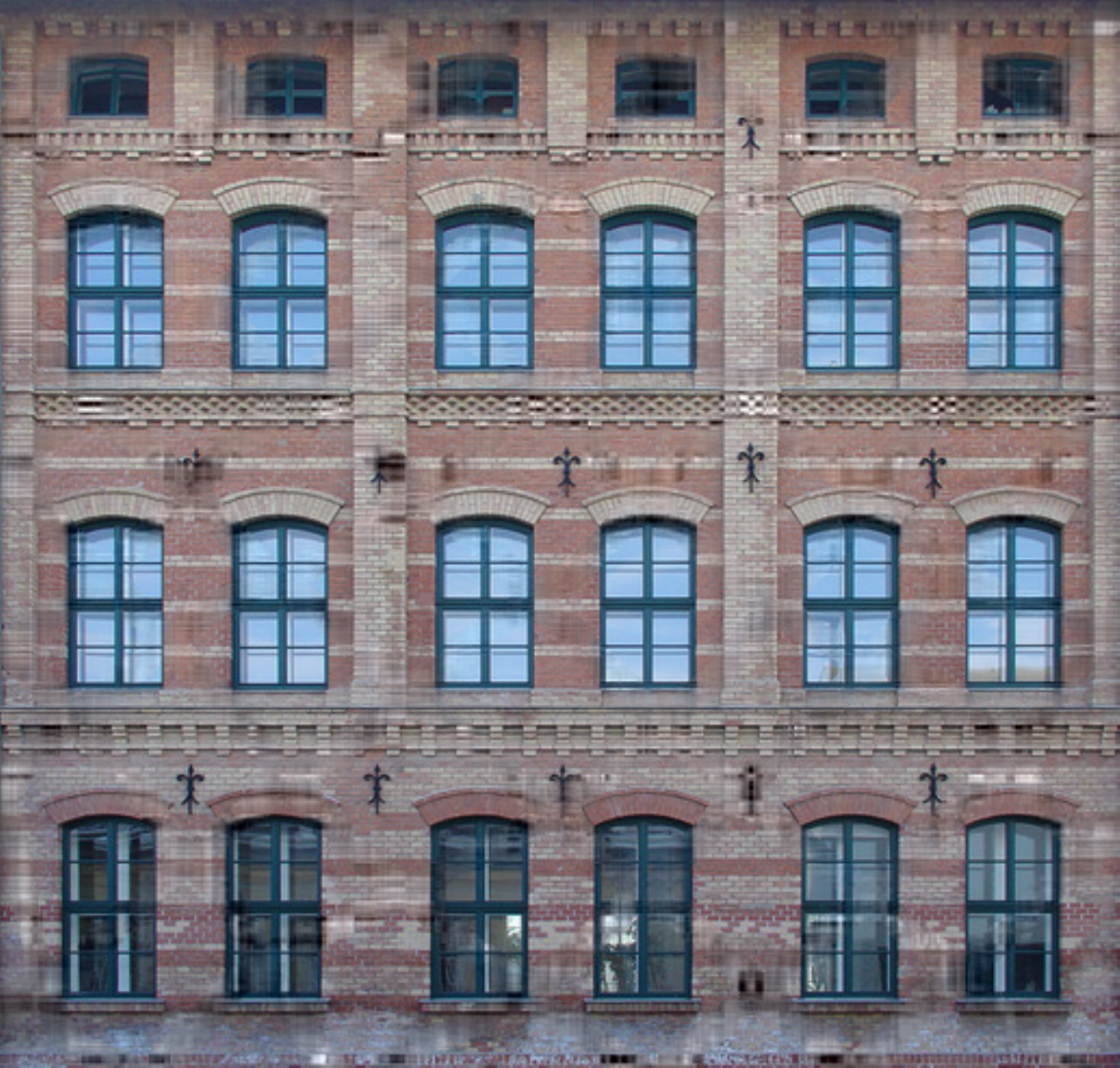}
\end{subfigure}
	\begin{subfigure}[b]{0.16\linewidth}
	\includegraphics[width=\linewidth]{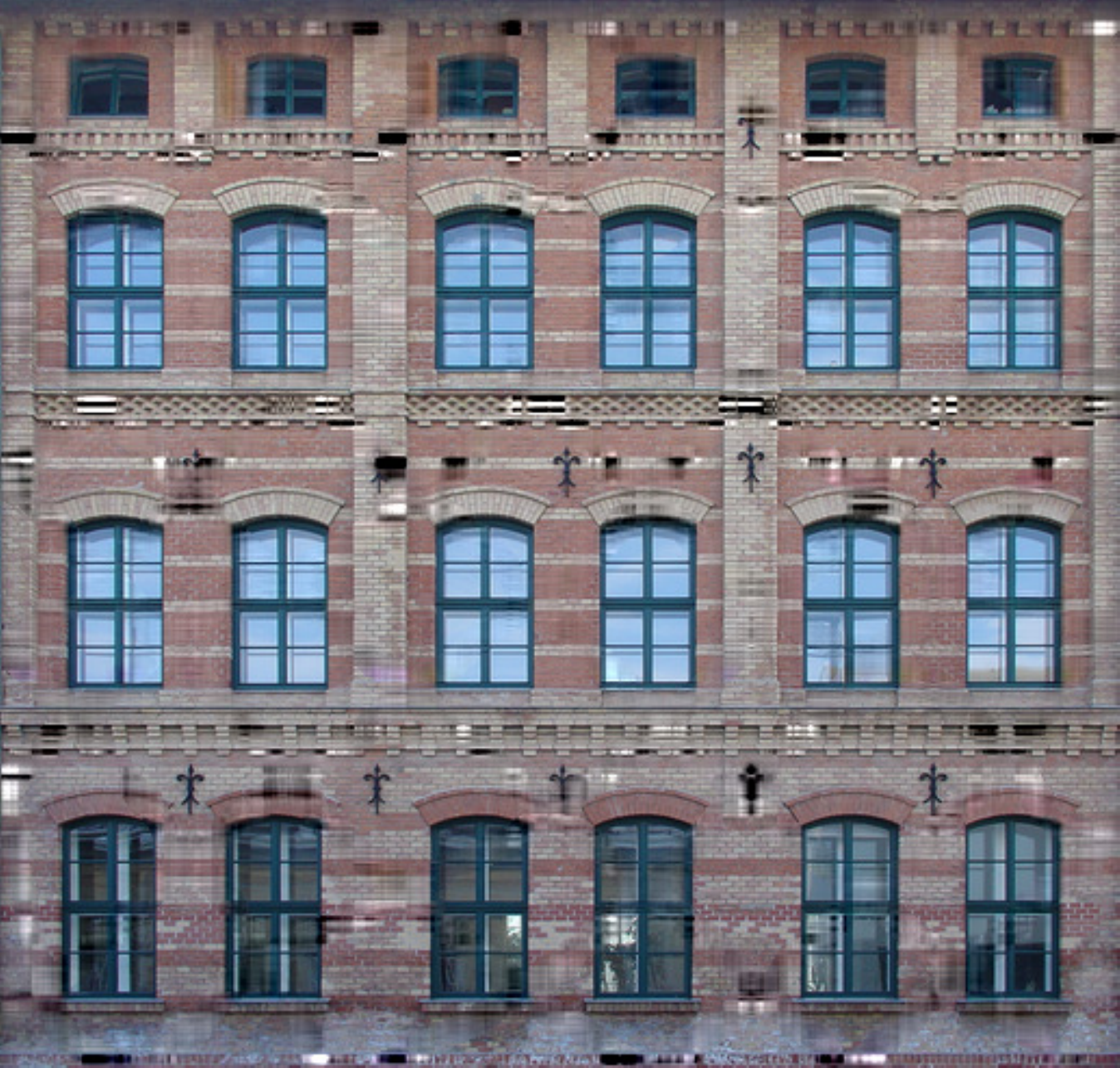}
\end{subfigure}
	\begin{subfigure}[b]{0.16\linewidth}
	\includegraphics[width=\linewidth]{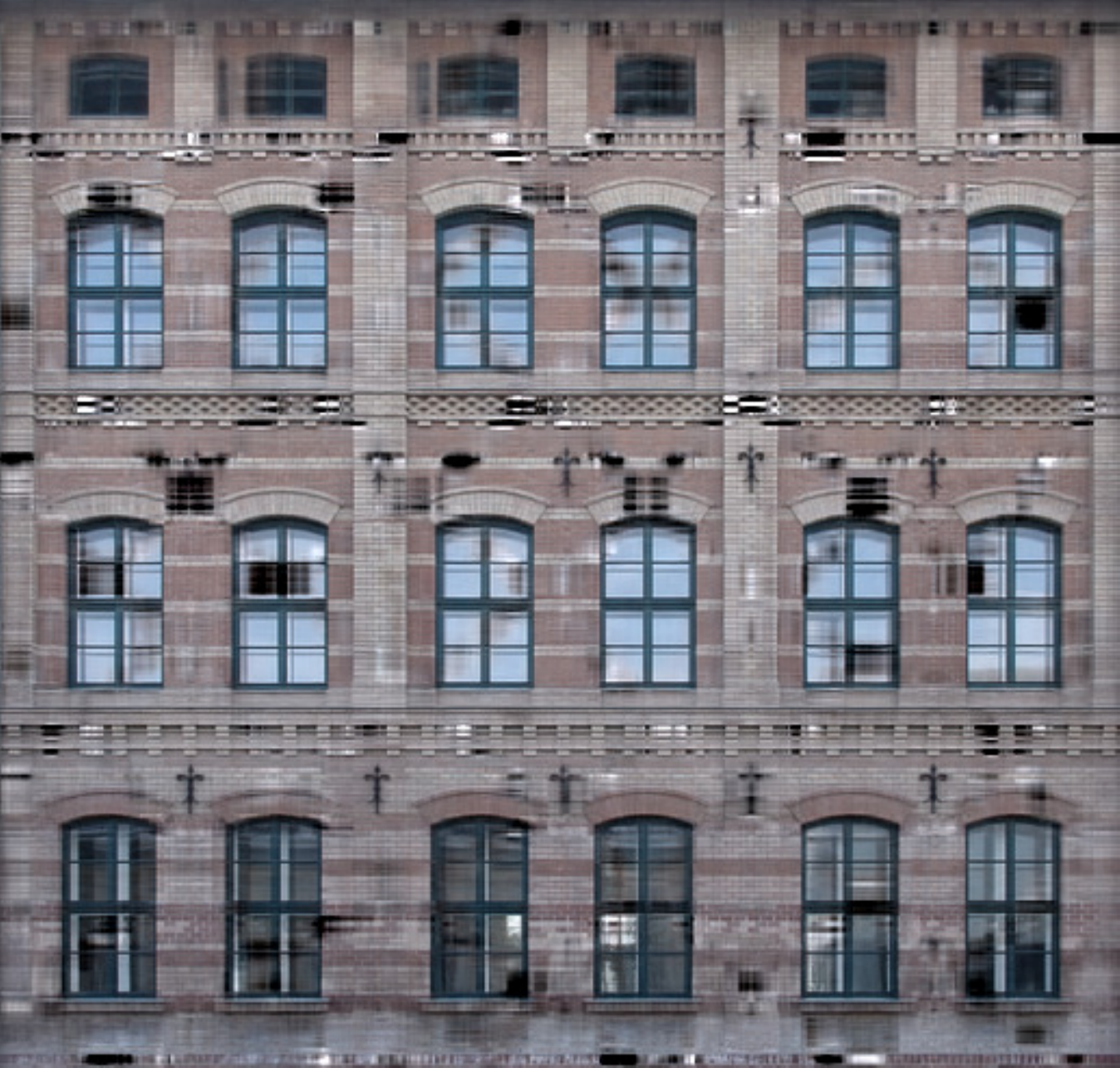}		
		
         \end{subfigure}\\

            \begin{subfigure}[b]{0.16\linewidth}
						\includegraphics[width=\linewidth]{picture/buding1}
					\end{subfigure}
		\begin{subfigure}[b]{0.16\linewidth}
						\includegraphics[width=\linewidth]{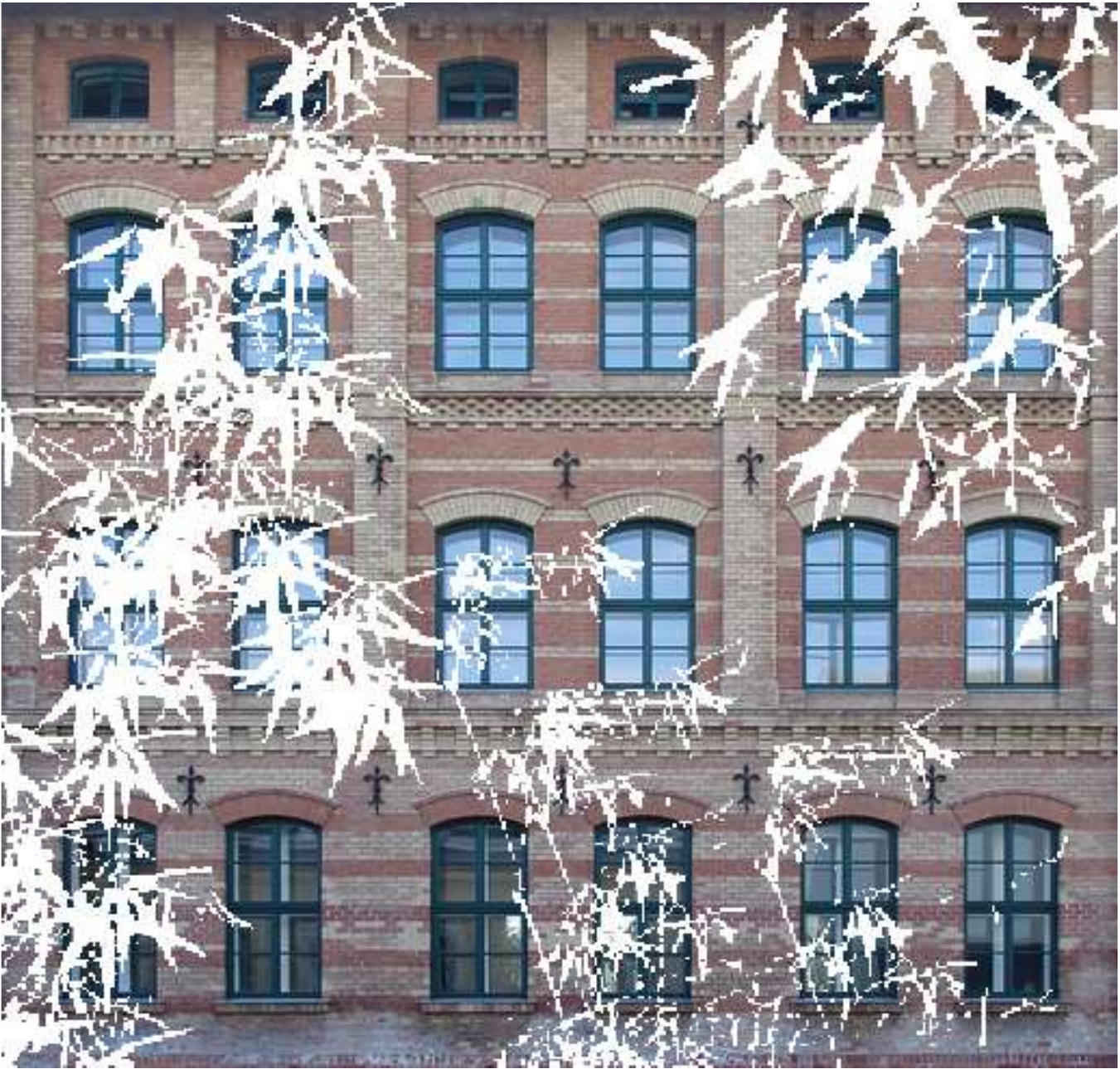}
			
		\end{subfigure}
		\begin{subfigure}[b]{0.16\linewidth}
						\includegraphics[width=\linewidth]{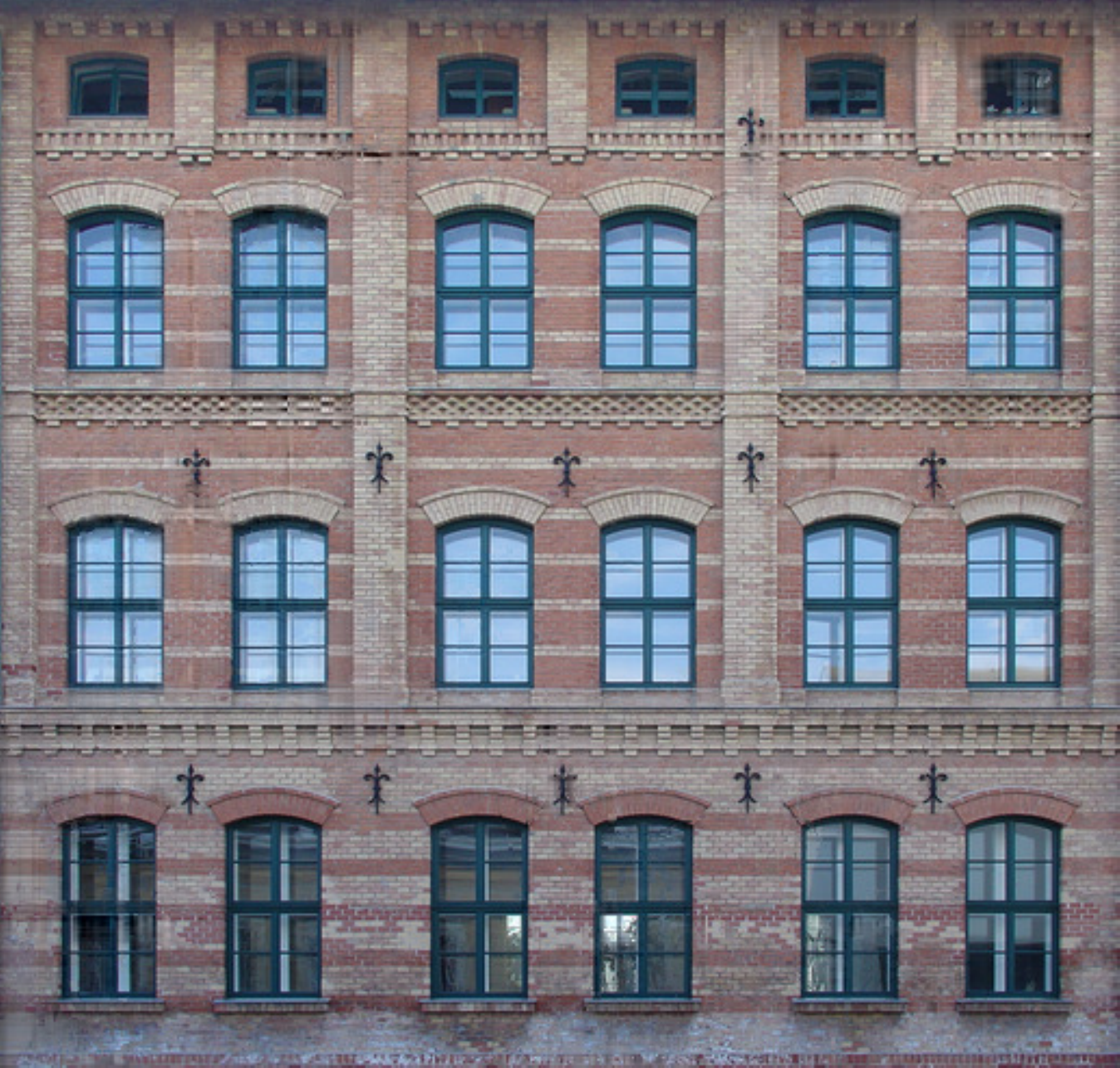}			
		\end{subfigure}    	
		\begin{subfigure}[b]{0.16\linewidth}
						\includegraphics[width=\linewidth]{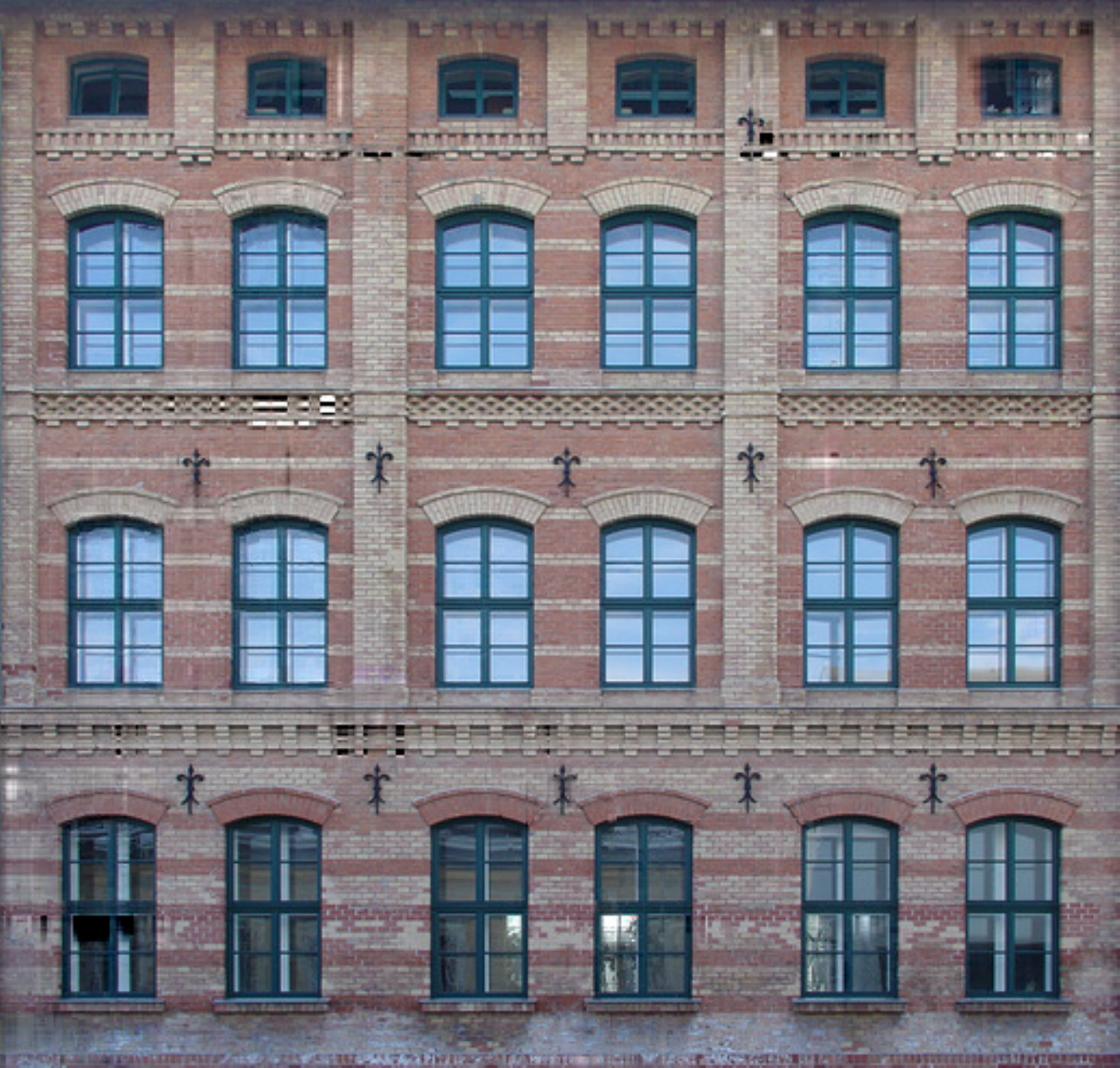}
			
		\end{subfigure}
		\begin{subfigure}[b]{0.16\linewidth}
						\includegraphics[width=\linewidth]{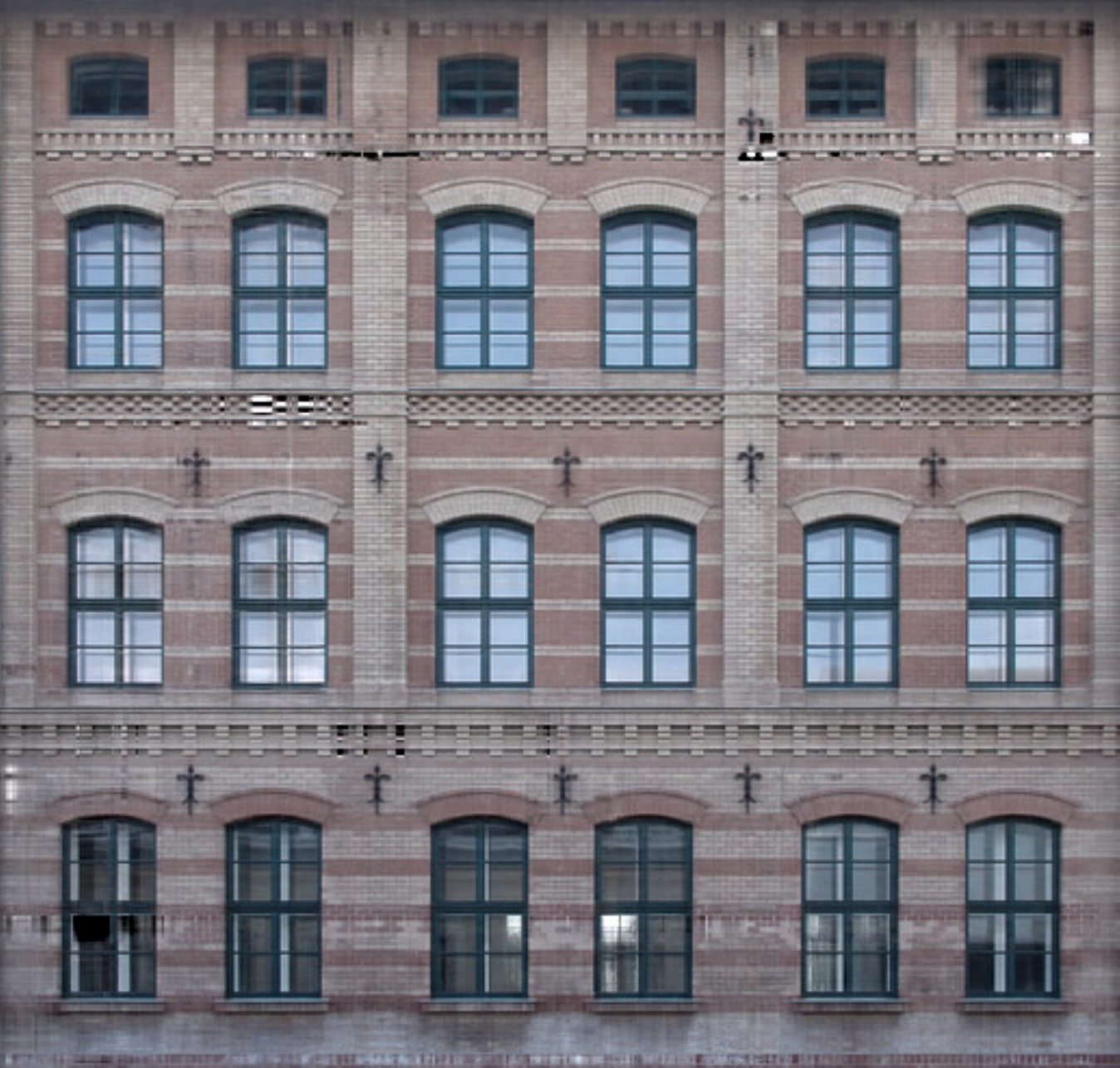}		
		\end{subfigure}
\\
       	\begin{subfigure}[b]{0.16\linewidth}
						\includegraphics[width=\linewidth]{picture/buding1}
			\caption{Original}
		\end{subfigure}
		\begin{subfigure}[b]{0.16\linewidth}
						\includegraphics[width=\linewidth]{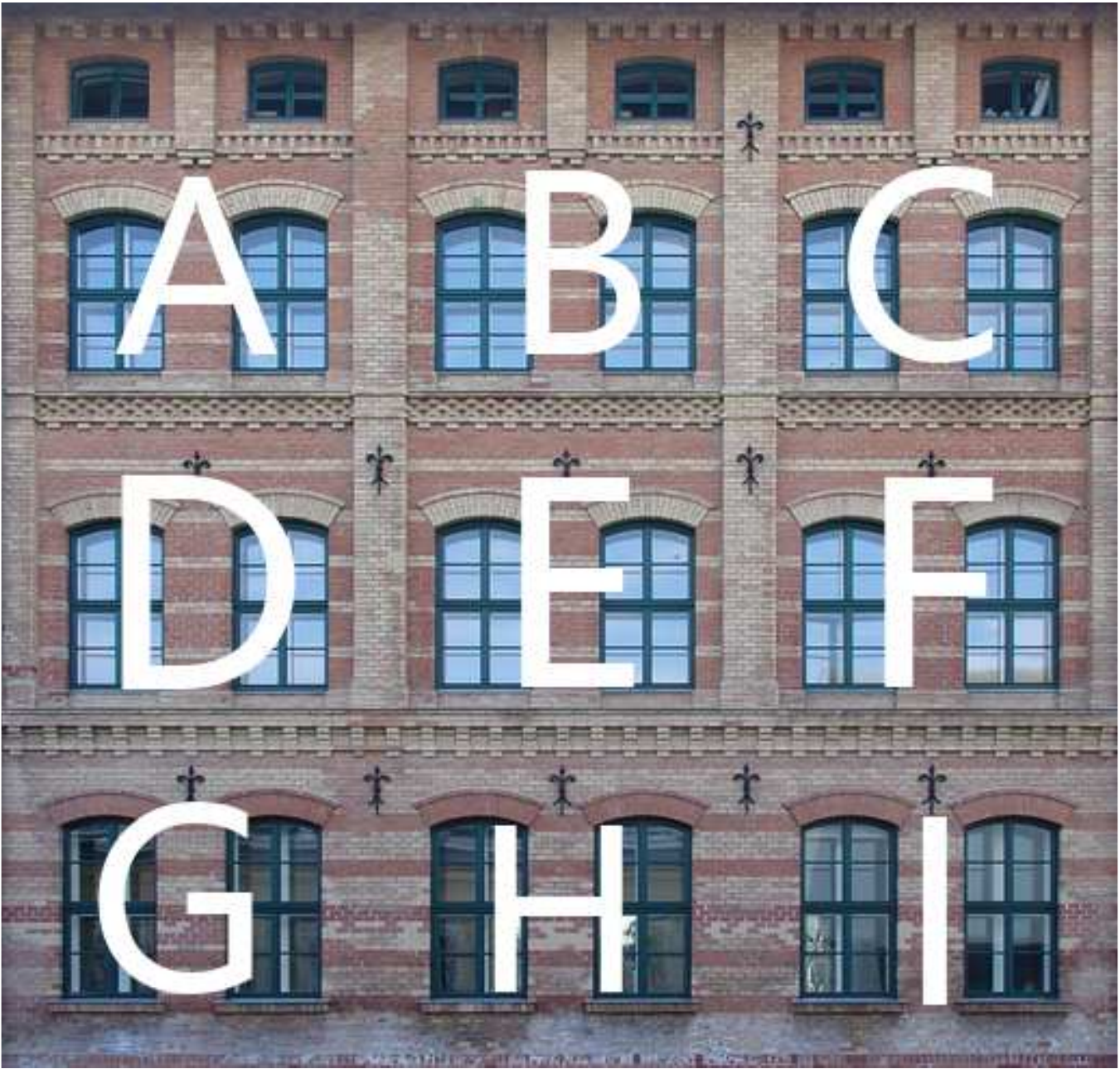}
			\caption{Observation}
		\end{subfigure}
		\begin{subfigure}[b]{0.16\linewidth}
						\includegraphics[width=\linewidth]{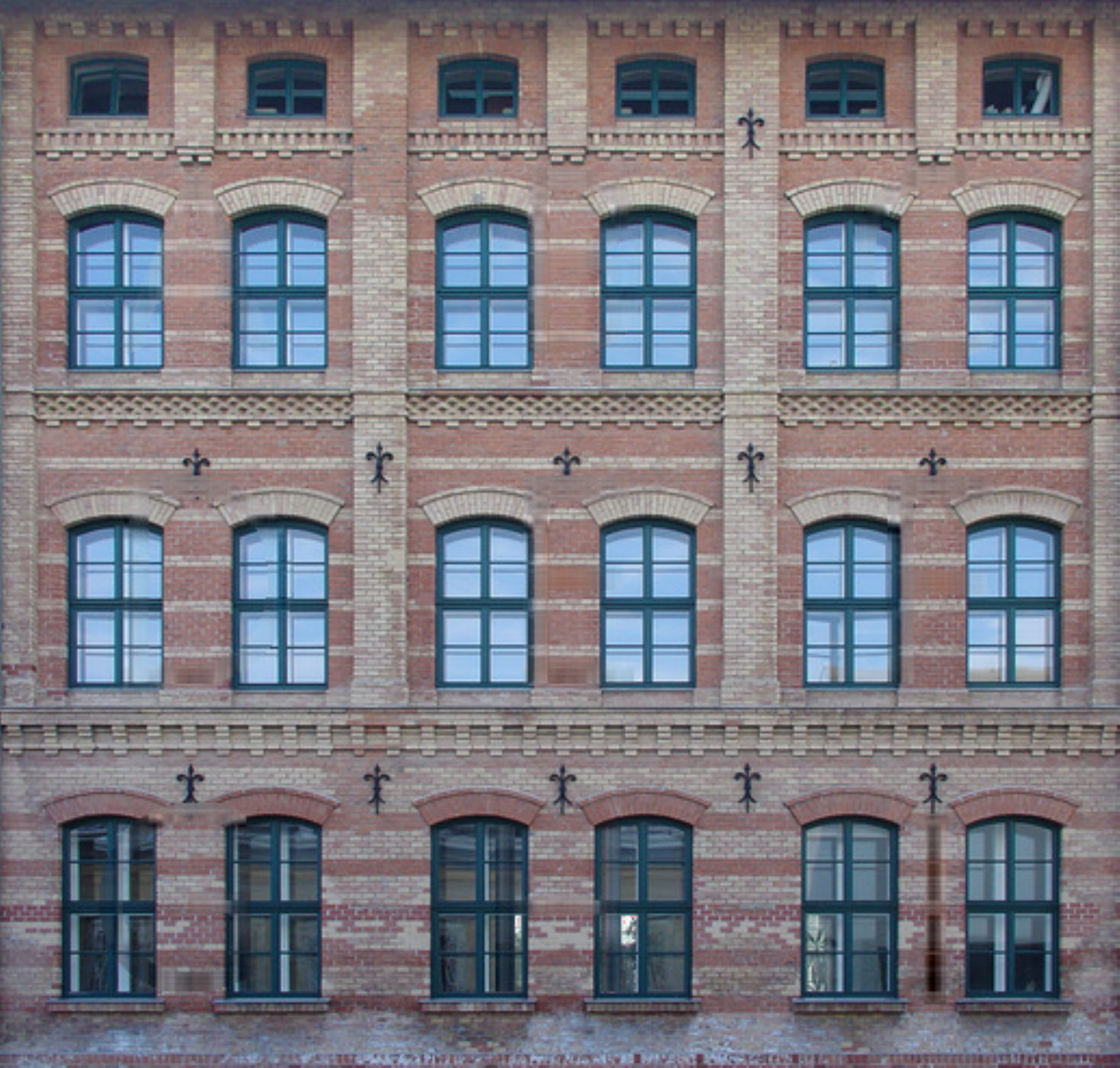}
			\caption{MTRTC}
		\end{subfigure}    	
		\begin{subfigure}[b]{0.16\linewidth}
						\includegraphics[width=\linewidth]{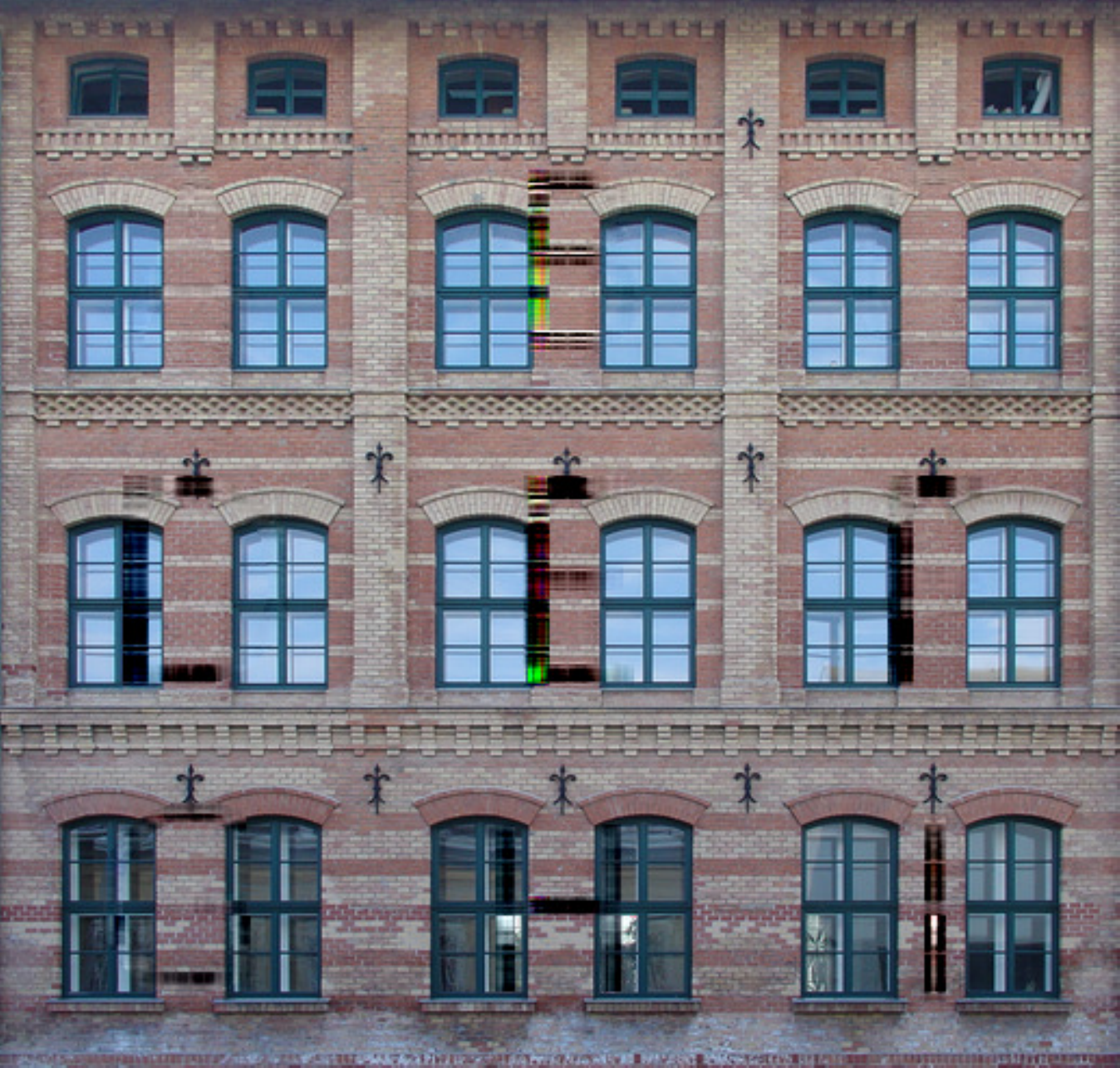}
			\caption{TCTF}
		\end{subfigure}
		\begin{subfigure}[b]{0.16\linewidth}
					\includegraphics[width=\linewidth]{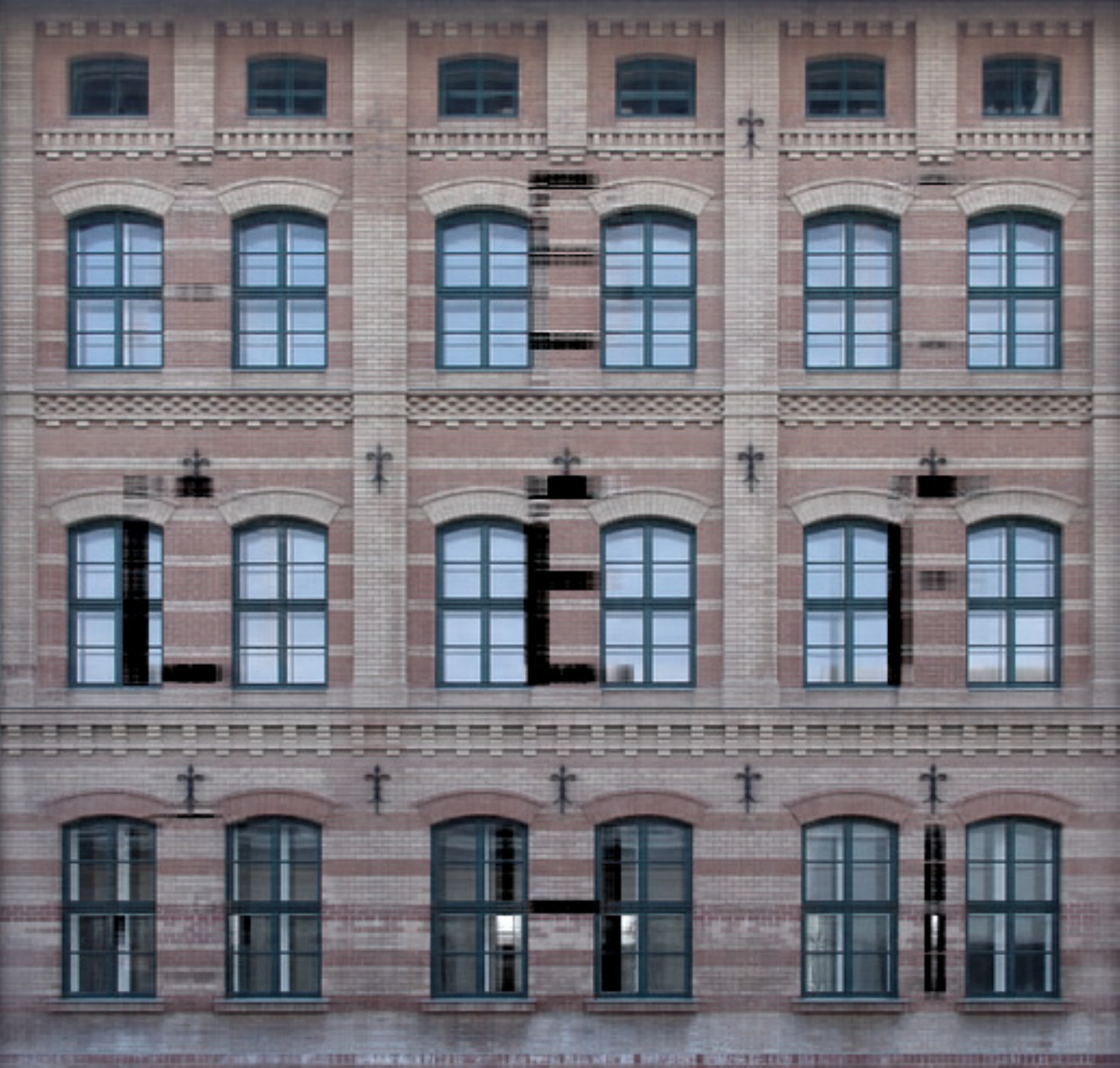}
			\caption{TMac}
		\end{subfigure}
\centering
	\caption{ Recovery performance comparison on the three masked images}
	\label{fig:building}	
\end{figure*}

\begin{table}
\caption{Comparison on the PSNR and the RSE by MTRTC, TCTF and TMac}
\label{tab:building}
\centering
		\begin{tabular}{|c|cc|cc|cc|}
		\hline
			{\multirow{2}{*}{\diagbox{Mask}{Method}}}&\multicolumn{2}{c|}{MTRTC}&\multicolumn{2}{c|}{TCTF} & \multicolumn{2}{c|}{TMac}
			\\\cline{2-7} &  PSNR     & RSE
			&PSNR         & RSE
			&PSNR         & RSE
			\\   \hline
			Grid & {\bf 25.26 }&  {\bf 0.1048 } &     22.57 &     0.1429  &     20.31 &    0.1854   \\
			Leaves & {\bf 30.16 } & {\bf 0.0596 }  &  28.69  &    0.0706  &   25.91  &   0.0972 \\
			Letters & {\bf 30.82 } & {\bf 0.0553 } &   23.00  &     0.1359  &  21.70 & 0.1579 \\
			\hline
		\end{tabular}
\end{table}

\iffalse
\subsection{MRI Volume Dataset}
The resolution of the MRI volume dataset\footnote{\url{http://www.bic.mni.mcgill.ca/ServicesBrainWeb/HomePage}} is of size $ 217 \times 181 $ with 181 slices and we pick the first $ 100 $ slices. In NTD, TMac , LTRTC and T-LTRTC, we set the initial Tucker rank to be $ \left( 40,40,20\right)  $, the initial tubal rank $ \left( 40,\cdots,40\right) $ in TCTF. We set the maximum iteration steps of all algorithms to $ 500 $ steps. It can be seen from Figure \ref{fig:image_MRI} that the TCTF restored the most noisy picture, so its restoration effect is the worst. NTD and TMac are not ideal for recovery where the pixels differ greatly. The pictures recovered by T-LTRTC \textcolor{red}{and LTRTC} are relatively good. At the same time, in order to see more clearly the effect of various algorithms to restore pictures, we also list the pictures of various algorithms to restore pictures minus the original pictures. \textcolor{red}{For better visualization, we add 0.5 to the pixel.} We can clearly see that the picture corresponding to T-LTRTC and LTRTC have almost \textcolor{red}{no white spots???no outline of the original image}, indicating the \textcolor{red}{least???best} recovery effect.

From the histogram \ref{fig:bar_MRI}, we can know more clearly that the four methods of restoring the picture effect are T-LTRTC, LTRTC, TMac, NTD, TCTF. LTRTC's PSNR value is improved by 35.47\%, 19.67\%, and 176.15\% on the basis of NTD, TMac, and TCTF.
\fi

\subsection{Video Simulation}
We evaluate our proposed methods MTRTC and ST-MTRTC on the widely
used YUV Video Sequences\footnote{\url{http://trace.eas.asu.edu/yuv/}}. Each sequence contains at least 150 frames and we pick the first $ 60 $ frames. In the experiments, we test our proposed methods and other methods on three videos with $ 144 \times 176 $ pixels. %The frame sizes of the first two videos both are $ 144 \times 176 $ pixels and that of the last one is $ 288\times 352 $ pixels.
We test the videos with random missing data of sampling ratio $p=0.3$.  We set the initial multi-tubal rank $ (r_{u}^{l})^0=10, u \in  {\bf [2]}, (r_{3}^{l})^0=60, l \in {\bf [n_u]} $ in MTRTC and ST-MTRTC,  the initial tubal rank $ (30,30,30) $ in TCTF and the initial Tucker rank $ (60,60,10) $ in TMac. In experiments, the maximum iteration number is set to be 800 and the termination precision $ \varepsilon $ is set to be 1e-5.

The data between two adjacent frames of the video usually have not drastic change.
To detect such stability, we calculate the data pairs of the corresponding positions between two adjacent frames. The difference for two adjacent frames of the video slots  ($ k $ and $ k+1 $) is defined as
\begin{equation*}
frame(i,j,k)=| { C}_3^k(i,j)-{ C}_3^{k+1}(i,j) |.
\end{equation*}
The smaller the $ frame(i, j, k) $ is, the more stable the data  between two adjacent frames of the video at frame $ k $ is. By computing the normalized difference values between two adjacent frames, we measure the stability between two adjacent frames of the video at frame $ k $ as
\begin{small}
$$
\Delta gap(i, j, k)=\frac{|{ C}_3^k(i,j)-{ C}_3^{k+1}(i,j)|}{\max\limits _{1 \leqslant i \leqslant n_1,1 \leqslant j \leqslant n_2, 1 \leqslant k \leqslant n_3-1}|{ C}_3^k(i,j)-{ C}_3^{k+1}(i,j)|}.
$$
\end{small}
Here $ \max\limits _{1 \leqslant i \leqslant n_1,1 \leqslant j \leqslant n_2, 1 \leqslant k \leqslant n_3-1}|{ C}_3^k(i,j)-{ C}_3^{k+1}(i,j)| $ means the maximal gap between any two adjacent frames of the video.
We plot the CDF of $ \Delta frame(i, j, k) $ in Figure \ref{fig:CDF-video}. The X-axis represents the normalized difference values between two adjacent frames slots, i.e., $ \Delta frame(i, j, k) $. The Y-axis represents the cumulative probability. We can see that the value $ \Delta frame(i, j, k)<0.6 $ is more than 80\%. These results indicate that the temporal stability exists in the real video data. Hence we apply ST-MTRTC  in video inpainting with Toeplitz matrix being a temporal constrained matrix $H$. Furthermore, $ \beta_{1} = \beta_{2} = 0 $, which mean that  $F$ and $G$ are zero matrices.
\begin{figure}[htbp]
	\centering
	\begin{subfigure}[b]{1\linewidth}
		%		\subfloat[hall]{
		\begin{subfigure}[b]{0.325\linewidth}
			\centering
			\includegraphics[width=\linewidth]{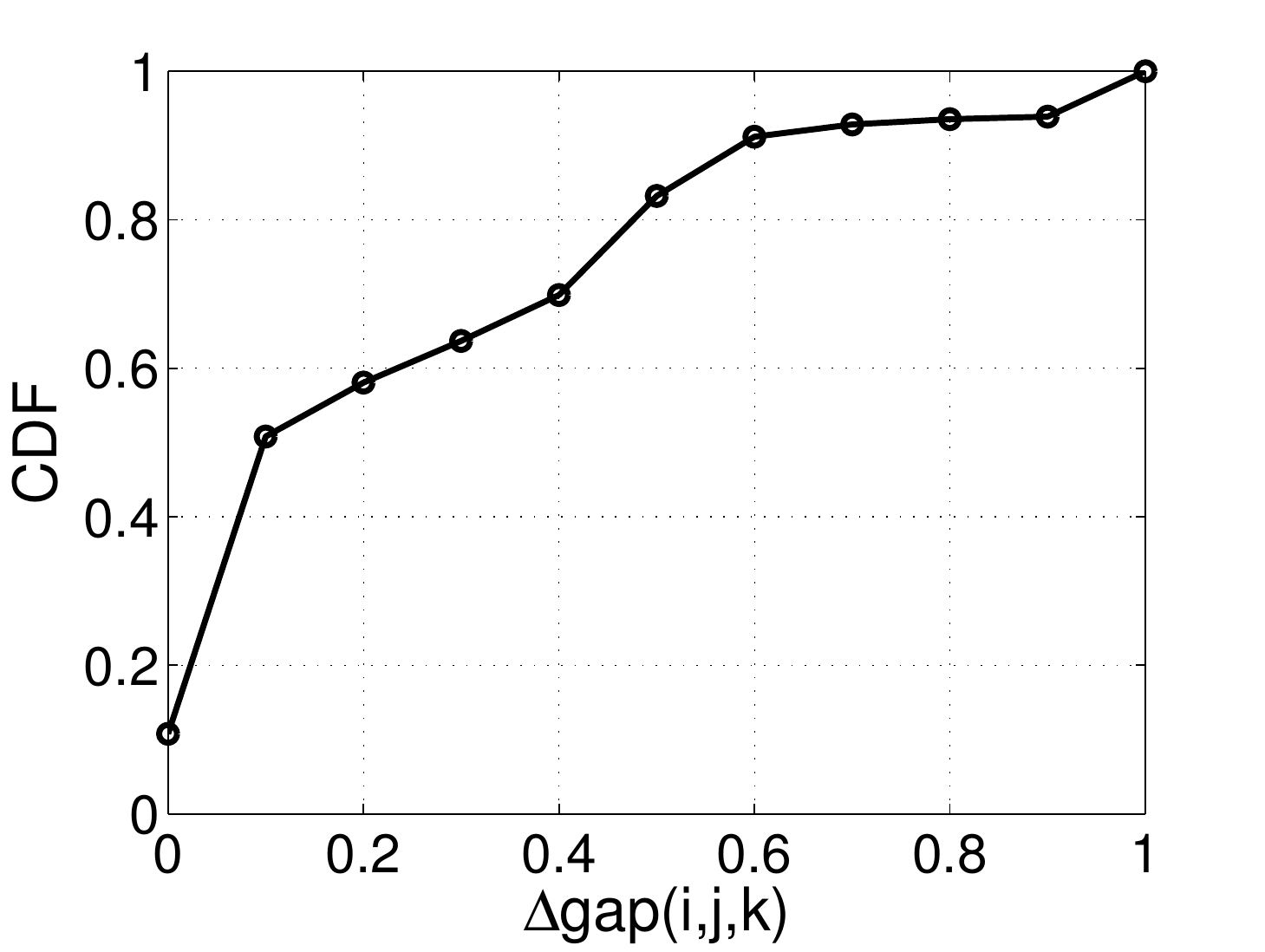}
			\caption{Mother}
		\end{subfigure}
		%		\subfloat[akiyo]{    	
		\begin{subfigure}[b]{0.325\linewidth}
			\centering
			\includegraphics[width=\linewidth]{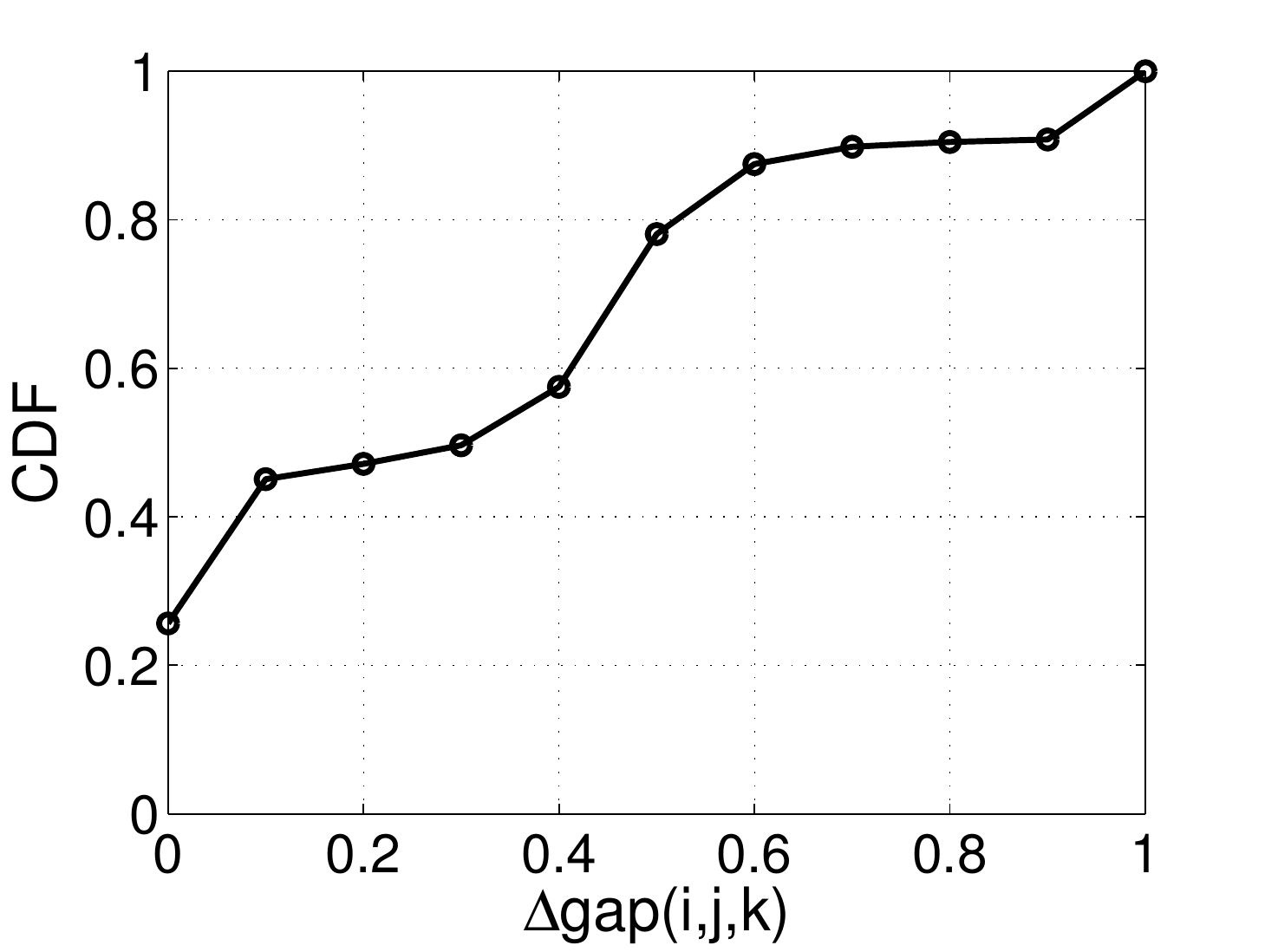}
			\caption{Container}
		\end{subfigure}
		%		\subfloat[waterfall]{
		\begin{subfigure}[b]{0.325\linewidth}
			\centering
			\includegraphics[width=\linewidth]{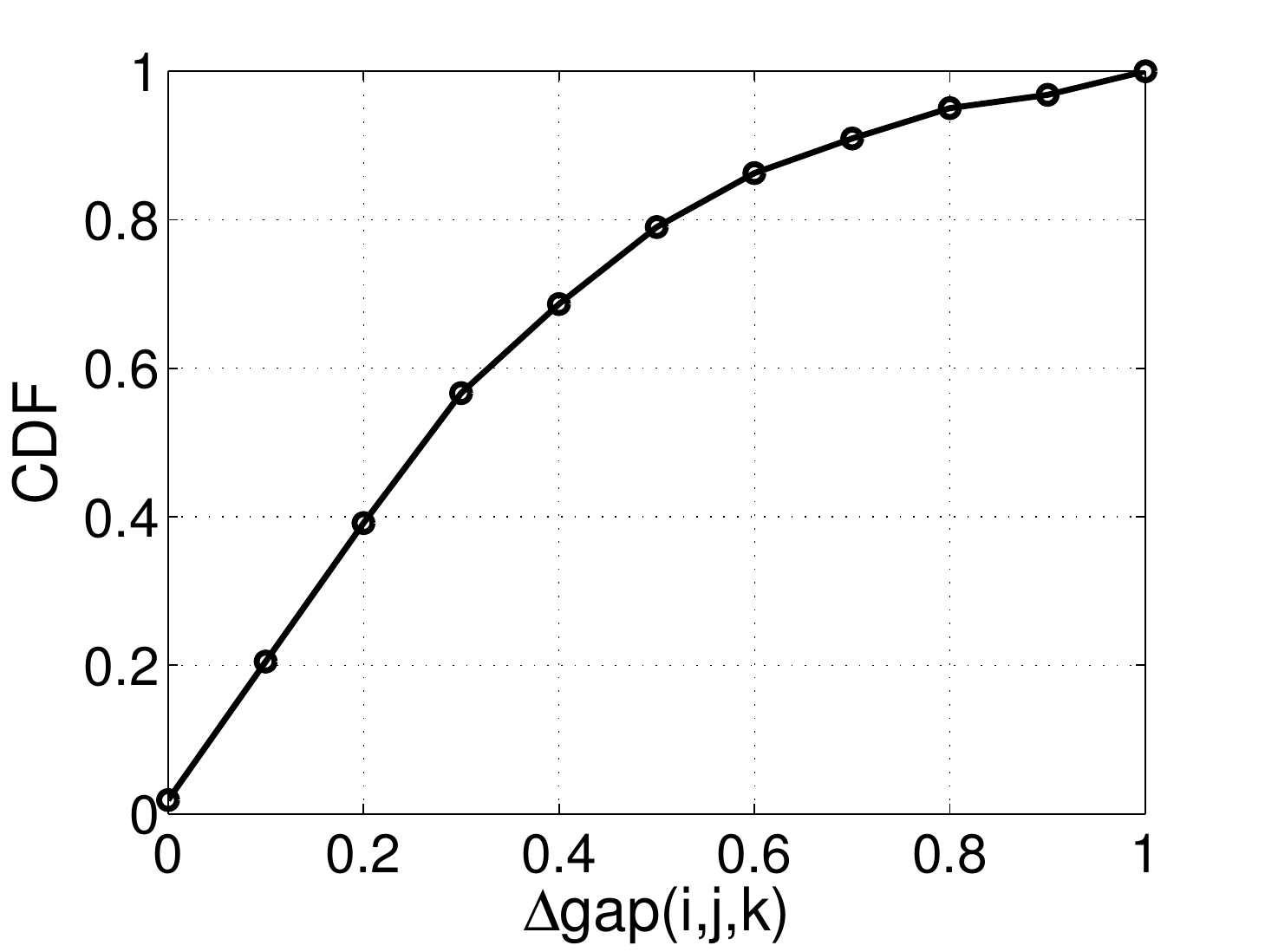}
			\caption{Bridge}
		\end{subfigure}
	\end{subfigure}
	\vfill
	\caption{An empirical study of three sets of real video data}
	\label{fig:CDF-video}
\end{figure}

\begin{figure}[htbp]
	\centering
	\begin{subfigure}[b]{1\linewidth}
%		\subfloat[Original]{
			\begin{subfigure}[b]{0.161\linewidth}
				\centering
				\includegraphics[width=\linewidth]{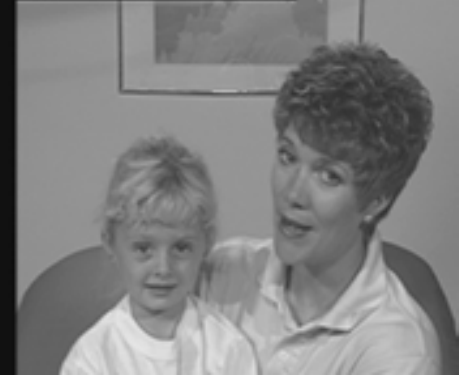}\vspace{0pt}
				\includegraphics[width=\linewidth]{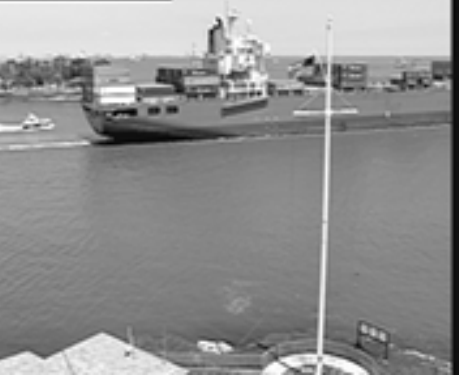}\vspace{0pt}
				\includegraphics[width=\linewidth]{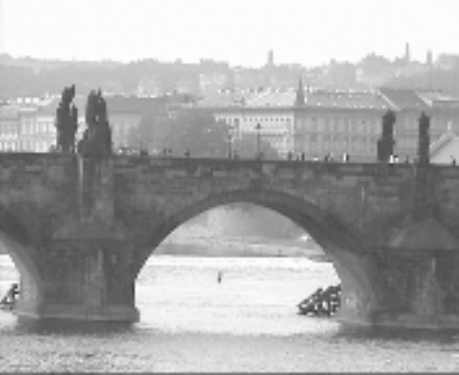}
				\caption{\tiny{Original}}
		\end{subfigure}
%		\subfloat[Observation]{    	
			\begin{subfigure}[b]{0.161\linewidth}
				\centering
				\includegraphics[width=\linewidth]{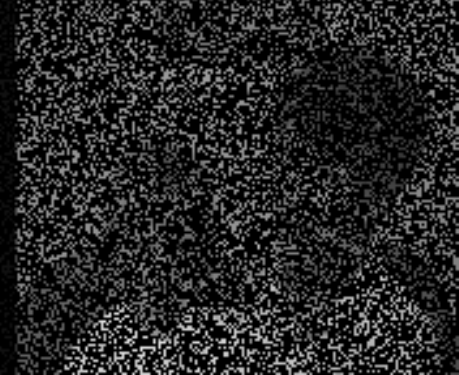}\vspace{0pt}
				\includegraphics[width=\linewidth]{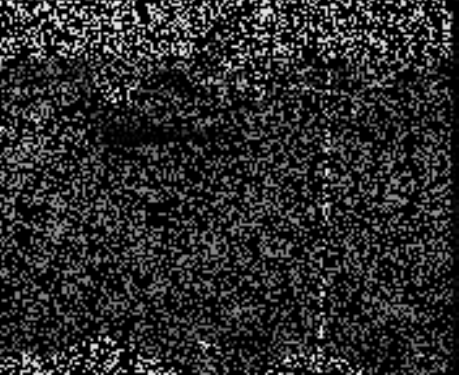}\vspace{0pt}
				\includegraphics[width=\linewidth]{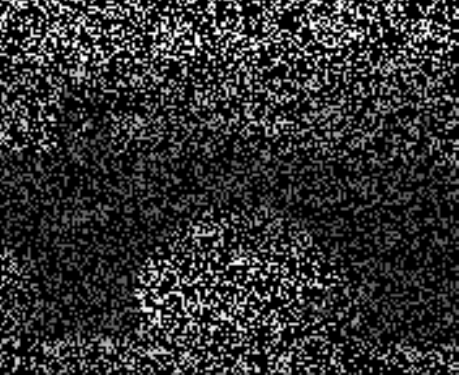}
				\caption{\tiny{Observation}}
		\end{subfigure}
%		\subfloat[TRTF]{
			\begin{subfigure}[b]{0.161\linewidth}
				\centering
				\includegraphics[width=\linewidth]{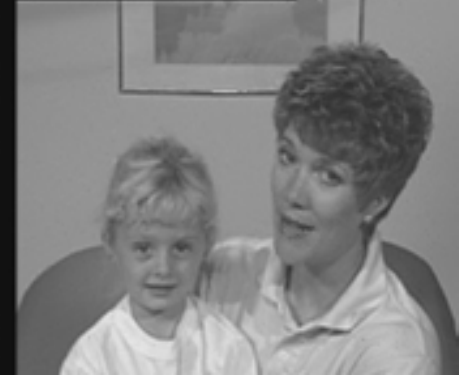}\vspace{0pt}
				\includegraphics[width=\linewidth]{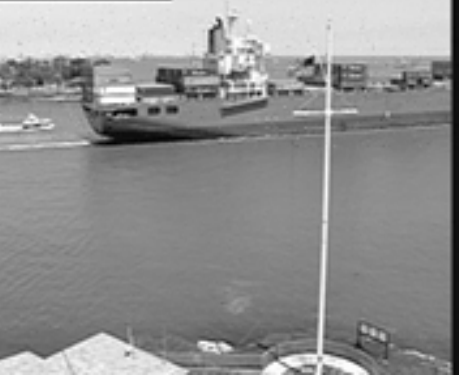}\vspace{0pt}
				\includegraphics[width=\linewidth]{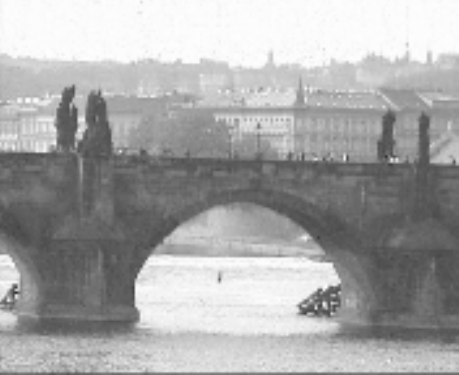}
				\caption{\tiny{MTRTC}}
		\end{subfigure}
%			\subfloat[ST-TRTF]{
		\begin{subfigure}[b]{0.161\linewidth}
			\centering
			\includegraphics[width=\linewidth]{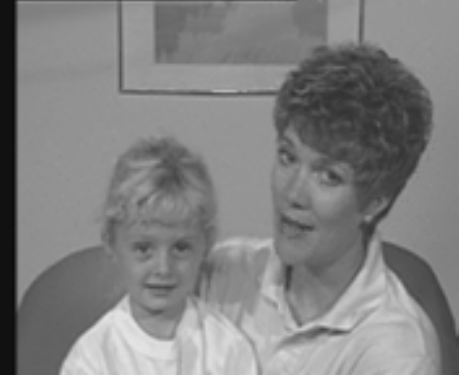}\vspace{0pt}
			\includegraphics[width=\linewidth]{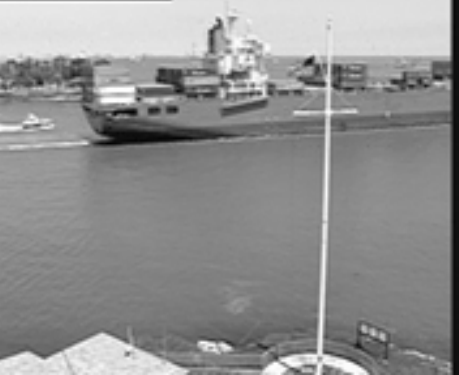}\vspace{0pt}
			\includegraphics[width=\linewidth]{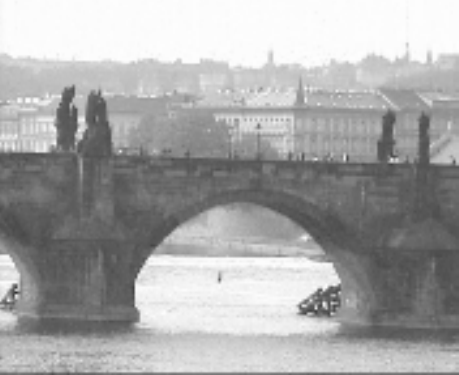}
			\caption{\tiny{ST-MTRTC}}
	    \end{subfigure}
%		\subfloat[TCTF]{
			\begin{subfigure}[b]{0.161\linewidth}
				\centering
				\includegraphics[width=\linewidth]{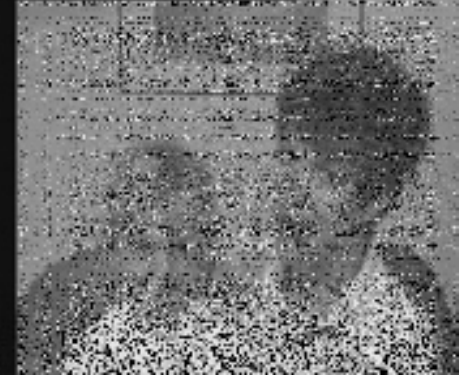}\vspace{0pt}
				\includegraphics[width=\linewidth]{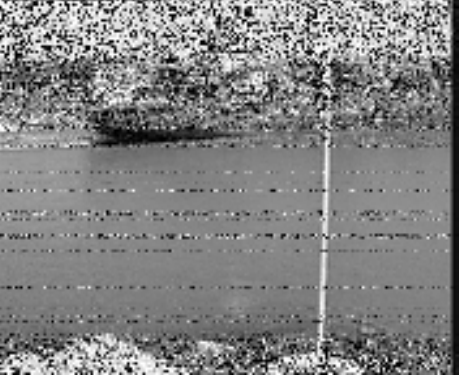}\vspace{0pt}
				\includegraphics[width=\linewidth]{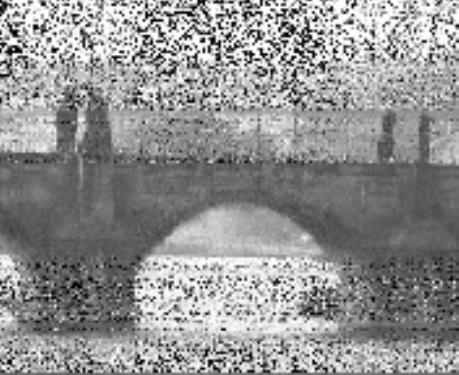}
				\caption{\tiny{TCTF}}
		\end{subfigure}
%		\subfloat[TMac]{
			\begin{subfigure}[b]{0.161\linewidth}
				\centering
				\includegraphics[width=\linewidth]{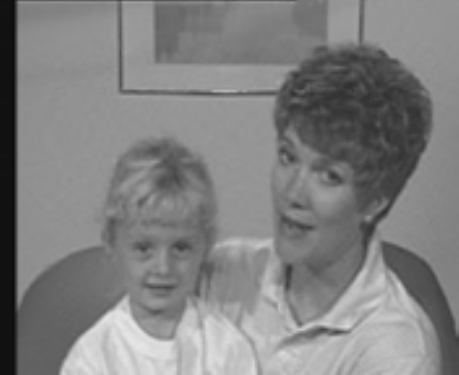}\vspace{0pt}
				\includegraphics[width=\linewidth]{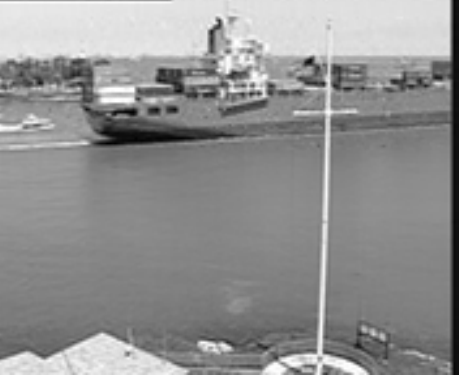}\vspace{0pt}
				\includegraphics[width=\linewidth]{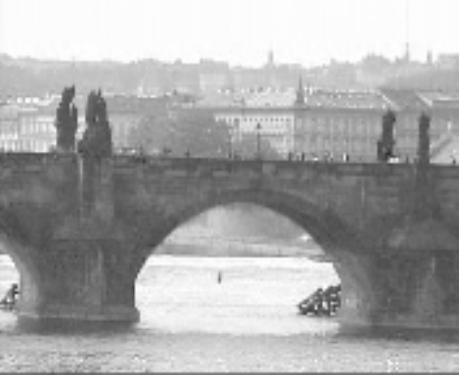}
				\caption{\tiny{TMac}}
			\end{subfigure}
		
	\end{subfigure}
	\caption{Recovery performance comparison on the three videos}
	\label{fig:test-video}
\end{figure}

{\scriptsize
\begin{table*}
	\centering
	\caption{Comparison on the PSNR, the RSE and the running time on the three videos}
\begin{tabular}{|l|ccc|ccc|ccc|}
			\hline
			{\multirow{2}{*}{\diagbox{\scriptsize Method}{\scriptsize Video}}} & \multicolumn{3}{c|}{ Mother} &\multicolumn{3}{c|}{ Container} & \multicolumn{3}{c|}{Bridge}
			\\\cline{2-10}
            & {\scriptsize PSNR }     &{\scriptsize RSE}     &  time
			&{\scriptsize PSNR  }       &{\scriptsize RSE}         & time
			&{\scriptsize PSNR}         &{\scriptsize RSE}         &time
			   \\  \hline		
{MTRTC} & 37.02 &0.024 & \textbf{34.14} & 40.53 & 0.016 & \textbf{50.46} & 34.79& 0.026 & 36.74  \\
{ST-MTRTC} & \textbf{37.79} & \textbf{0.022} & 46.33 & \textbf{42.58} & \textbf{0.012} & 60.63 & \textbf{35.55} & \textbf{0.024} & \textbf{33.37 } \\
{TCTF}  & 14.19  & 0.338  & 94.59  & 13.11  & 0.367  & 95.82  & 11.93  & 0.357  & 93.45  \\
{TMac}  & 35.92  & 0.028  & 39.11  & 34.45  & 0.032  & 77.77  & 33.88  & 0.028  & 38.84  \\
			\hline
		\end{tabular}
	\label{tab:test-video}
\end{table*}}

\par  Figure \ref{fig:test-video} shows the 18th frame of the three videos. Table \ref{tab:test-video} displays  the numerical results, which show that  MTRTC performs better than TCTF and TMac on PSNR and RSE. Especially for the container video, PSNR of MTRTC has increased by $ 209.53\% $ and $ 26.36\% $ over TCTF and TMac, respectively. On consumed time, MTRTC also takes the least time to recover the three videos among all algorithms. %For the third video with larger pixels, the consumed time of MTRTC is reduced by $ 257\% $ and $ 182\% $ over TCTF and TMac, respectively.

Numerical results displayed in Table \ref{tab:test-video} show that ST-MTRTC performs better than MTRTC on PSNR and RSE. The consumed time of ST-MTRTC is similar to MTRTC. Even in container video, the PSNR returned by ST-MTRTC has increased by $ 5.06\% $ over MTRTC. These results indicate that the temporal stability exists in the real video data, which improves the performance of MTRTC.

\iffalse
Now we select container video to compare the efficiency at different sample rate.  From Figure \ref{fig:vedio-akiyo},  the RSE and PSNR of MTRTC are better than that of TCTF and TMac.  Clearly, ST-MTRTC achieves the best on PSNR and RSE in these four methods. The running time of ST-MTRTC and MTRTC and the changing trend of running time with sample rate are basically the same. Although the consumed time of ST-MTRTC increased when the sample rate is less than or equal to $ 0.25 $,  the accuracy of ST-MTRTC is the best. Furthermore,  the running time of ST-MTRTC is the shortest one among all algorithms  when the sample rate is larger than $ 0.25$. All these results demonstrate that ST-MTRTC is the best one.

\begin{figure}[htbp]
	\centering
	\begin{subfigure}[b]{1\linewidth}
		%		\subfloat[RSE]{
		\begin{subfigure}[b]{0.325\linewidth}
			\centering
			\includegraphics[width=\linewidth]{picture/vedio_RSE}
			\caption{RSE}
		\end{subfigure}
		%		\subfloat[PSNR]{    	
		\begin{subfigure}[b]{0.325\linewidth}
			\centering
			\includegraphics[width=\linewidth]{picture/vedio_PSNR}
			\caption{PSNR}
		\end{subfigure}
		%		\subfloat[time]{
		\begin{subfigure}[b]{0.325\linewidth}
			\centering
			\includegraphics[width=\linewidth]{picture/vedio_time}
			\caption{time}
		\end{subfigure}
	\end{subfigure}
	\vfill
	\caption{Comparison between algorithms for the different sampling ratio p.}
	\label{fig:vedio-akiyo}
\end{figure}
\fi
\subsection{Internet Traffic Simulation}

\begin{minipage}[b]{0.3\linewidth}
	\includegraphics[width=1\linewidth]{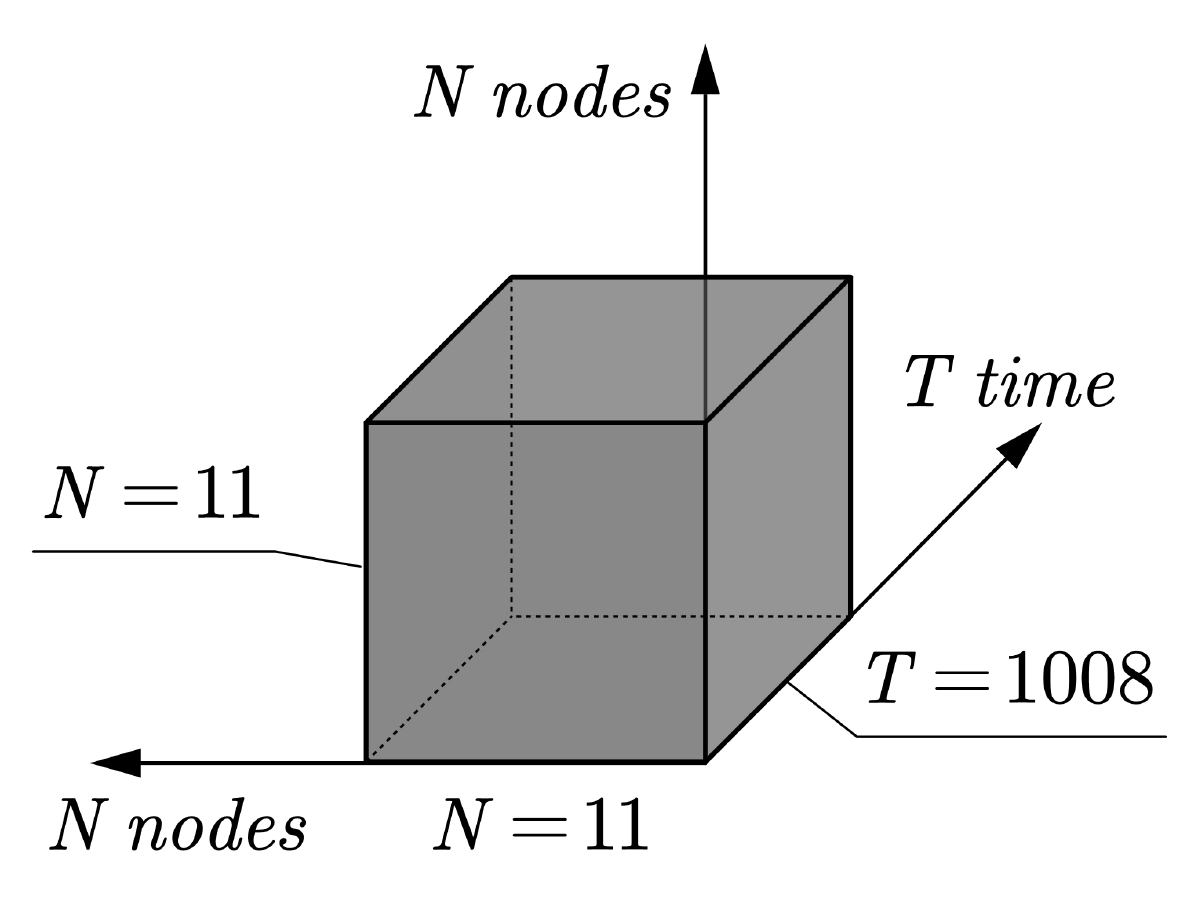}
	%\caption{Traffic tensor model.}
\end{minipage}
\hfill
\begin{minipage}[b]{0.65\linewidth}
We model the traffic data as a third order tensor $ \mathcal{M} \in {\mathbb R^{D \times T \times O}} $. Here $ O $ corresponds to the number of OD pairs with $ O = N \times N $ ($ N $ is the number of nodes in the network), and there are $ D $ days to consider with each day having $ T $ time slots.
\end{minipage}
	
We uses Abilene trace data \cite{TAODC} as an example to illustrate this model. The traffic data are collected between 144 OD pairs in 168 days, and the measurements are made every 5 minutes which corresponds to 288 time slots every day. We use a complete one week traffic data. Therefore, the trace data can be modeled as a third order tensor $ \mathcal{M} \in {\mathbb R^{7 \times 288 \times 144}} $. We use the normalized mean absolute error (NMAE) in the missing values as a metric of the recovered data. The NMAE is defined as follows
$$
\mathrm{NMAE}=\frac{\sum_{(i,j,k) \notin \Omega}\left|\M_{ijk}-\hat{\C}_{ijk}\right|}{\sum_{(i,j,k) \notin \Omega}\left|\M_{i jk}\right|}.
$$
%We set the initial multi-tubal rank $ (r_{u}^{l})^0=10, u \in  {\bf [2]}, (r_{3}^{l})^0=60, l \in {\bf [n_u]} $ in MTRTC and ST-MTRTC,  the initial tubal rank $ (1,\,30) $ in TCTF and the initial tucker rank $ (60,60,10) $ in TMac. We set the maximum iteration steps of all algorithms to 800 steps and the termination precision is set to be 1e-5.

\iffalse
We model the traffic data as a 3-way tensor $ \mathcal{M} \in {\mathbb R^{N \times N \times T}} $. Here `$ N $' corresponds to $ N $ nodes, and we use a complete one week traffic data  with $ 10 $-minute sampling, i.e., $ T=\frac{60}{10} \times 24 \times 7 = 1008 $ time intervals.

\begin{minipage}[b]{0.3\linewidth}
	\includegraphics[width=1\linewidth]{picture/tensor}
	%\caption{Traffic tensor model.}
\end{minipage}
\hfill
\begin{minipage}[b]{0.65\linewidth}
	Two weeks of data are selected randomly in the public traffic trace data Abilene \cite{TAODC}. For the Abilene trace, $ N=11$ and $T =1008 $.
\end{minipage}
\fi
\begin{figure}[htbp]
	\centering
	\begin{subfigure}[b]{1\linewidth}
		%		\subfloat[RSE]{
		\begin{subfigure}[b]{0.493\linewidth}
			\centering
			\includegraphics[width=\linewidth]{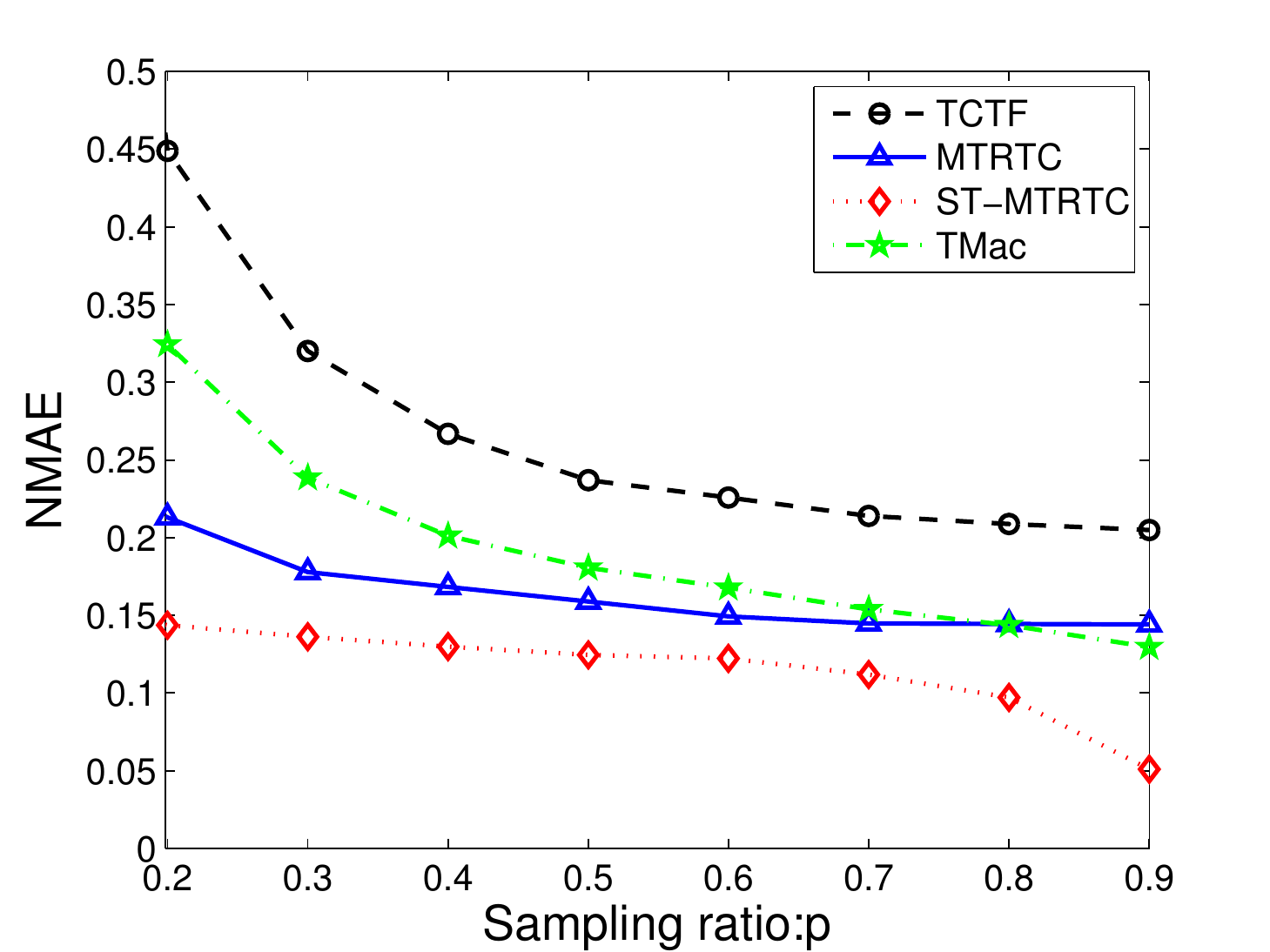}
			\caption{The first week of data}
		\end{subfigure}
		%		\subfloat[Computation time]{    	
		\begin{subfigure}[b]{0.493\linewidth}
			\centering
			\includegraphics[width=\linewidth]{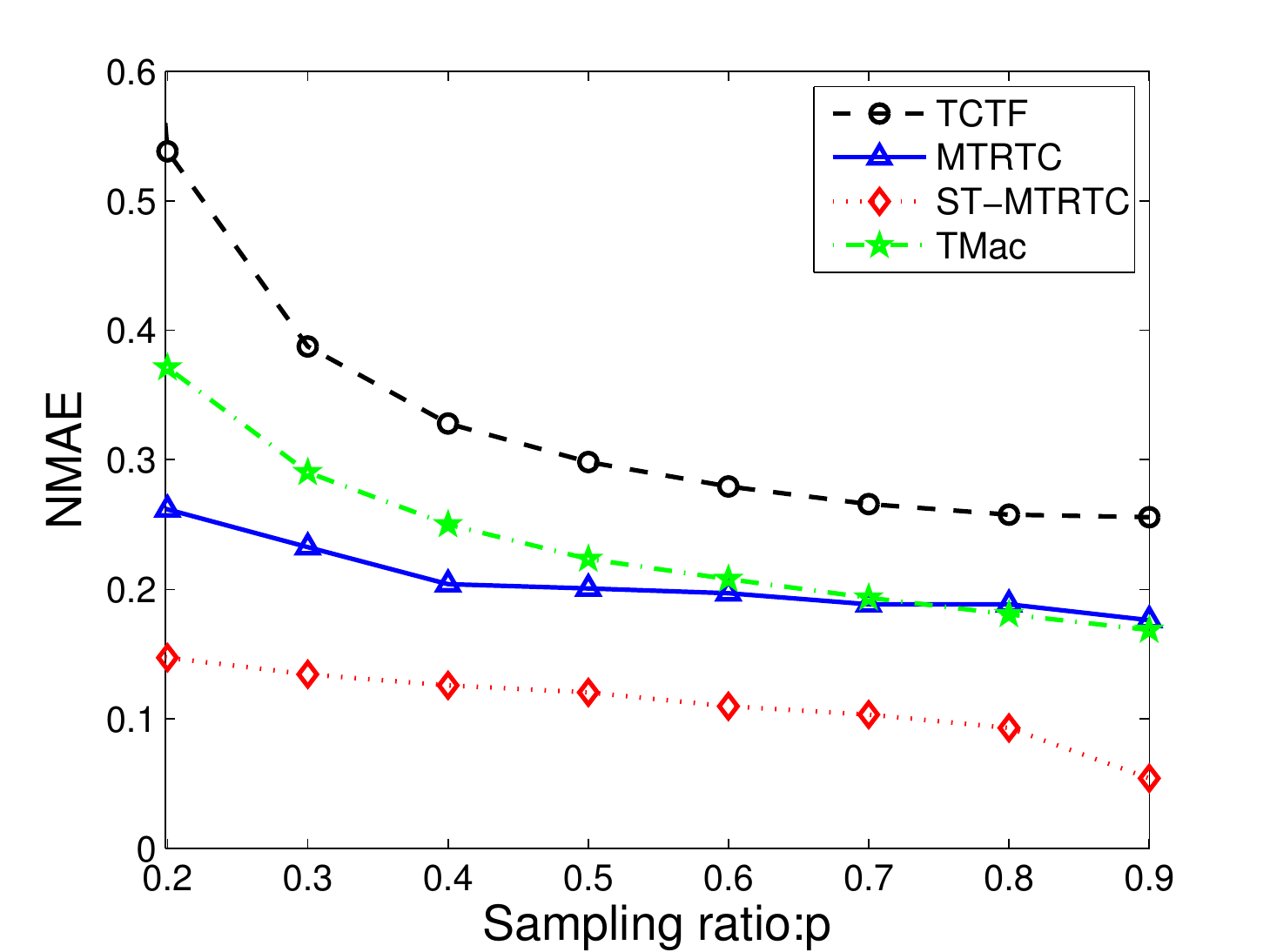}
			\caption{The second week of data}
		\end{subfigure}
	\end{subfigure}
	\vfill
	\caption{Comparison on the NMAE by four methods of  different sampling ratios}
	\label{fig:traffic}
\end{figure}

Figure \ref{fig:traffic} shows the recovered results in Abilene dataset by four algorithms. The X-axis represents the sample rate of data, and the Y-axis represents  RSE. As the sample rate increases, the RSE value gradually decreases.
Among the four methods, ST-MTRTC has the best recovery effect. Note that  ST-MTRTC can still recover lost data with very low error  even if the sample rate is very low. Furthermore,
 MTRTC lags behind ST-MTRTC, which means a spatio-temporal structure in the network traffic data works well.

For further comparison, we illustrate the recovered data for the 139th OD pair of Abilene data. To this end,  we select the first 144 data per day. As shown in Figure \ref{fig:trafcom}, some of data recovered by TCTF and TMac are far from the original data  when the sample rate is lower than $ p = 0.6 $.  However, the data recovered by  ST-MTRTC  fits the original data well. That is,  ST-MTRTC can recover the data of low sample rate with high accuracy.
 Although the accuracy of the TCTF and TMac methods  raise with the increasing of sample rate, ST-MTRTC also outperforms TCTF and TMac. These results indicate that ST-MTRTC is the best method to recover internet traffic data.

 \begin{figure}[htbp]
\centering
 \begin{subfigure}[b]{0.45\linewidth}
   \centering
   \includegraphics[width=1\linewidth]{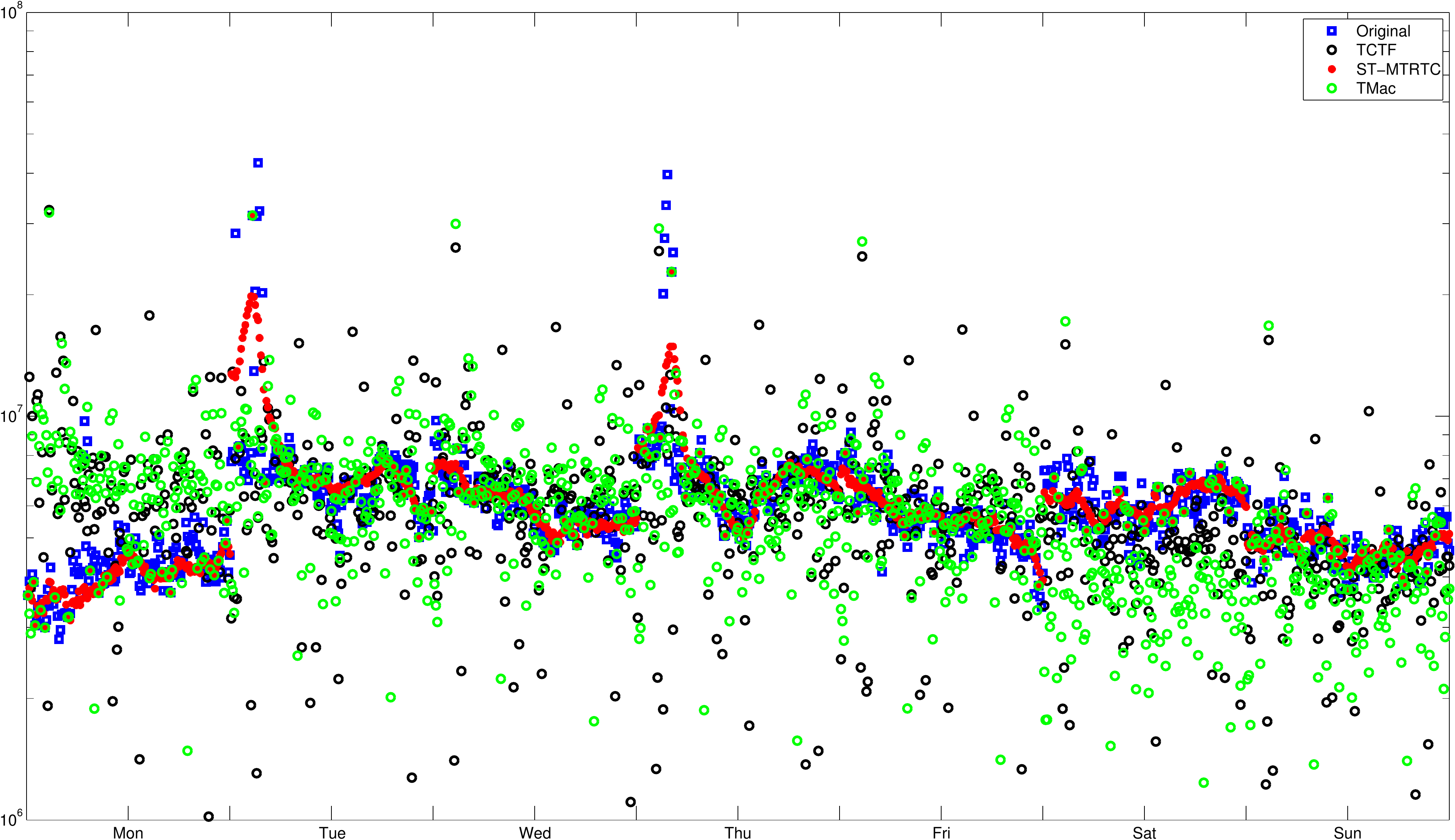}
 \caption{Sampling ratio $p=0.2$}
 \end{subfigure}
 \begin{subfigure}[b]{0.45\linewidth}
   \centering
   \includegraphics[width=1\linewidth]{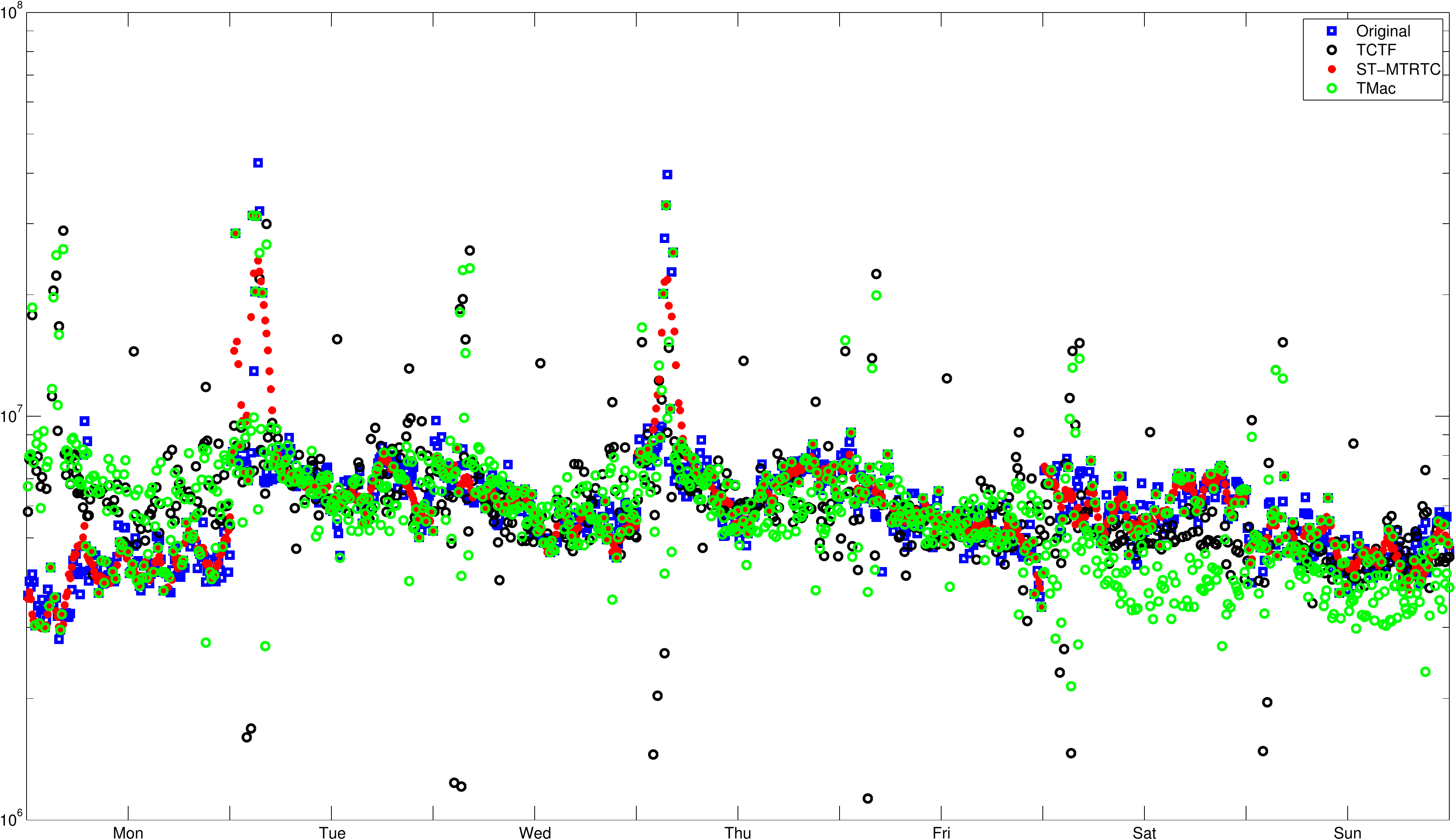}
 \caption{Sampling ratio $p=0.4$}
    \end{subfigure}

 \begin{subfigure}[b]{0.45\linewidth}
  \centering
  \includegraphics[width=1\linewidth]{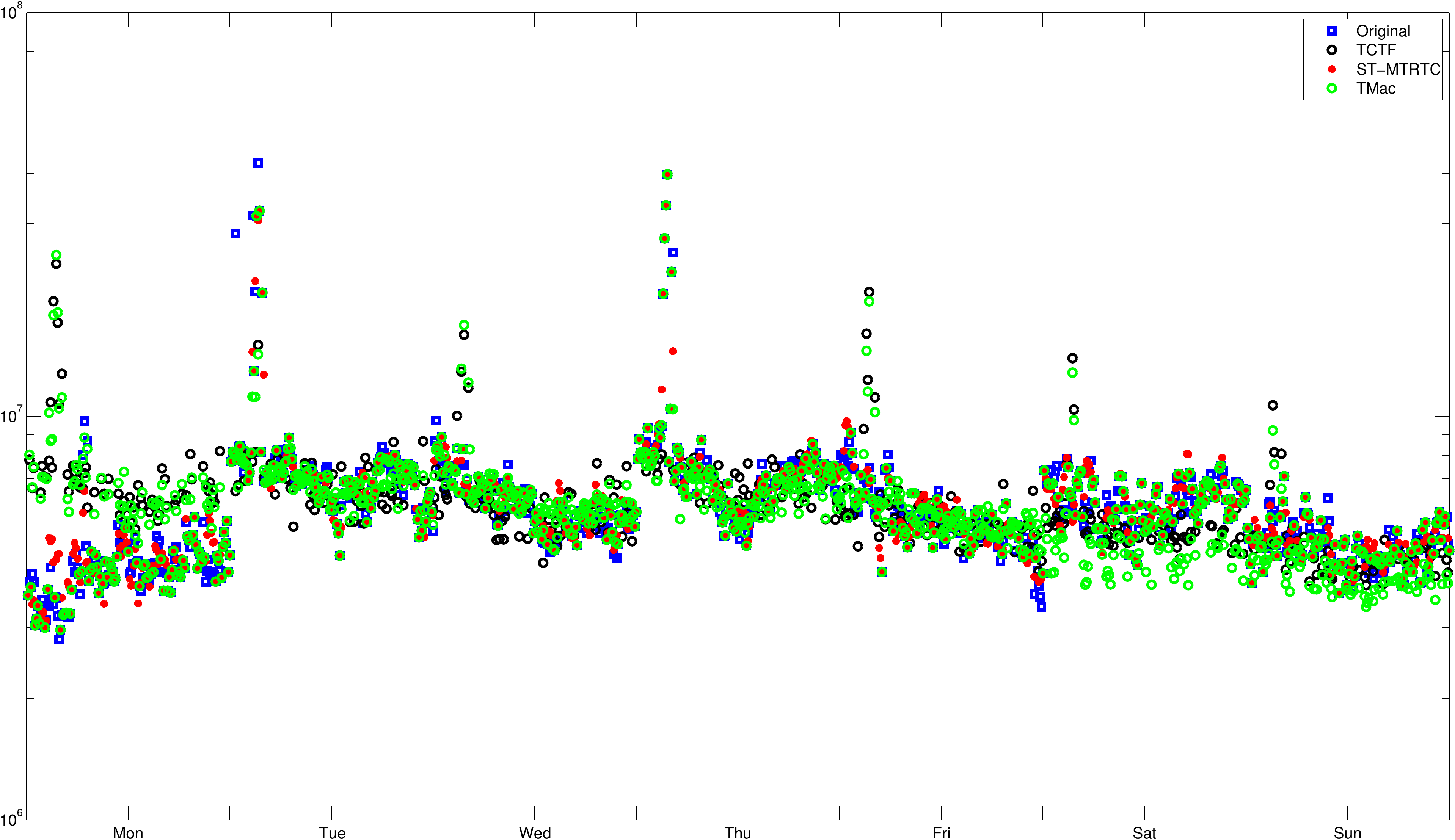}
  \caption{Sampling ratio $p=0.6$}
 \end{subfigure}
 \begin{subfigure}[b]{0.45\linewidth}
 \centering
 \includegraphics[width=1\linewidth]{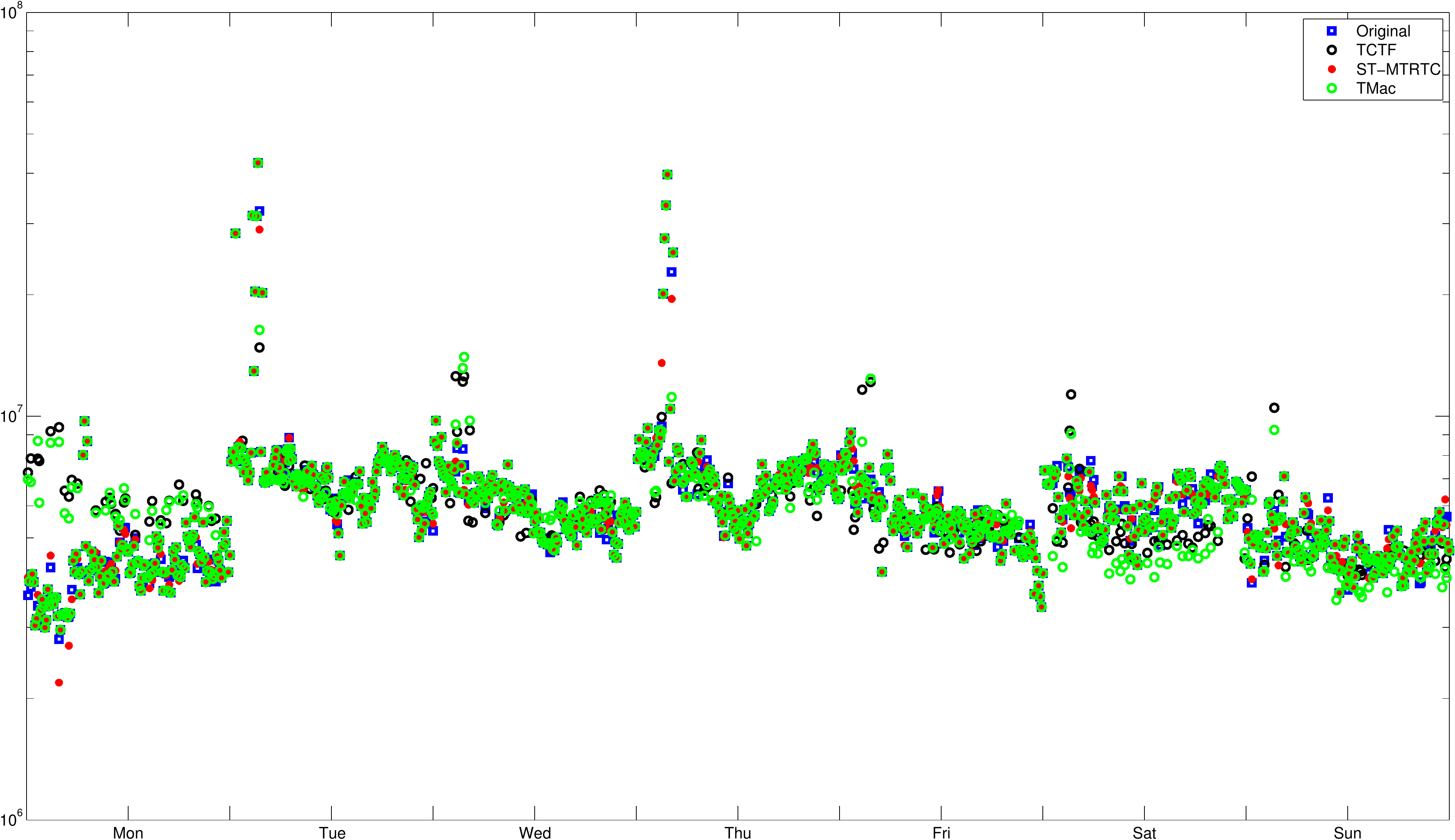}
 \caption{Sampling ratio $p=0.8$}
    \end{subfigure}
    \centering
    \caption{Recovery performance comparison of different sampling ratios}
    \label{fig:trafcom}
\end{figure}

\section{Conclusion}
In this paper, we extended tubal rank to multi-tubal rank and then established a relationship between multi-tubal rank and Tucker rank.  The tubal rank focuses on one mode of the tensor, while  multi-tubal rank considers all three modes of the tensor together. Based on multi-tubal rank, we established a new tensor completion model and applied a tensor factorization based method for solving the established problem. In addition, we applied spatio-temporal characteristics to the video inpainting and internet traffic simulation to modify the established model as a novel one. A modified tensor factorization based method was presented to solve such data completion problem, which got better performance without increasing the computational cost. Experimental results showed that the performance of our proposed methods were significantly better than existing methods in the literature.

%\bibliographystyle{IEEEtran}
%\bibliography{IEEEabrv,mylib}

\end{document}